\documentclass[11pt]{article}
\usepackage{amssymb, authblk}
\usepackage{amsmath, bbm}
\usepackage{fullpage}
\usepackage{amsthm, natbib, hyperref}
\usepackage{graphicx}
\usepackage{verbatim}
\usepackage{mathtools}

\usepackage{algorithm}% http://ctan.org/pkg/algorithms
\usepackage{algpseudocode}% http://ctan.org/pkg/algorithmicx
\usepackage[dvipsnames]{xcolor}
\usepackage{tikz}
\usetikzlibrary{arrows.meta}
\usetikzlibrary{decorations.pathreplacing}

\algnewcommand\algorithmicinput{\textbf{INPUT:}}
\algnewcommand\INPUT{\item[\algorithmicinput]}
\algnewcommand\algorithmicoutput{\textbf{OUTPUT:}}
\algnewcommand\OUTPUT{\item[\algorithmicoutput]}

\DeclareMathAlphabet{\mathpzc}{OT1}{pzc}{m}{it}

\usepackage{hyperref}[]
\hypersetup{
    colorlinks=true,
    linkcolor=blue,
    filecolor=magenta,      
    urlcolor=purple,
      citecolor=blue,
    }

\usepackage{cleveref}

\newtheorem{theorem}{Theorem}
\newtheorem{lemma}[theorem]{Lemma}

\newtheorem{definition}{Definition}
\newtheorem{remark}{Remark}
\newtheorem{assumption}{Assumption}

\allowdisplaybreaks

\DeclareMathOperator*{\argmax}{arg\,max}

\date{\vspace{-5ex}}

\title{Optimal nonparametric multivariate change point detection and localization}
\author[1]{Oscar Hernan Madrid Padilla}
\author[2]{Yi Yu}
\author[3]{Daren Wang}
\author[4]{Alessandro Rinaldo}

\affil[1]{\small Department of Statistics, University of California Los Angeles}
\affil[2]{\small Department of Statistics, University of Warwick}
\affil[3]{\small Department of Statistics, University of Chicago}
\affil[4]{\small Department of Statistics and Data Science, Carnegie Mellon University}

\begin{document}
\maketitle
\begin{abstract}
	We study the multivariate nonparametric change point detection problem, where the data are a sequence of independent $p$-dimensional random vectors whose distributions are piecewise-constant with Lipschitz densities changing at unknown times, called change points. We quantify the size of the distributional change at any change point with the supremum norm of the difference between the corresponding densities.  We are concerned with the localization task of estimating the positions of the change points.  In our analysis, we allow for  the model parameters to vary with the total number of time points, including the minimal spacing between consecutive change points and the magnitude of the smallest distributional change.  We provide information-theoretic lower bounds on both the localization rate and the minimal signal-to-noise ratio required to guarantee consistent localization.  We formulate a novel algorithm based on kernel density estimation that nearly achieves the minimax lower bound, save possibly for logarithm factors.  We have provided extensive numerical evidence to support our theoretical findings.
	
\end{abstract}

\section{Introduction}

We study the nonparametric multivariate change point detection problem, where we are given a sequence of independent random vectors $\{X(t)\}_{t = 1}^T \subset \mathbb{R}^p$ with unknown distributions $\{P_t \}_{t = 1}^T$ such that, for an unknown sequence of {\it change points} $\{ \eta_k\}_{k=1}^K \subset \{2, \ldots, T\}$ with $1 = \eta_0 < \eta_1 < \ldots < \eta_K \leq T < \eta_{K+1} = T+1$, we have 
	\begin{equation} \label{eqn:breaks}
		P_{t} \neq P_{t-1} \quad \mbox{if and only if} \quad t \in \{\eta_1, \ldots, \eta_K \}.
	\end{equation}
	Our goal is to accurately estimate the number of change points $K$ and their locations.

Change point localization problems of this form arise in a variety of application areas, including  biology, epidemiology, social sciences, climatology, technology diffusion, advertising, to name but a few.  Due to the high demand from real-life applications, change point  detection is a well-established topic in statistics with a rich literature. Some early efforts include seminal works by \cite{Wald1945}, \cite{Yao1988}, \cite{YaoAu1989},  \cite{YaoDavis1986}. More recently, the change point detection literature has been  brought back to the spotlight due to significant methodological and theoretical advances, including \cite{AueEtal2009}, \cite{killick2012optimal}, \cite{fryzlewicz2014wild}, \cite{FrickEtal2014}, \cite{cho2016change}, \cite{wang2016high}, \cite{wang2018optimal}, among many others, in different aspects of parametric change point detection problems.  See \cite{wang2018univariate} for a more comprehensive review.

Most of the exiting results in the change point localization literature rely on  parametric assumptions on the underlying distributions and on the nature of their changes. Despite the popularity and applicability of parametric change point detection methods, it is also important to develop more general and flexible change point localization procedures that are applicable over larger, possibly nonparametric, classes of distributions.  Several efforts in this direction have been recently made for univariate data. \cite{pein2017heterogeneous} proposed  a version of the SMUCE algorithm \citep{FrickEtal2014} that is more sensitive to simultaneous changes in mean and variance; \cite{zou2014nonparametric}  introduced a  nonparametric estimator  that can detect general distributions shifts; \cite{padilla2018sequential} considered  a nonparametric procedure for sequential change point detection;  \cite{fearnhead2018changepoint} focused on univariate mean change point detection constructing an estimator  that is robust to outliers; \cite{vanegas2019multiscale} proposed an estimator  for detecting changes in pre-specified quantiles of the generative model; and \cite{padilla2019optimal}  developed a nonparametric version of binary segmentation \citep[e.g.][]{ScottKnott1974} based on the Kolmogorov--Smirnov statistic.

In multivariate nonparametric settings, the literature on change point analysissi comparatively  limited.  \cite{arlot2012kernel} considered a penalized kernel least squares estimator, originally proposed by \cite{harchaoui2007retrospective}, for multivariate change point problems  and derive an oracle inequality. \cite{garreau2018consistent} obtained  an upper bound on the localization rate afforded by this method, which is further improved computationally in \cite{celisse2018new}.   \cite{matteson2014nonparametric} also proposed a methodology for multivariate nonparametric change point localization and show that that it can  consistently estimate the change points.  \cite{zhang2017pruning} provided a computationally-efficient algorithm, based on a pruning routine based on dynamic programming.   

In this paper we investigate the multivariate change point localization problem in fully nonparametric settings where the underlying distributions are only assumed to have piecewise and uniformly (in $T$, the total number of time points) Lipschitz continuous densities and the magnitudes of the distributional changes are measured by the supremum norm of the differences between the corresponding densities. We formally introduce our model next. 

\begin{assumption}[Model setting]
	\label{assump-model}
Let  $\{X(t)\}_{t = 1}^T \subset \mathbb{R}^p$ be a sequence of independent vectors  satisfying \eqref{eqn:breaks}. Assume that, for each $t =1 ,\ldots, T$, the distribution $P_t$ has a bounded Lebesgue  density function $f_t: \, \mathbb{R}^p \to \mathbb{R}$ such that 
	\begin{equation} \label{eqn:lip}
		\max_{t = 1, \ldots, T} \bigl|f_t(s_1) - f_t(s_2)\bigr| \leq C_{\mathrm{Lip}} \|s_1 - s_2\|, \quad \mbox{for all } s_1, s_2 \in \mathcal{X},
	\end{equation} 
	where  $\mathcal{X} \subset \mathbb{R}^p$ is the union of the supports of all the density functions $f_t$, $\|\cdot\|$ represents the $\ell_2$-norm, and $C_{\mathrm{Lip}} > 0$ is an absolute constant. 	
	We let 
	\[
		\Delta = \min_{k = 1, \ldots, K+1} \{\eta_k - \eta_{k-1}\} \leq T
	\]
	denote the minimal spacing between any two consecutive change points.  For each $k=1,\ldots,K$, we set
	\[
		\kappa_k = \sup_{z \in \mathbb{R}^p} \bigl|f_{\eta_{k}}(z) - f_{\eta_{k}-1}(z)\bigr| =  \|f_{\eta_k} - f_{\eta_{k}-1}\|_{\infty}
	\]
	as the size of the change at the $k$th change point. Finally, we let
	\begin{equation}\label{eq-as1-kappa}
		\kappa = \min_{k = 1, \ldots, K} \kappa_k >0,
	\end{equation}
	be the minimal such change.
\end{assumption}

The uniform Lipschitz condition \eqref{eqn:lip} is a rather mild  requirement on the smoothness of the underlying densities.  
The use of the supremum distance  is a natural choice in  nonparametric density estimation settings.  If we assume the domain $\mathcal{X}$ to be compact, then the supremum distance is stronger than the $L_1$~distance (total variation distance). %if in addition, we assume that $p = 1$, then the supremum distance is also stronger than the Kolmogorov--Smirnov distance.  
%Moreover, assuming the domain $\mathcal{X}$ to be compact, the Lipschitz condition \eqref{eqn:lip} implies that  all the density functions  are uniformly bounded.  

%The assumption \eqref{eqn:lip} could be weakened, by only requiring 
%We remark that there are alternatives to \eqref{eqn:lip}.  For instance, in \cite{kim2018uniform}, instead of assuming the Lipschitz condition, which implies the boundedness of density functions, the authors introduced a condition on the boundedness of how the probability volume decays with respect to the volume dimension.  This alternative would enable theoretical analysis in a non-Euclidean space.

The  model parameters $\Delta$ and $\kappa$  are allowed to change with the total number of time points $T$. This modeling choice allows us to consider change point models for which it becomes increasingly difficult to identify and estimate the change point locations accurately as we acquire more data.   For simplicity, we will not explicitly express the dependence of $\Delta$ and $\kappa$ on $T$ in our nation. %Throughout the paper, there are a number of absolute constants, for instance the Lipschitz constant $C_{\mathrm{Lip}}$, staying unchanged despite the diverging nature of $T$.
The dimension $p$ is instead treated as a fixed constant, as is customary in nonparametric literature, although our analysis could be extended to allow $p$ to grow with $T$ with a more careful tracking of the constants; see, e.g.~\cite{pmlr-v54-mcdonald17a}.
We will refer to any relationship among $\Delta$ and $\kappa$ that holds as $T$ tends to infinity as a {\it parameter scaling} of the model in \Cref{assump-model}.

The change point localization task can be formally stated as follows. We seek to construct change point estimators $1 < \hat{\eta}_1 < \ldots < \hat{\eta}_{\widehat{K}} \leq T$ of the true change points $\{\eta_k\}_{k = 1}^K$ such that,  with probability tending to $1$ as $T \rightarrow \infty$,
	\[
		\widehat{K} = K \quad \text{and} \quad \max_{k=1,\ldots,K} | \hat{\eta}_k - \eta_k| \leq \epsilon,
	\]
	where $\epsilon = \epsilon(T,\Delta,\kappa)$.  We say that the change point estimators  $\{\hat{\eta}_k\}_{k = 1}^{\hat{K}}$  are consistent if the above holds with 
	\begin{equation}\label{eq:consistency}
		\lim_{T \rightarrow \infty} \epsilon/\Delta = 0.	
	\end{equation}
	 We refer to $\epsilon$ as the localization error and to the sequence  $\{ \epsilon/\Delta \}$ as the localization rate.

\subsection{Summary  of the results}

The contributions of this paper are as follows.

\begin{itemize}
	\item We show that the difficulty of the localization task can be completely characterized in terms of the signal-to-noise ratio $\kappa^{p+2} \Delta$.  Specifically,  the space of the model parameters $(T, \Delta, \kappa)$ can be separated into an infeasible region, characterized by the scaling
		\begin{equation}\label{eq-intro-1}
			\kappa^{p+2}\Delta \lesssim 1
		\end{equation}
		and where no algorithm is guaranteed to produce consistent estimators of the change points (see \Cref{lem-snr-lb}), and a feasible region, in which
		\begin{equation}\label{eq-intro-2}
			\kappa^{p+2} \Delta \gtrsim \log^{1 + \xi}(T), \quad \text{for any } \xi > 0.
		\end{equation}
		Under the feasible scaling, we develop the MNP (multivariate nonparametric) change point estimator, given in \Cref{algorithm:WBS}, that is provably consistent.  The gap between \eqref{eq-intro-1} and \eqref{eq-intro-2} is a poly-logarithmic factor in $T$, which implies that our procedure is consistent under nearly all scalings for which this task is feasible. %We have not optimized constants involved in the specification of the impossibility and feasibility regions \eqref{eq-intro-1} and \eqref{eq-intro-2}, respectively. 
	
		%Even though this paper is in a completely nonparametric multivariate setting, 
		%The phase transition we have identified here should be compared to analogous findings in the recent literature on change point localization in nonparametric and high-dimensional settings: see, e.g., \cite{padilla2019optimal}, \cite{WangEtal2017}, \cite{wang2018univariate},  \cite{wang2018optimal}.  %matches the phase transition for localization found  with univariate nonparametric setting \citep{padilla2018sequential} and various parametric settings \citep[e.g.][]{wang2017optimal, wang2018optimal}.
	
%	\item We provide an information-theoretic lower bound on the localization error, showing that if $ \kappa^{p+2} \Delta  \gtrsim \zeta_T$, with any sequence $\{\zeta_T\}$ satisfying $\lim_{T \to \infty} \zeta_T = \infty$,  $\epsilon$ is larger than $\kappa^{-(p+2)}$, up to constants; see  \Cref{lemma-error-opt}.  that if $ \kappa^{p+2} \Delta  \gtrsim \zeta_T$, with any sequence $\{\zeta_T\}$ satisfying $\lim_{T \to \infty} \zeta_T = \infty$, then a lower bound on $\epsilon$ is $\kappa^{-(p+2)}$, which shows that our localization error upper bound $\log(T) \kappa^{-(p+2)}$ in \Cref{thm-wbs}, is nearly minimax optimal.  

	\item We show that the localization error achieved by the MNP procedure is of order $\log(T) \kappa^{-(p+2)}$  across the entire feasibility region given in \eqref{eq-intro-2}; see \Cref{thm-wbs}. We verify that this rate is nearly minimax optimal by deriving an information-theoretic lower bound on the localization error, showing that if $ \kappa^{p+2} \Delta  \gtrsim \zeta_T$, for any sequence $\{\zeta_T\}$ satisfying $\lim_{T \to \infty} \zeta_T = \infty$,  then the localization error  is larger than $\kappa^{-(p+2)}$, up to constants; see  \Cref{lemma-error-opt}. Interestingly, the dependence on the dimension $p$ is exponential, and matches the optimal dependence in the multivariate density estimation problems assuming Lipschitz-continuous densities.  We elaborate on this point further  in \Cref{sec-bandwidth}. The numerical experiments in \Cref{sec-exp} confirm the good performance of our algorithm.  

\item The MNP estimator is a computationally-efficient procedure for nonparametric change point localization in multivariate settings and can be considered a multivariate nonparametric extension of the binary segmentation methodology \citep{ScottKnott1974} and its, variant wild binary segmentation \citep{fryzlewicz2014wild}. The MNP estimator deploys a version of the CUSUM statistic \citep{Page1954} based on kernel density estimators.
We remark that some of our auxiliary results on consistency of kernel density estimators are obtained through non-trivial adaptation of existing techniques that allow for non-i.i.d. the data and may be of independent interest.  %The fixed sample results derived in \Cref{sec:appendx_a} are interesting \emph{per se} due to the non-i.i.d.~nature of our data.   

%	\item We propose a computationally-efficient algorithm, \Cref{algorithm:WBS}, which is a multivariate nonparametric version of binary segmentation \citep{ScottKnott1974} and its variant wild binary segmentation \citep{fryzlewicz2014wild}.  
\end{itemize}

%We would like to compare our paper with \cite{padilla2019optimal}, which worked on a univariate nonparametric change point detection problem, where the data are scalars and the distributional changes are measured in the Kolmogorov--Smirnov distance.  In this paper, we allow for $p > 1$ and adopt a supremum distance, which is more natural and popular than the multivariate versions of Kolmogorov--Smirnov statistics \citep[e.g.][]{justel1997multivariate,polonik1999} in the multivariate nonparametric literature.  This difference also implies that the phase transition phenomena in these two papers should not be directly compared with by setting $p = 1$.  In terms of theoretical difficulties, we require a careful analysis of multivariate kernel density estimators constructed by non-i.i.d.~data, while \cite{padilla2018sequential} considered empirical distribution estimators.  

\vskip 3mm
The rest of the paper is organized as follows. In \Cref{sec:methodology} we introduce the MNP procedure and in \Cref{sec:theory} we study its consistency and optimality. Simulation experiments demonstrating the effectiveness of the MNP algorithm and its competitive performance relative to existing procedures are reported in \Cref{sec-exp}.  The proofs and technical details are left in the Appendices.

\section{Methodology}
\label{sec:methodology}

%To motivate, we first consider the task of comparing two density functions  $f, g: \, \mathbb{R}^p \to \mathbb{R}$.  If $f(x) = g(x)$ for all $x \in \mathbb{R}^p$, then $f(x) - g(x) = 0$, $x \in \mathbb{R}^p$.  It is natural to seek  density estimators $\hat{f}(x)$ and $\hat{g}(x)$ and  to  assess if their differences are significantly away from 0.  It is, intuitively, reasonable to  construct CUSUM statistics combined with density estimators.  For the latter, we adopt kernel density estimators with a general kernel function, the conditions of which will be specified later.  Specifically, we have the following definition.

Our procedures for change point detection and localization is a nonparametric extension of the traditional CUSUM statistic and it relies on kernel density estimators.

\begin{definition}[Multivariate nonparametric CUSUM]\label{def-mul-non-cusum}
	Let $\{X(i)\}_{i=1}^T$ be a sample in $\mathbb{R}^p$. For any integer triplet $(s, t, e)$ satisfying $0 \leq s < t < e \leq T$ and any $x \in \mathbb{R}^p$,  the multivariate nonparametric CUSUM statistic is defined as the function
	\[
	x \in \mathbb{R}^p \mapsto	\widetilde{Y}^{s, e}_{t}(x) = \sqrt{\frac{(t-s)(e-t)}{e-s}} \left\{\hat{f}_{s+1, t, h}(x) - \hat{f}_{t+1, e, h}(x)\right\},
	\]
	where
	\begin{equation}\label{eq-f-kernel}
		\hat{f}_{s, e, h}(x) = \frac{h^{-p}}{e - s} \sum_{i = s + 1}^e \mathpzc{k} \left(\frac{x - X(i)}{h}\right)
	\end{equation}
	and $\mathpzc{k}(\cdot)$ is a kernel function \citep[see e.g.][]{parzen1962estimation}. In addition, define	
	\begin{equation}\label{eq-ytilde}
		\widetilde{Y}^{s, e}_t = \max_{i = 1, \ldots, T}\left|\widetilde{Y}^{s, e}_t (X(i))\right|.
	\end{equation}
	%where the kernel density estimators are constructed using the sample $\{X(j)\}_{j = 1}^T$ as in \eqref{eq-f-kernel}.
\end{definition}

\begin{remark}
The statistic $\widetilde{Y}^{s, e}_t$ can be seen as an estimator of 
	\[
		\sup_{z \in \mathbb{R}^p}\left|\widetilde{Y}^{s, e}_t(z)\right|.
	\]	
\end{remark}

%\begin{remark}
%Recall that in \Cref{assump-model}, we use the supremum norm of the difference of the underlying densities at the change points to quantify the minimal magnitude  $\kappa$ of the distributional change.  
%An alternative definition would be to instead define $\kappa$ as
%	\begin{equation}\label{eq-alternative-jump}
%		\min_{k = 1, \ldots, K+1} \sup_{s \in \mathbb{R}^p} \frac{1}{h^p}\left|\int_{\mathbb{R}^p} \mathpzc{k}\left(\frac{s - x}{h}\right) \bigl\{f_{\eta_k}(x) - f_{\eta_{k-1}}(x)\bigr\}\, dx\right|,
%	\end{equation}
%	where $h>0$ is a fixed bandwidth. As argued in \cite{kim2018uniform}, this definition avoids completely the bias introduced by the kernel smoothing and in fact will deliver dimension independent rates both for density estimation and change point localization. However, at the same time, it would effectively prevent the problem 
%	%This alternative definition avoids the bias introduced by the kernel smoothing and facilitates the theoretical arguments.  In the context of change point detection, however, the definition \eqref{eq-alternative-jump} is less intuitive and interpretable than \eqref{eq-as1-kappa}.
%\end{remark}

%With the definitions of the multivariate CUSUM statistics in \Cref{def-mul-non-cusum}, we state our approach in \Cref{algorithm:WBS}.
\Cref{algorithm:WBS} below presents a multivariate nonparametric version of the univariate nonparametric change point detection method proposed in \cite{padilla2019optimal}, wild binary segmentation \citep{fryzlewicz2014wild} and binary segmentation (BS) \citep[e.g.][]{ScottKnott1974}. The resulting procedure consists of repeated application of the BS algorithm over random time intervals and using the multivariate nonparametric CUSUM statistic in \Cref{def-mul-non-cusum}. The inputs of \Cref{algorithm:WBS} are a sequence $\{X(t)\}_{t=1,\ldots,T}$ of random vectors in $\mathbb{R}^p$, a tuning parameter $\tau>0$ and a bandwidth $h>0$.  Detailed theoretical requirements on the values of $\tau$ and $h$ are discussed below in \Cref{sec:theory}, and \Cref{sec_sim} offers  guidance on how to select them in practice.  In particular, the lengths of the sub-intervals are of order at least $h^{-p}$, where $h >0$ is the value of the bandwidth used to define the multivariate nonparametric CUSUM statistic. This is to ensure that each sub-interval will contain enough points to yield a  reliable density estimator.

%\textcolor{red}{Not sure I understand this comment: elaborate or expunge?}
Furthermore, in \Cref{algorithm:WBS} we scan through all time points between $s_m + h^{-p}$ and $e_m - h^{-p}$ in the interval $(s_m, e_m)$.  This is done for technical reasons, to avoid working with intervals  that have insufficient data, which would be the case when $e_m - s_m$ is small. %The latter would make the 
%If one works on $t \in \{s_m + 1, \ldots, e_m - 1\}$ instead, then an adaptive bandwidth is necessary to achieve optimality. 

Finally, the computational cost of the algorithm is of order $O(T^2 M \cdot \mathrm{kernel})$, where $M$ is the number of random intervals and ``kernel'' stands for the computational cost of calculating the value of the kernel function evaluated at one data point.  The dependence on the dimension $p$ is only through the evaluation of the kernel function.

\begin{algorithm}[htbp]
\begin{algorithmic}
	\INPUT Sample $\{X(t)\}_{t=s}^{e} \subset \mathbb{R}^p$, collection of intervals $\{ (\alpha_m,\beta_m)\}_{m=1}^M$, tuning parameter $\tau > 0$ and bandwidth $h > 0$.
	\For{$m = 1, \ldots, M$}  
		\State $(s_m, e_m) \leftarrow [s, e]\cap [\alpha_m, \beta_m]$
		\If{$e_m - s_m > 2h^{-p} + 1$}
			\State $b_{m} \leftarrow \argmax_{s_m + h^{-p} \leq t \leq e_m - h^{-p}}   \widetilde Y^{s_m, e_m}_{t}$
			\State $a_m \leftarrow \widetilde Y^{s_m, e_m}_{b_{m}}$
		\Else 
			\State $a_m \leftarrow -1$	
		\EndIf
	\EndFor
	\State $m^* \leftarrow \argmax_{m = 1, \ldots, M} a_{m}$
	\If{$a_{m^*} > \tau$}
		\State add $b_{m^*}$ to the set of estimated change points
		\State MNP$((s, b_{m*}),\{ (\alpha_m,\beta_m)\}_{m=1}^M, \tau)$
		\State MNP$((b_{m*}+1,e),\{ (\alpha_m,\beta_m)\}_{m=1}^M,\tau) $

	\EndIf  
	\OUTPUT The set of estimated change points.
\caption{Multivariate Nonparametric Change Point Detection. MNP $((s, e), \{ (\alpha_m,\beta_m)\}_{m=1}^M, \tau, h)$}
\label{algorithm:WBS}
\end{algorithmic}
\end{algorithm}

\section{Theory}
\label{sec:theory}

In this section  we  prove that the change point estimator MNP returned by \Cref{algorithm:WBS} is consistent based on the model described in \Cref{assump-model}, under the parameter scaling
	\[
		\kappa^{p+2}  \Delta \gtrsim \log^{1 + \xi}(T),
	\]
	for any $\xi > 0$; see \Cref{thm-wbs}. In addition, we show in \Cref{lem-snr-lb} that no consistent estimator exists if the above scaling condition is not satisfied, up to a poly-logarithmic factor. Finally, in \Cref{lemma-error-opt}, we demonstrate that the localization rate returned by the MNP procedure is nearly minimax rate-optimal.

\subsection{Optimal change point localization}

We begin by stating some assumptions on the kernel $\mathpzc{k}(\cdot)$ used to compute the kernel density estimators involved in the definition of the  multivariate nonparametric CUSUM statistic.

\begin{assumption}[The kernel function]\label{assump-kernel}
	 Let $\mathpzc{k}: \, \mathbb{R}^p \to \mathbb{R}$ be a kernel function with $\|\mathpzc{k}\|_{\infty}, \|\mathpzc{k}\|_2 < \infty$ such that,
	\begin{itemize}
	\item [(i)] the class of functions
		\[
			\mathcal{F}_{\mathpzc{k}, [l, \infty)} = \left\{\mathpzc{k}\left(\frac{x - \cdot}{h}\right): \, x \in \mathcal{X}, h \geq l\right\}
		\]
		from $\mathbb{R}^p$ to $\mathbb{R}$ is separable in $L_{\infty}(\mathbb{R}^p)$, and is a uniformly bounded VC-class with dimension $\nu$, i.e.~there exist positive numbers $A$ and $\nu$ such that, for every positive measure $Q$ on $\mathbb{R}^p$ and for every $u \in (0, \|\mathpzc{k}\|_{\infty})$, it holds that
		\[
			\mathcal{N}(\mathcal{F}_{\mathpzc{k}, [l, \infty)}, L_2(Q), u) \leq \left(\frac{A\|\mathpzc{k}\|_{\infty}}{u}\right)^{\nu};
		\]
	\item [(ii)] for a fixed $m > 0$, 
		\[
			\int_0^{\infty} t^{p-1} \sup_{\|x\| \geq t} |\mathpzc{k}(x)|^m\, dt < \infty.
		\]
	\item [(iii)] there exists a constant $C_{\mathpzc{k}} > 0$ such that
		\[
			\int_{\mathbb{R}^p}\mathpzc{k}(z) \|z\| \, dz \leq C_{\mathpzc{k}}.
		\]
	\end{itemize}

\end{assumption}

\Cref{assump-kernel} (i) and (ii) correspond to Assumptions~4 and 3 in \cite{kim2018uniform} and are fairly standard conditions used in the nonparametric density estimation literature, see  \cite{gine1999laws}, \citet{gine2001consistency}, \cite{sriperumbudur2012consistency}. They hold for most commonly used kernels, such as  uniform, Epanechnikov and Gaussian kernels. \Cref{assump-kernel} (iii) is a mild integrability assumption on the kernel.  %Note that by the  boundedness of the underlying density functions assumed in \Cref{assump-model}, it follows from Proposition~3 in \cite{kim2018uniform} that the volume dimension is $d_{\mathrm{vol}} = p$ under our assumptions.   These are standard conditions on the kernel function,  commonly-used kernel functions, including the uniform, Epanechnikov and Gaussian kernels, satisfy the conditions.

Next, we require the following signal-to-noise condition on the parameters of the model in order to guarantee that the MNP estimator is consistent. 

\begin{assumption}\label{assump-rates} 
Assume that for a given $\xi > 0$, there exists an absolute constant $C_{\mathrm{SNR}} > 0$ such that
	\begin{equation}\label{eq-as2-1}
		\kappa^{p+2}  \Delta > C_{\mathrm{SNR}} \log^{1 + \xi}(T).  
	\end{equation}
\end{assumption}

\Cref{assump-rates} can be relaxed by only requiring that $\kappa^{p+2}  \Delta > C_{\mathrm{SNR}} \log(T) e_T$, for any arbitrary sequence $\{e_T\}$ diverging to infinity, as $T$ goes unbounded. As we will see later, the above scaling is not only sufficient for consistent localization but almost necessary, aside for a  poly-logarithmic factor in $T$; see \Cref{lem-snr-lb}. This implies that the MNP estimator is consistent for nearly all parameter scalings for which the localization task is possible.

\begin{theorem}\label{thm-wbs}

Assume that the sequence $\{X_t \}_{t=1}^T$ satisfies the model described in \Cref{assump-model} and the signal-to-noise ratio condition \Cref{assump-rates}.  Let $\mathpzc{k}(\cdot)$ be a kernel function satisfying \Cref{assump-kernel}. 
Then, there exist positive universal constants $C_R$, $c_{\tau, 1}$, $c_{\tau, 2}$ and $c_h$, such that if \Cref{algorithm:WBS} is applied to the sequence $\{X_t \}_{t=1}^T$ using any collection  $\{(\alpha_r, \beta_r)\}_{r = 1}^R \subset \{1, \ldots, T\}$ of random time intervals with endpoints drawn independently and uniformly from $\{1, \ldots, T\}$ with 
		\begin{equation}\label{eq:CR}
			\max_{r = 1, \ldots, R}(\beta_r - \alpha_r) \leq C_R\Delta \quad \text{almost surely,}
		\end{equation}
tuning parameter $\tau$ satisfying
		\begin{equation}\label{eq-thm4-tau}
			c_{\tau, 1}\max\left\{h^{-p/2}\log^{1/2}(T), \, h \Delta^{1/2} \right\} \leq \tau \leq c_{\tau, 2}  \kappa \Delta^{1/2}
		\end{equation}
	and bandwidth $h$ given by
		\begin{equation}\label{eq-thm-h-cond}
			h = c_{h} \kappa, %\quad \text{where} %\quad  c_{\tau, 1} \max\{1, \, c_h^{-p/2}\} < c_{\tau, 2}.
		\end{equation}
then the resulting change point estimator $\{ \hat{\eta}_k\}_{k=1}^{\hat{K}}$  satisfies
		\begin{align}
		& \mathbb{P}\left\{\widehat{K} = K \quad \mbox{and} \quad \epsilon_k = |\hat{\eta}_k - \eta_k| \leq C_{\epsilon}\kappa^{-2}_k \kappa^{-p} \log(T), \, \forall k = 1, \ldots, K\right\} \nonumber\\
		\geq & 1 -  3T^{-c} - \exp\left\{\log\left(\frac{T}{\Delta}\right) - \frac{R\Delta}{4C_RT}\right\}, \label{eq-thm4-result}
		\end{align}
		for universal positive constants $C_\epsilon$ and $c$.
\end{theorem}

The constants in \Cref{thm-wbs} are well-defined provided that the constant $C_{\mathrm{SNR}}$ in the signal-to-noise ratio \Cref{assump-rates} is sufficiently larger. Their dependence can be tracked in the proof of the \Cref{thm-wbs}, given in \Cref{sec:proof.thm}. In particular, it must hold that $c_{\tau, 1} \max\{1, \, c_h^{-p/2}\} < c_{\tau, 2}$.

It is worth emphasizing that  we provide individual localization errors $\epsilon_k$, one for each true change point, in order  to avoid false positives in the iterative search of change points in \Cref{algorithm:WBS}.  Using \eqref{eq-thm4-result} and setting
	\[
		\epsilon = \max_{k = 1, \ldots, K} \epsilon_k,
	\] 
	our result further yields the general localization consistency guarantee defined in \eqref{eq:consistency} since, as $T \to \infty$,
	\begin{align*}
		 \frac{\epsilon}{\Delta} \leq C_{\epsilon}\frac{\log(T)}{\Delta \kappa^{p+2}} \leq \frac{C_{\epsilon}}{C_{\mathrm{SNR}}} \frac{\log(T)}{\log^{1+\xi}(T)} \to 0,
	\end{align*}
	where the second inequality follows from the definition of $\kappa$ in \eqref{eq-as1-kappa}, and the third follows from \Cref{assump-rates}.  %In addition, if  $K = O(1)$, then
%	the consistency still holds when $\xi = 0$. \textcolor{red}{Why? I don't think this is true in this case}
	%if one is only interested in $\epsilon/T$ and $\Delta = O(T)$, then the consistency still holds when $\xi = 0$.

The tuning parameter $\tau$ plays the role of a threshold for detecting change points in \Cref{algorithm:WBS}.  In particular, for the time points with maximal CUSUM statistics, if their CUSUM statistic values exceed $\tau$, then they are included in the change point estimators.  This means that, with large probability, the upper bound in \eqref{eq-thm4-tau} ought be smaller than the smallest population CUSUM statistics at the true change points, and the lower bound in \eqref{eq-thm4-tau} should be larger than the largest sample CUSUM statistics when there are no change points.  In detail, the upper bound is determined in \Cref{lem-2}, and the lower bound comes from Lemmas~\ref{lem-1} and \ref{lem-4}.  Lemma \ref{lem-1} is dedicated to the variance of the kernel density estimators at the observations, whereas  Lemma \ref{lem-4} focuses on the deviance between the sample and population maxima.  Lastly, the set of values for $\tau$ is not empty, by the inequalities
	\[
		c_{\tau, 1}h^{-p/2} \log^{1/2}(T) \leq c_{\tau, 1}c_h^{-p/2} \kappa^{-p/2} \log^{1/2}(T) < c_{\tau, 2}\kappa\Delta^{1/2}
	\]
	and
	\[
		c_{\tau, 1} \max\{1, \, c_h^{-p/2}\} < c_{\tau, 2}. 
	\]

The probability lower bound in \eqref{eq-thm4-result}  controls the events $\mathcal{A}_1(\gamma_{\mathcal{A}}, h)$, $\mathcal{A}_2(\gamma_{\mathcal{A}}, h)$, $\mathcal{B}(\gamma_{\mathcal{B}})$ and $\mathcal{M}$ defined and studied in Lemmas~\ref{lem-1}, \ref{lem-4} and \ref{lem-event-M} respectively, with 
	\[
		\gamma_{\mathcal{A}} = C_{\gamma_{\mathcal{A}}} h^{-p/2}\log^{1/2}(T) \quad \mbox{and} \quad \gamma_{\mathcal{B}} = C_{\gamma_{\mathcal{B}}}  h \Delta^{1/2},
	\]
	where $ C_{\gamma_{\mathcal{A}}},  C_{\gamma_{\mathcal{B}}} > 0$ are absolute constants.  
The lower bound on the probability in \eqref{eq-thm4-result} tends to 1, as $T$ grows unbounded, provided that the number $R$ of random intervals $(\alpha_r, \beta_r)$ is such that
	\[
		R \gtrsim \frac{T}{\Delta}\log\left(\frac{T}{\Delta}\right).
	\]

The assumption \eqref{eq:CR} is imposed to guarantee that each of the random intervals used in the MNP procedure contains a bounded number of change points. Thus, if $K = O(1)$,  this assumption can be discarded.  More generally, it is  possible to drop this assumption even when $\Delta = o(T)$, in which case the MNP estimator would still yield consistent localization, albeit with a localization error inflated by a polynomial factor in $T/\Delta$, under a stronger signal-to-noise ratio condition. Assumptions of this nature are commonly used in the analysis of the WBS procedure. For a discussion on the necessity of assumption $\eqref{eq:CR}$ in order to derive optimal rates, see \cite{padilla2019optimal}.

%if In fact, it is automatically upper bounded by $T/\Delta$.   If instead $\Delta = o(T)$, and we  abandon the  assumption that  $C_R$ is an absolute constant, then both \eqref{eq-as2-1} and the localization error in \eqref{eq-thm4-result} are to be inflated by polynomial factors in $T/\Delta$.  \textcolor{red}{More discussions on $C_R$ can be found in \cite{padilla2019optimal}.}

\begin{remark}[When $\kappa = 0$]
	\Cref{thm-wbs} builds upon the assumption that $\kappa > 0$, which implies that there exists at least one change point.  In fact, an immediate consequence of Step 1 in the proof of \Cref{thm-wbs} is the consistency for the simpler task of merely deciding if  there are  change points or not.  To be specific, if there are no true change points, then with the bandwidth and tuning parameter satisfying
		\[
			h > (\log(T)/T)^{1/p} \quad \mbox{and} \quad \tau \geq  c_{\tau, 1}\max\left\{h^{-p/2}\log^{1/2}(T), \, h T^{1/2} \right\},
		\]
		it holds that
		\[
			\mathbb{P}\{\widehat{K} = 0\} \to 1,
		\]
		as $T$ goes unbounded.  %\textcolor{red}{Wouldn't we also want to have that $\hat{K}>0$ if there exist change points, with probability tending to $1$?}%Having said this, we do not claim we show the optimality of change point testing, since testing is beyond the scope of this paper.
\end{remark}

\subsection{Change point localization versus  density estimation}	
\label{sec-bandwidth}
We now discuss how the change point localization problem relates to the classical task of optimal density estimation.  For simplicity, assume equally-spaced change points, so that the data consist of $K$ independent samples of size $\Delta$ from each of the underlying distributions. 

If we knew the locations of the change points -- or, equivalently, the number of change points --  then we could compute $K$ kernel density estimators, one for each sample.  Recalling that we assume the underlying densities to be Lipschitz and using well-known results about minimax density estimation, choosing the bandwidth to be of order 
	\[
		h_1   \asymp \left(\frac{\log(\Delta)}{\Delta}\right)^{1/(p+2)}
	\]
	would yield $K$ kernel density estimators that are minimax rate-optimal in the $L_\infty$-norm for each of the underlying densities. In contrast, the choice of the  bandwidth for the change point detection task is 
	\[
		h_{\mathrm{opt}} \asymp \kappa,
	\]
	as given in \eqref{eq-thm-h-cond}. In light of the minimax results established in the next section, such a choice of $h_{\mathrm{opt}}$ further guarantees that the localization rate afforded by the MNP algorithm is almost minimax rate-optimal. 

In virtue of \Cref{assump-rates} and the boundedness assumption on the densities, it holds that
	\[
		h_1 \lesssim h_{\mathrm{opt}}.
	\]
	The choice of bandwidth for optimal change point localization in the present problem is no smaller than the  choice for optimal density estimation. In particular, the two bandwidth coincides, i.e.~$h_1 \asymp h_{\mathrm{opt}}$, when the signal-to-noise ratio is smallest, i.e.~when \Cref{assump-rates} is an equality. As we will see below in  \Cref{lem-snr-lb}, change point localization is not possible when the the signal-to-noise ratio \Cref{assump-rates} fails, up to a slack factor that is poly-logarithmic in $T$. As a result, $h_1$  and $h_{\mathrm{opt}} \log^\xi(\Delta)$ are of the same order (up to a poly-logarithmic term in $T$) only under (nearly) the worst possible condition for localization. On the other hand, if $\kappa$ is vanishing in $T$ at a rate slower than $\left(\log(\Delta)/\Delta\right)^{1/(p+2)}$ (while still fulfilling \Cref{assump-rates}), then change point localization can be solved optimally using kernel density estimators that are suboptimal for density estimation, since they are based on bandwidths that are larger than the ones needed for optimality. Thus we conclude that the optimal sample complexity for the localization problem is strictly better than the optimal sample complexity needed for estimating all the underlying densities, unless  the difficulty of the change localization problem is maximal, in which case they coincide. At the opposite end of the spectrum, if $\kappa$ is bounded away from $0$, then the optimal change point localization can still be achieved  using {\it biased} kernel density estimators with bandwidths bounded away from zero. 
	%, even though it and $h_1 \asymp h_{\mathrm{opt}}$ holds in some settings including $\Delta \asymp T$.
	
%As for the density estimation aspect, increasing the bandwidth increases the bias of density estimation.  Recall our core task is to detect change points, therefore as long as the jump size is still larger than the sum of the density estimation bias and the fluctuation, the changes are still detectable.  
More generally, and quite interestingly,  our analysis reveals that there is a rather simple and intuitive way of describing how the difficulty of density estimation problem relates to the difficulty of consistent change point localization, at least in our problem. 
Indeed, it follows from the proof of \Cref{thm-wbs} (see also \eqref{eq-thm4-tau} in the statement of \Cref{thm-wbs}) that, in order  for MNP to return a consistent -- and, as we will see shortly, nearly minimax optimal --  estimator of the change point, the following should hold:
	\begin{equation}\label{eq:boundLinfty}
		\kappa \sqrt{\Delta} \gtrsim \gamma_{\mathcal{A}} + \gamma_{\mathcal{B}} \asymp  h^{-p/2} \log^{1/2}(T) + h \sqrt{\Delta}. % 
	\end{equation}
	Assuming for simplicity $\log(\Delta) \asymp \log(T)$, the right hand side of the previous expression divided by $\sqrt{\Delta}$ precisely corresponds to the sum of the magnitudes of the bias and of the random fluctuation for the kernel density estimator over each sub-interval, both measured in the $L_\infty$-norm. From this we immediately see that the MNP procedure will estimate the change points optimally provided that $\kappa$, the smallest magnitude of the distributional change at the change point, is larger than the $L_\infty$ error in estimating the underlying densities  via  kernel density estimation, {\it assuming full knowledge of the change point locations.} Though simple, we believe that this characterization is non-trivial and illustrates nicely the differences between the the task of density estimation of that of change point localization.

We conclude this section by providing some rationale as to why the optimal choice of $h$ for the purpose of change point localization happens to be $\kappa$, which in light of the inequality \eqref{eq:boundLinfty}, is the largest value $h$ is allowed to take in order for MNP to be consistent. We offer three different perspectives.
\begin{itemize}
	\item (Localization error). It can be seen in \Cref{lem-7} or in inequality \eqref{eq:coro wbsrp 1d re1-1} in the proof of \Cref{thm-wbs} that the localization error is such that 
	\[
		\epsilon_k \lesssim \frac{\gamma^2_{\mathcal{A}}}{\kappa_k^2} = \frac{\log(T)}{\kappa_k^2 h^p}, \quad k \in \{1, \ldots, K\}.
	\]
	Therefore, the larger the bandwidth $h$ is, the smaller the localization error.
	\item (Signal-to-noise ratio).  Since we require $\gamma_{\mathcal{A}} \lesssim \kappa\sqrt{\Delta}$, it needs to hold that 
		\[
			\kappa^2 h^{p} \Delta \gtrsim \log(T);
		\]	
		since in \eqref{eq-lem8-1111} in the proof of \Cref{lem-4} we require 
		\[
			 \kappa_k \sqrt{C_{\epsilon} \log(T) V_p^2 \kappa_k^{-2} \kappa^{-p}} \leq \gamma_{\mathcal{B}},
		\]
		where $V_p = \pi^{p/2} (\Gamma(p/2 + 1))^{-1}$ is the volume of a unit ball in $\mathbb{R}^p$, it needs to hold that
		\[
			\kappa^p h^2 \Delta \gtrsim \log(T).
		\]
		Therefore, the larger the bandwidth $h$ is, the smaller the signal-to-noise ratio needs to be.
	%\item (The design of \Cref{algorithm:WBS}). \textcolor{red}{I would expunge this} Since the binary segmentation search in the interval $(s, e)$ goes through all points between $s + h^{-p}$ and $e - h^{-p}$.  The design is meant to prompt the optimality.  It needs to hold that 
	%\[
%		2h^{-p} < \Delta.%
%	\] 
\end{itemize}

\subsection{Minimax lower bounds}
	
For the model given in Assumption~\ref{assump-model}, we will describe low signal-to-noise ratio parameter scalings for which consistent localization is not feasible. These scalings are complementary  to the ones in \Cref{assump-rates}, which, by \Cref{thm-wbs}, are sufficient for consistent localization.

\begin{lemma}\label{lem-snr-lb}
Let $\{X(t)\}_{t = 1}^{T}$ be a sequence of random vectors satisfying Assumption~\ref{assump-model} with one and only one change point and let  $P^T_{\kappa,  \Delta}$ denote the corresponding joint distribution. Then, there exist universal positive constants $C_1$, $C_2$ and $c < \log(2)$ such that, for all $T$ large enough,
\[
	\inf_{\hat{\eta}} \sup_{P \in \mathcal{Q}} \mathbb{E}_P\bigl(\bigl|\hat{\eta} - \eta(P)\bigr|\bigr) \geq \Delta/4,
	\]
	where 
	\[
	\mathcal{Q} = \mathcal{Q} (C_1,C_2,c) = \left\{P^T_{\kappa, \Delta}: \, \Delta < T/2,\,\kappa < C_1, \, \kappa^{p+2}  \Delta \leq c, \, C_{\mathrm{Lip}}\leq  C_2\right\},
	\]
	the quantity $\eta(P)$ denotes the true change point location of $P \in \mathcal{Q}$ and the infimum is over all possible estimators of the change point location.
\end{lemma}

The above result offers an information theoretic lower bound on the minimal signal-to-noise ratio required for localization consistency. It implies that \Cref{assump-rates} used by the MNP procedure, is, save for a poly-logarthmic term in $T$, the weakest possible scaling condition  on the model parameters any algorithm can afford. Thus, \Cref{lem-snr-lb} and \Cref{thm-wbs} together reveal a phase transition  over the parameter scalings, separating the impossibility regime in which no algorithm is consistent from the one in which MNP accurately estimates the change point locations. 
%Lemmas~\ref{lem-snr-lb} and \ref{lemma-error-opt} (see below) provide two minimax lower bounds on the detection boundary and the localization error rate, respectively.  These show that \Cref{thm-wbs} is nearly minimax optimal.

Our next result shows that the localization rate  achieved by \Cref{algorithm:WBS}  is indeed almost minimax optimal, aside possibly for a poly-logarithmic factor, over all scalings for which consistent localization is possible.

\begin{lemma}\label{lemma-error-opt}
	Let $\{X(t)\}_{t = 1}^{T}$ be a sequence of random vectors satisfying Assumption~\ref{assump-model} with one and only one change point and let  $P^T_{\kappa,  \Delta}$ denote the corresponding joint distribution. Then, there exist universal positive constants $C_1$ and $C_2$ such that, for any sequence $\{ \zeta_T \}$ satisfying $\lim_{T \rightarrow \infty} \zeta_T = \infty $, 
\[
	\inf_{\hat{\eta}} \sup_{P \in \mathcal{Q}} \mathbb{E}_P\bigl(\bigl|\hat{\eta} - \eta(P)\bigr|\bigr) \geq \max \left\{ 1, \frac{1}{4} \Big\lceil\frac{1}{V_p^2\kappa^{p+2} } \Big\rceil e^{-2} \right\},
	\]
where $V_p = \pi^{p/2} (\Gamma(p/2 + 1))^{-1}$ is the volume of a unit ball in $\mathbb{R}^p$, 
	\[
	\mathcal{Q} = \mathcal{Q} (C_1,C_2,\{ \zeta_T \}) = \left\{P^T_{\kappa, \Delta}: \, \Delta < T/2,\,\kappa < C_1, \, \kappa^{p+2}  V_p^2 \Delta \geq \zeta_T, \, C_{\mathrm{Lip}}\leq  C_2\right\},
	\]
	the quantity $\eta(P)$ denotes the true change point location of $P \in \mathcal{Q}$ and the infimum is over all possible estimators of the change point location.
\end{lemma}

The previous result demonstrates that that the performance of the MNP procedure is essentially non-improvable, except possibly for as poly-logarithmic term in $T$. In particular,  adapting choosing the bandwidth in a way that depend on the lengths of the working intervals 
is not going to bring significant improvements over a fixed choice.

%\textcolor{red}{I would remove this comment; discuss?} 

%\section{Conclusion}
\section{Experiments}\label{sec-exp}

In this section we describe several computational experiments  illustrating the effectiveness of the MNP procedure for estimating change point locations across a variety of scenarios.
We organize  our  experiments into two subsections, one consisting of examples with simulated data and the other based on a real data example.  Code  implementing our method  can be found in \url{https://github.com/hernanmp/MNWBS}.

\subsection{Simulations}
\label{sec_sim}
We start our experiments section by assessing the performance of Algorithm \ref{algorithm:WBS}  in a wide  range of situations. We compare our MNP procedure against the energy based method (EMNCP) from \cite{matteson2014nonparametric},  the sparsified binary segmentation  (SBS) method from \cite{cho2015multiple},   the double CUSUM  binat segmentation estimator  (DCBS) from  \cite{cho2016change},  and 
the kernel change point detection procedure  (KCPA)\citep{celisse2018new, arlot2019kernel}.

As a measure of performance we use the absolute error $\vert \widehat{K} - K\vert$,  averaged over 100 Monte  Carlo simulations, where  $\widehat{K}$ is the estimated number of change points returned by the estimators. In addition, we use the  one-sided Hausdorff distance 
\[
d(\widehat{\mathcal{C}}| \mathcal{C}) \,=  \,  \underset{\eta \in \mathcal{C} }{\max }\,\underset{ x\in \widehat{\mathcal{C}}   }{\min}\,\vert x - \eta\vert,
\] 
where $\mathcal{C}  = \{ \eta_1, \ldots, \eta_K  \}$ is the set of true change points and $\widehat{\mathcal{C}}$ is the set of estimated change points.  We report the medians of both $d(\widehat{\mathcal{C}}| \mathcal{C})$ and $d( \mathcal{C}|\widehat{\mathcal{C}}) $ over 100 Monte Carlo simulations. We use the convention that when $\widehat{\mathcal{C}} = \emptyset$, we define $d(\widehat{\mathcal{C}}| \mathcal{C})  = \infty$ and  $d(\mathcal{C}| \widehat{\mathcal{C}})  = - \infty$.

With regards to the implementation of the EMNCP method,  we use the R \citep{R} package \texttt{ecp} \citep{ecp}. The calculation of the change points is done via the function \texttt{e.divisive()}. Furthermore,  the methods SBS and DCBS methods  we use the R \citep{R} package \texttt{hdbinseg}, whereas for KCPA we use the  R \citep{R} package  \texttt{KernSeg}. 

As for the MNP  method  described in Algorithm \ref{algorithm:WBS}, we use the  Gaussian kernel and set $M = 50$. We also  set   $h = 5\times (30 \log (T)/T)^{1/p}$, a choice that is guided by \eqref{eq-thm-h-cond}. Specifically,  the intuition is that  we need  to have $(\log (T)/\Delta)^{1/p} < h$, hence if there are at most  30  change points, then our choice of $h$ is reasonable.

With fixed  $h$, we then run Algorithm \ref{algorithm:WBS} with different  choices of the tunning parameter $\tau$.  This produces  a sequence of nested sets
\[
S_0  =\emptyset \subset S_1  \subset \ldots \subset  S_m,
\]
corresponding  to different  values of $\tau$. We then  borrow some inspiration from the selection procedure  in \cite{padilla2019optimal}. Specifically,  we start from $S_i$, with $i = m$, and for every  $\eta \in S_i \backslash S_{i-1} $  we   decide  whether $\eta$ is a change point or not. If at least  one  element   $\eta \in S_i \backslash S_{i-1} $ is declared  as a change point, then we stop  and set $\widehat{\mathcal{C}} = S_i$ as the set of estimated change points. Otherwise, we set $i = m-1$ and repeat the same procedure. We continue iteratively until the procedure  stops, or  $i=0$ in which case  $\widehat{\mathcal{C}}  = \emptyset$. The only  remaining ingredient is  how to  decide if $\eta \in S_i \backslash S_{i-1}$  is a change point or not. To that end,  we let $\eta_{(1)}, \eta_{(2)} \in S_{i-1}$, such that
\[
\eta \in [\eta_{(1)}, \eta_{(2)}]\,\,\,\,\text{and}\,\,\,\  (\eta_{(1)}, \eta_{(2)}) \cap  S_{i-1} = \emptyset.
\]
If $\eta <  \eta^{\prime}$ ($\eta >  \eta^{\prime}$) for all $\eta^{\prime} \in S_{i-1}$, then we set $\eta_{(1)} =1$  ($\eta_{(2)} = T$). Then, for  $v_1, \ldots, v_N \in \mathbb{ R}^p$  with  $\|v_l\|=1$ for  $l = 1,\ldots,N$,  we  calculate the  Kolmogorov--Smirnov (KS) statistic \citep[for instance, see][]{padilla2019optimal}. 
$$a_l =   \mathrm{KS}( \{v_l^{\top} X(t) \}_{ \eta_{(1)} +1 }^{\eta},\{v_l^{\top} X(t) \}_{ \eta +1 }^{\eta_{(2)}  }  ),$$
and the corresponding  $p$-value $\exp(-2 a_l^2)$. We   then  declare $\eta$ as  change point if at least  one adjusted $p$-value, using the false discovery rate control \citep{benjamini1995controlling}, is less than or equal to $0.0005$. This  choice is due to the  fact that  we do multiple tests  for different values of  $\tau$ and  their  corresponding estimated change points. The number of tests is in principle random, hence we  choose the value $0.0005$  since  $(1-0.0005)^{20}  \approx 0.99$,  and so it avoids false positives. Also, in our experiments, we set  $N = 200$.

%For each of these statistics  we then calculate   a p-value as  in $\exp(-2 a_l^2)$ and then  declare $\eta$ as  change point if at least  one statistic, after  controlling 

% for  testing whether or not    $\{v_l^T X(t) \}_{ \eta_{(1)} +1 }^{\eta}$  and   $\{v_l^T X(t) \}_{ \eta +1 }^{\eta_{(2)} }$ are draws  from the same distribution. Then we declare $\eta$ as a change point if 
 %$$\max_{ 1\leq l \leq N } a_l    >  r,$$
  %for some  $r$. In our experiments, we set  $N = 200$ and $r= 2$. Recall that for the two sample  KS  test the  value  $\sqrt{- (\log \alpha)/2}$  corresponds to the significance level $\alpha$. In particular,  a significance level of $\exp(-8) \approx 0.00033$, has a  cut off value of $2$.
% we 
% $S_{i-1} =  \emptyset$, we take  $\eta_1 =1$ and $\eta_2 = T$. Then, for  $v_1, \ldots, v_N \in \mathbb{ R}^p$  with  $\|v_l\|=1$ for  $l = 1,\ldots,N$
% we 

To evaluate  the quality  of the competing estimators, we construct  several change point models. In each case, we let $K=2$ and split the interval $[0,T]$ into  $3$ evenly-sized  intervals   denoted by $A_1, A_2$ and $A_{3}$. Furthermore, we   consider  $T \in \{150,300\}$ and $p \in \{10,20\}$.

\begin{figure}[t!]
	\begin{center}
		\includegraphics[width=3in,height=2in]{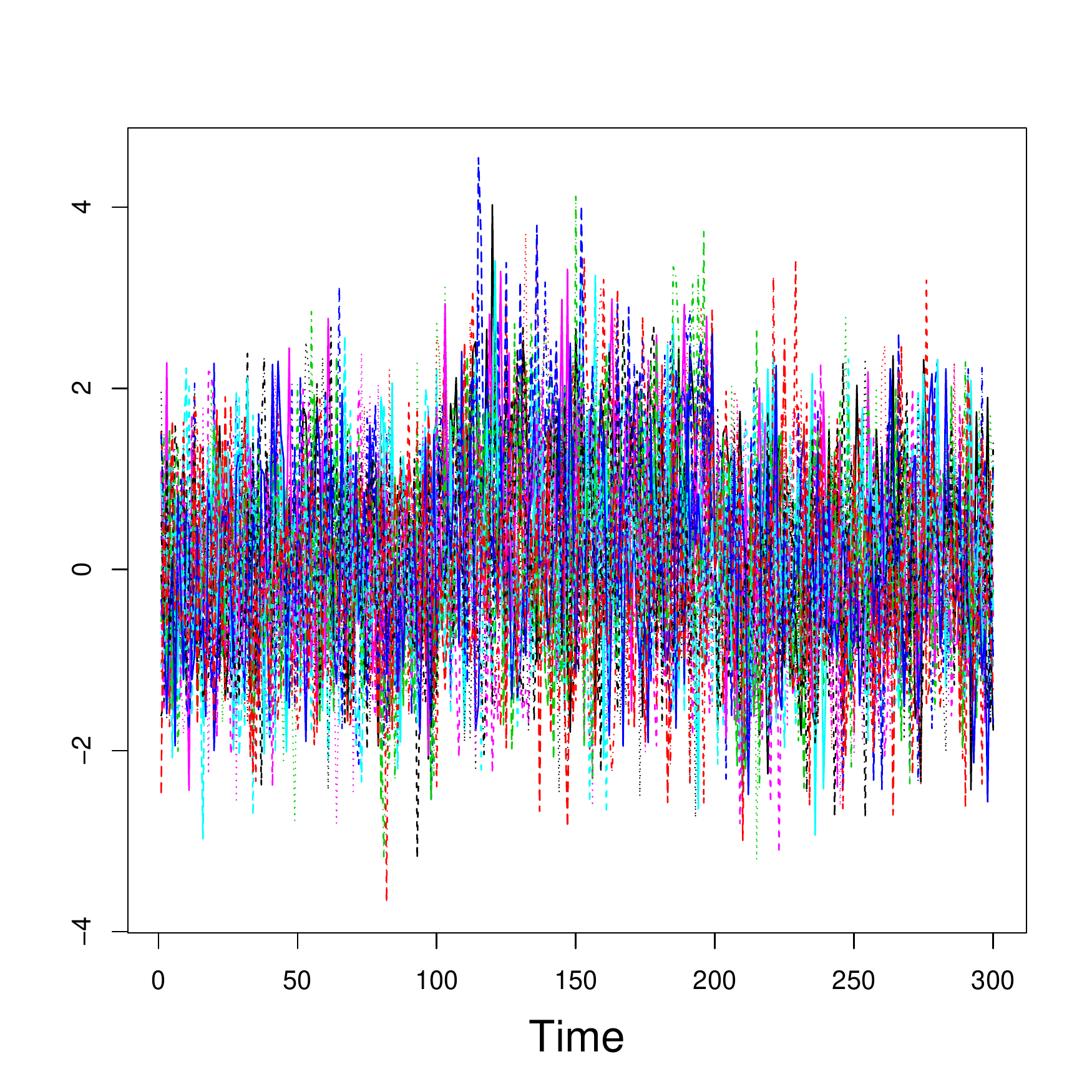} 
			\includegraphics[width=3in,height= 2in]{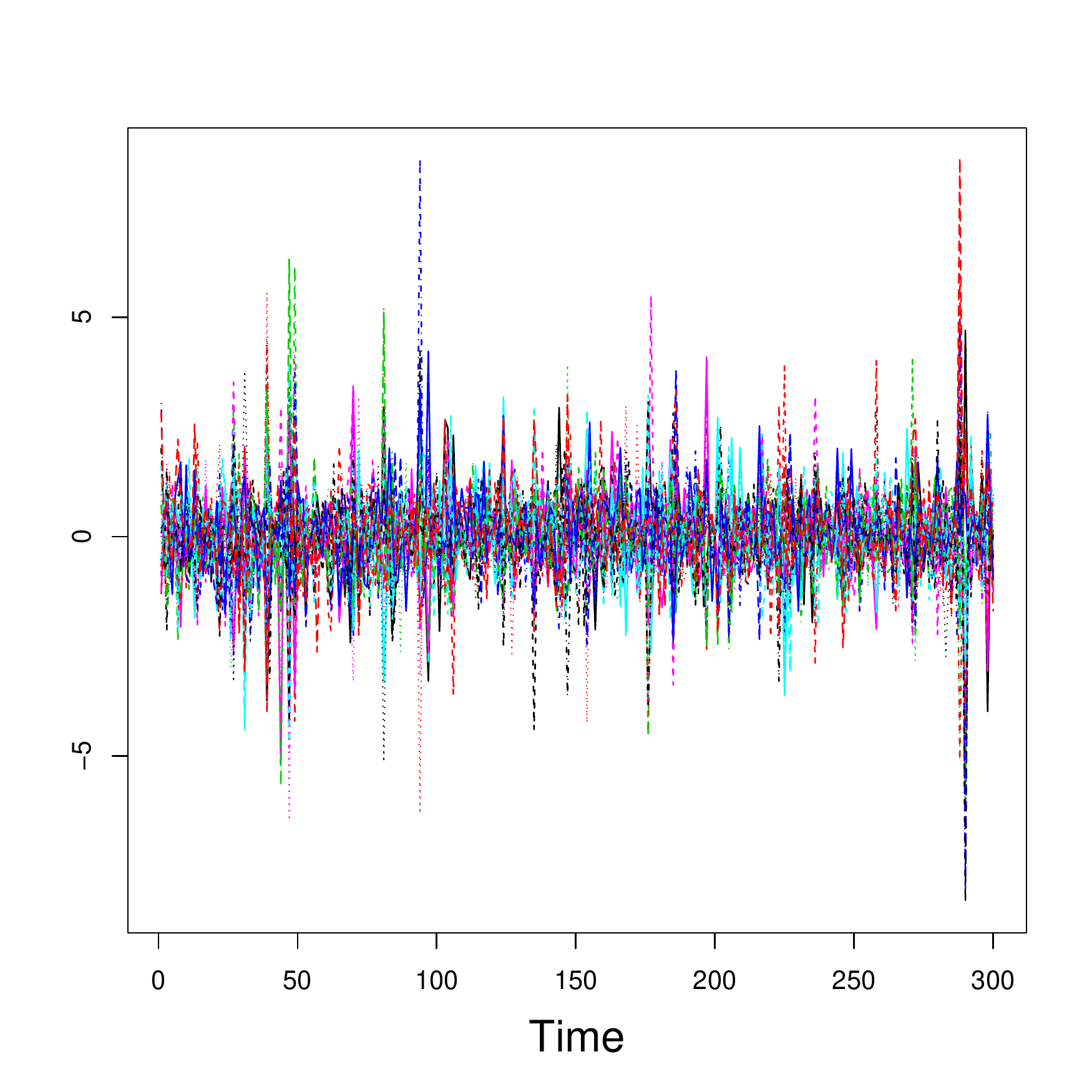} \\
				\includegraphics[width=3in,height=2in]{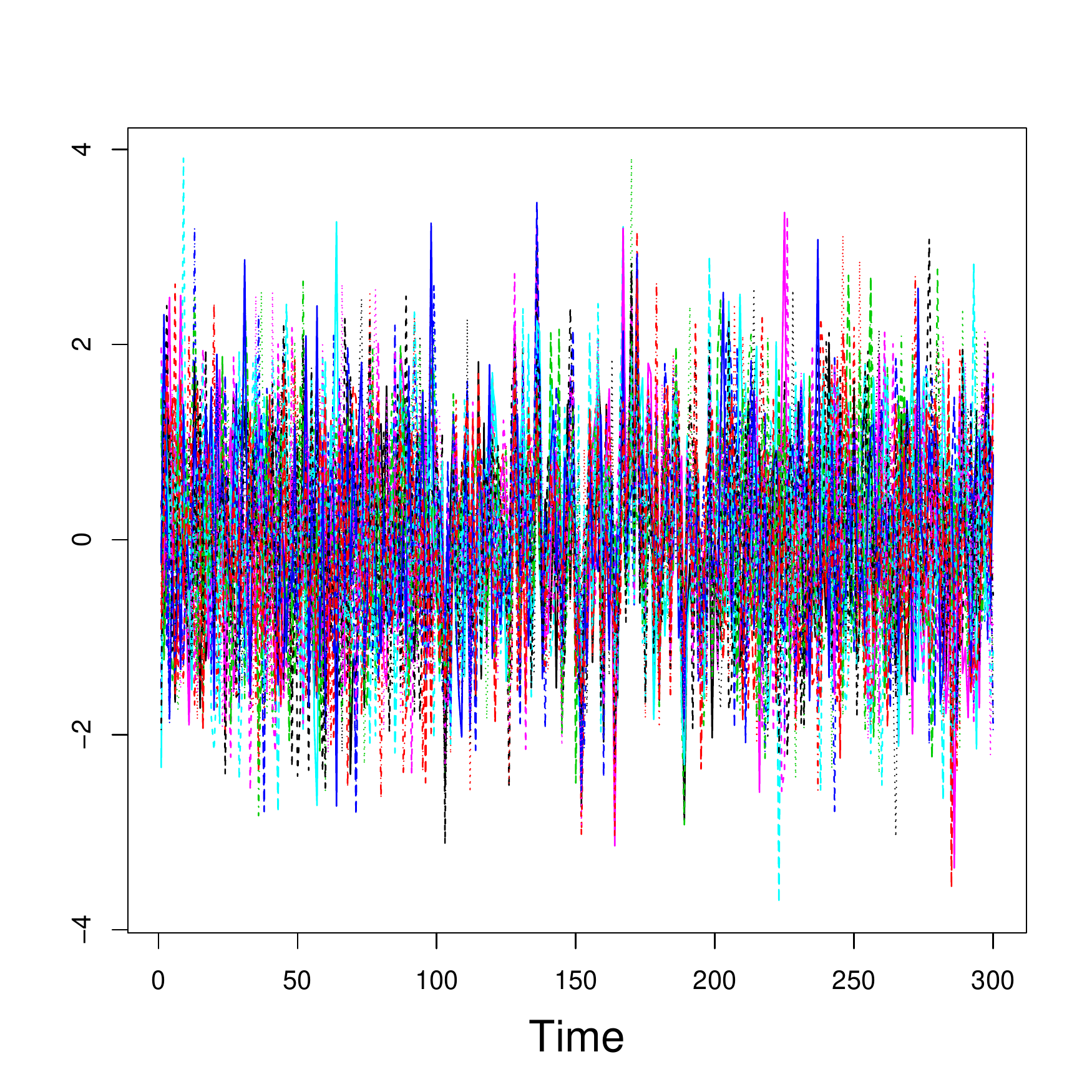} 
			\includegraphics[width=3in,height= 2in]{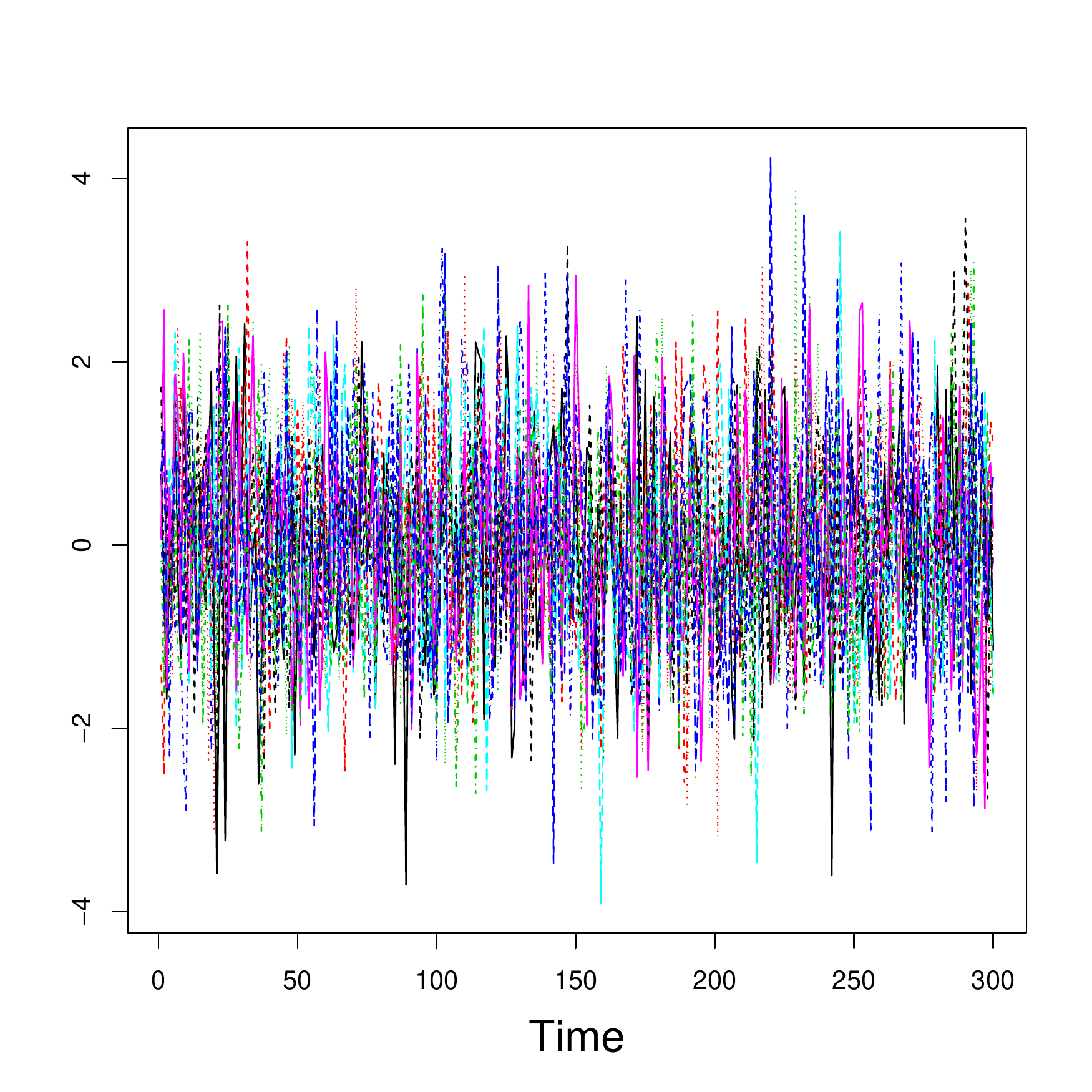}\\
					\includegraphics[width=3in,height= 2in]{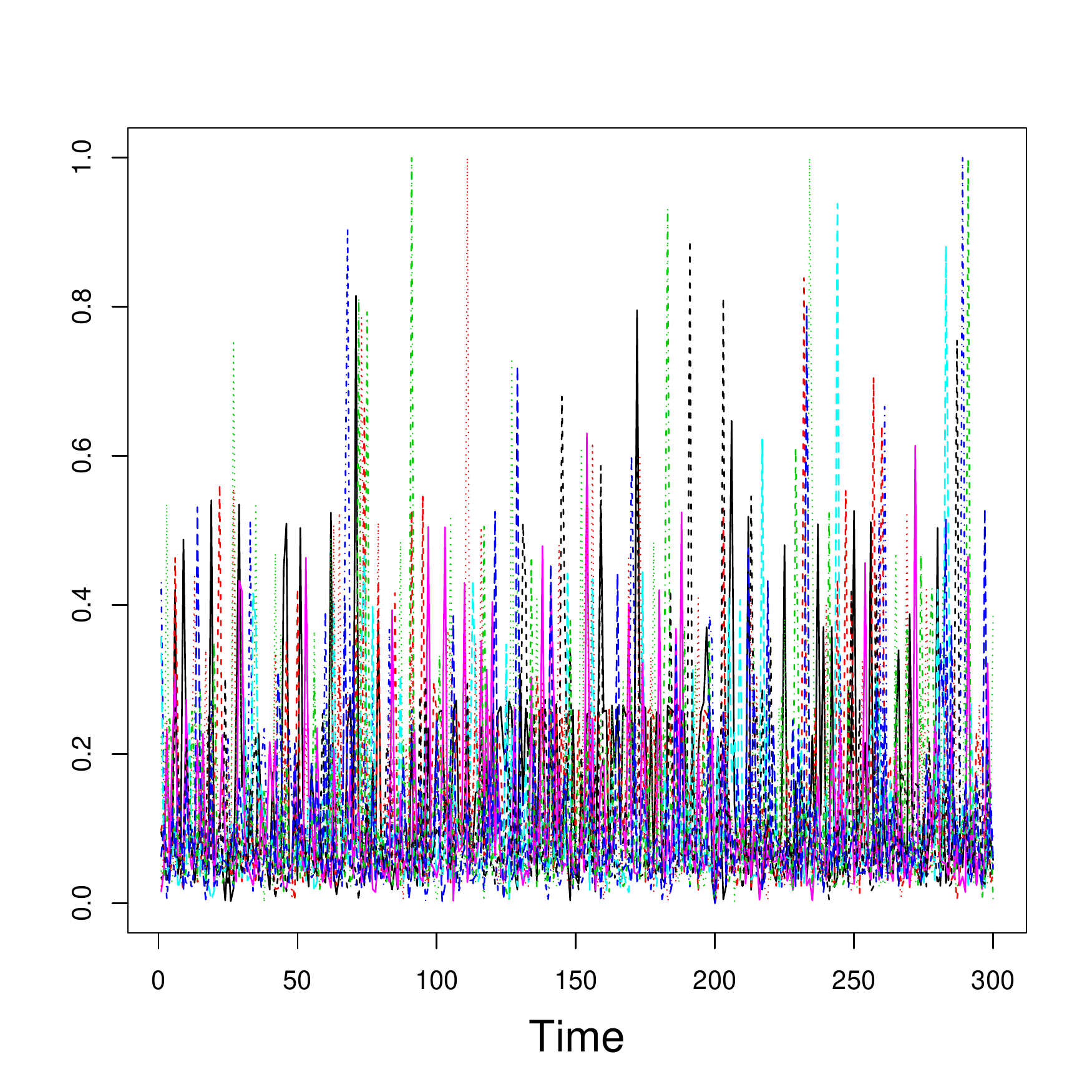}
						\includegraphics[width=3in,height= 2in]{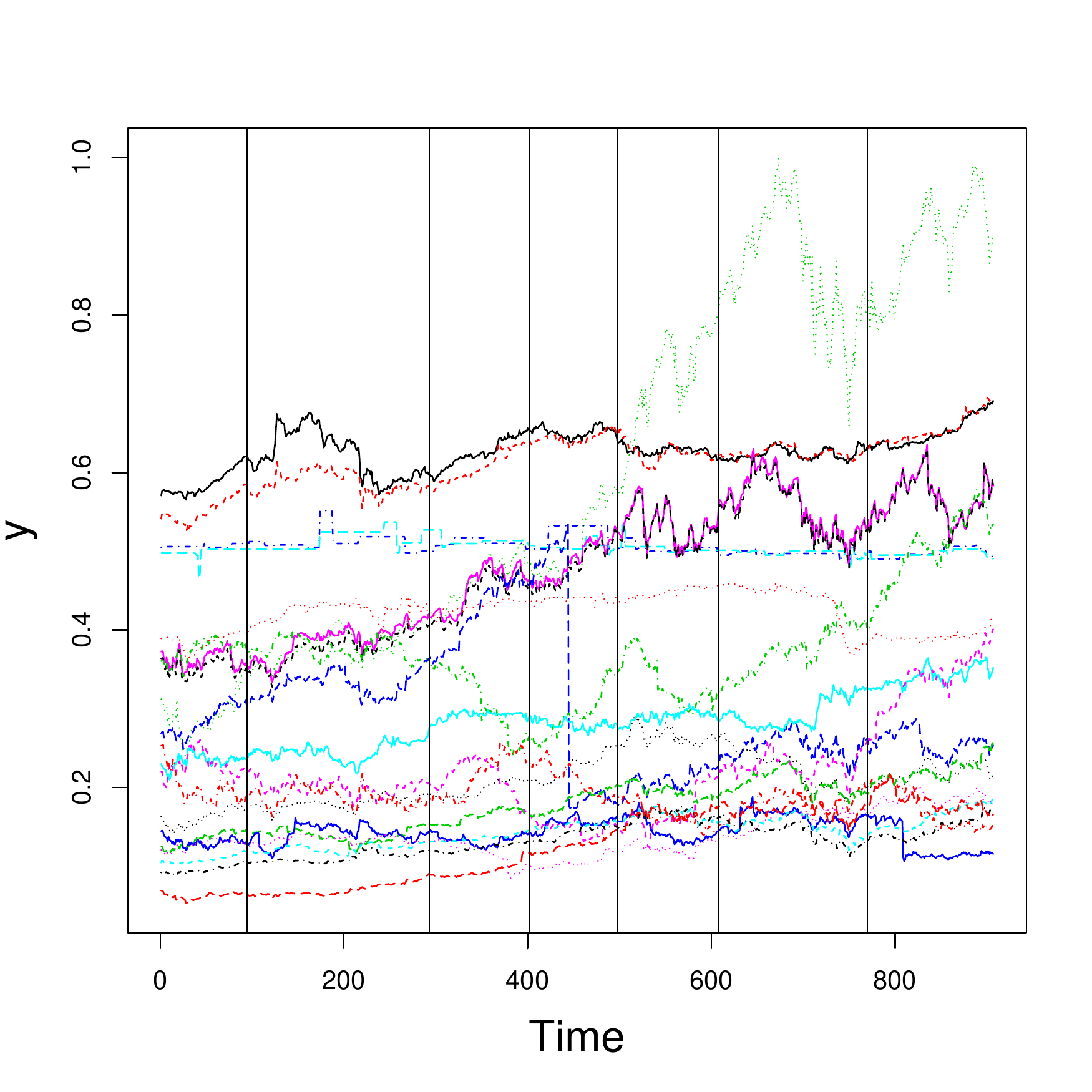} 
		\caption{\label{fig1}     From left to right and from top to bottom,  the first five plots  illustrate raw data generated  from Scenarios 1 to 5, respectively, with  one realization each. In each case, $T=300$ and $p=20$, with  the x-axis representing the time horizon, and the y-axis the values  of each measurement. Different curves  in each plot are associated  with different coordinates of the vector $X(t)$. The right panel in the third row illustrates  the raw data and estimated change points by MNP for the example in Section \ref{sec:real_Data}. }%
		% Densities  taken from \cite{padilla2018sequential}.  These are   used  in our generative models for the distributions between change points.  }
	\end{center}
\end{figure}
%plot_Example_4
\paragraph{Scenario 1.} We generate data as
	\[
		X(t) = \mu(t) +  \epsilon(t), \quad t \in \{1, \ldots, T\},
	\]
	where $\epsilon(t) \sim N(0, I_p)$ and $I_p$ is the $p \times p$ identity   matrix.  Moreover, the mean vectors satisfy
	\[
		\mu(t) =  \begin{cases}
			v^{(0)}  & t \in A_1 \cup A_3,\\
			v^{(1)} & \text{otherwise,}
		\end{cases} 
	\]
	where  $v^{(0)} = 0 \in \mathbb{ R}^p$, and $v_j^{(1)} = 1$  for $j \in \{1,\ldots,p/2\}$ and   $v_j^{(1)} = 0$ otherwise.

\paragraph{Scenario 2. } This is the  same  as  Scenario 1  but with the errors  satisfying
  $ \sqrt{3}\, \epsilon(1),\ldots, \sqrt{3}\, \epsilon(T)   \stackrel{\mbox{i.i.d.}}{\sim}  \mathrm{Mt}(I_p,3)$, which is the multivariate $t$-distribution with the scale matrix $I_p$ and the degrees of freedom three.  With respect to Scenario 1 we also change the value of $v^{(1)}$. This is now $v^{(1)} = 0.1 \cdot\boldsymbol{1}$ with  $\boldsymbol{1} = (1,\ldots,1)^{\top} \in \mathbb{R}^p$.

\paragraph{Scenario 3.} We generate  observations from the model
	\[
		X(t) \stackrel{\mbox{i.i.d.}}{\sim}  N(0,  \Sigma(t) ), \quad t \in \{1, \ldots, T\},
	\]
	where
	\[
		\Sigma(t) =  \begin{cases}
			I_p &  t \in A_1 \cup A_3, \\
			\frac{1}{2}I_p  + \frac{1}{2} \boldsymbol{1} \boldsymbol{1}^{\top}  & \text{otherwise.}
		\end{cases}
	\]

\paragraph{Scenario 4.} The observations are  constructed as $X(t) \stackrel{\mbox{i.i.d.}}{\sim} N(0, 1.25  I_p)$  for  $t \in A_1\cup A_3$, and  for  $t \in A_2$  we have
	\[
		X(t) | \{u_t = 1\} \stackrel{\mbox{i.i.d.}}{\sim}  N(0.5\cdot\boldsymbol{1}, I_p) \quad \mbox{and} \quad X(t) | \{ u_t = 2 \} \stackrel{\mbox{i.i.d.}}{\sim} N(- 0.5\cdot\boldsymbol{1}, I_p),
	\]
	where the i.i.d.~random variables $\{u_t\}$ satisfy $\mathbb{P}(u_t =1) = \mathbb{P}(u_t =2) = 1/2$.

\paragraph{Scenario 5.} The vector $X(t)$  satisfies   $X_j(t) \sim   g_1$ for  $t \in A_1 \cap A_3$  and  for all  $j \in \{1,\ldots,p\}$. In contrast, if   $t\in A_2$ we have that
	\[
    	X_j(t)  \sim  \begin{cases}
			g_1, &  j  \in \{1, 2 \}, \\
			g_2,  &  \text{otherwise.}
		\end{cases}
	\] 
	Here   $g_1$ and $g_2$  are the densities shown in the left and right panels in \Cref{fig2}, respectively. 

\begin{figure}[t!]
	\begin{center}
		\includegraphics[width=3.2in,height= 2.7in]{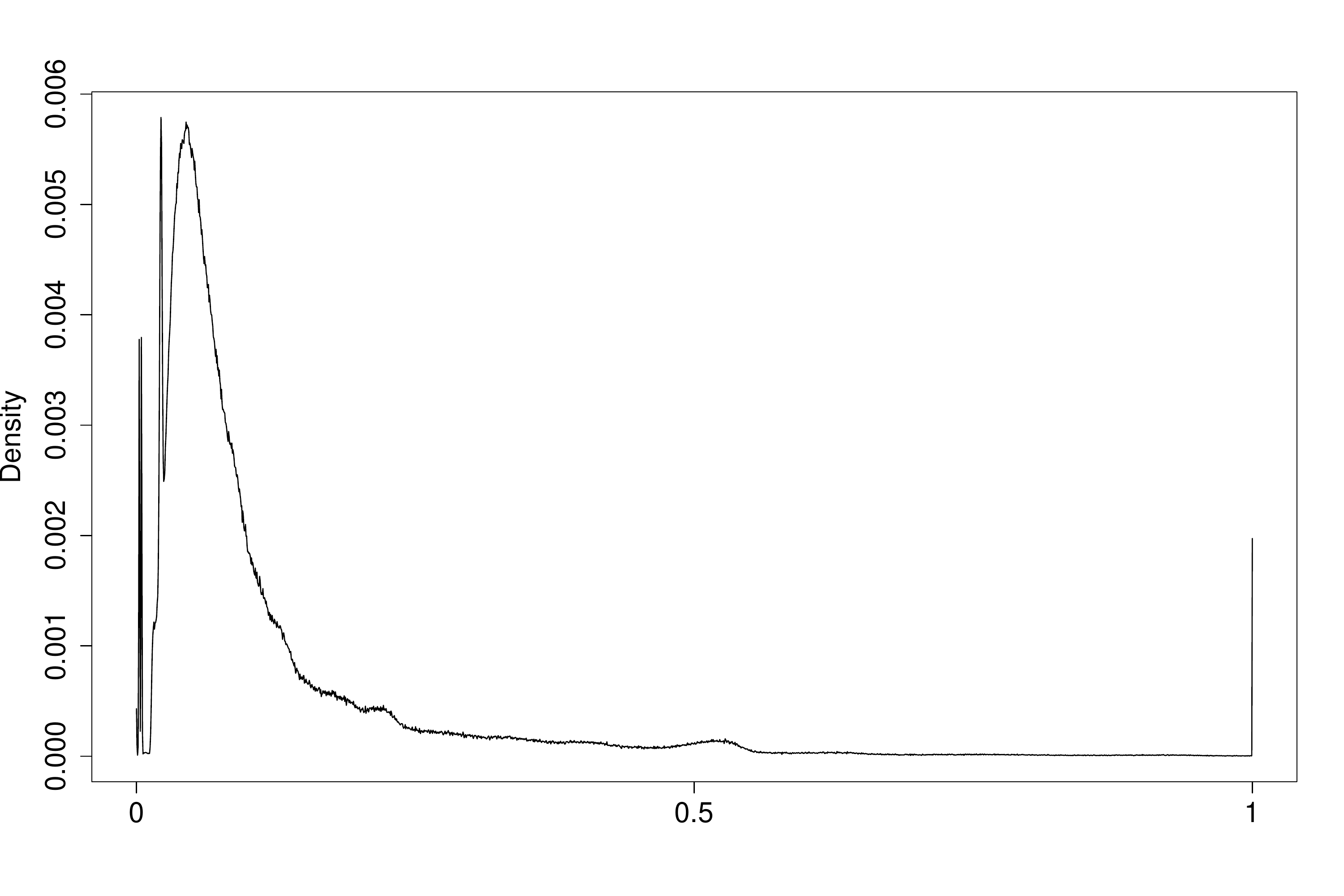} 
		\includegraphics[width=3.2in,height= 2.7in]{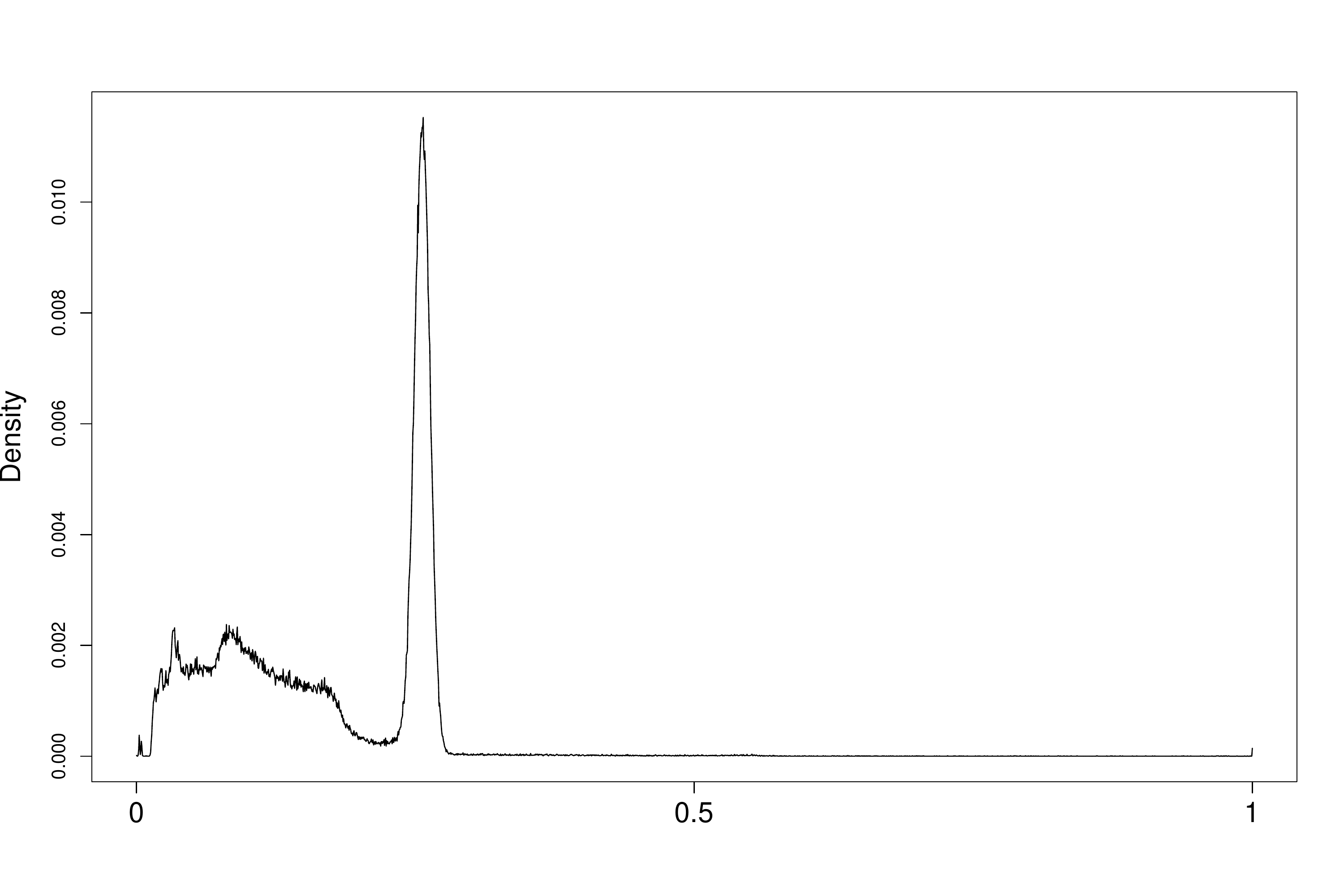}
		\caption{\label{fig2}  Densities  taken from \cite{padilla2018sequential} and used in Scenario 5.}
	\end{center}
\end{figure}

\begin{table}[t!]
	\centering
	\caption{\label{tab1} Scenario 1.}
	\medskip
	\setlength{\tabcolsep}{14pt}
	\begin{small}
		\begin{tabular}{rrrrrrr}
			%{ rrrrrrrrrrr }
			\hline
		Method   & Metric                 &    T =  300&    T =  300&    T =  150 &    T =  150 \\
	&                        &     p=20           &         p=10      &           p=20 &       p=10  \\
		\hline
%Method   & Metric                 &    T =  300&    T =  300      &    T =  150 &    T =  150 \\
%&                        &     p=20  &         p=10      &           p=20 &       p=10  \\
\hline
%	\hline
MNP    & $\vert K - \widehat{K}\vert$ &	     \textbf{0.0}            &  \textbf{0.0}             &  \textbf{0.0}             &   \textbf{0.0}           \\         
EMNCP    & $\vert K - \widehat{K}\vert$ &	    0.1             &     \textbf{0.0}          &     \textbf{0.0}          &      \textbf{0.0}        \\  
KCPA    & $\vert K - \widehat{K}\vert$ &	 1.4           &       0.9     &    1.3      &     1.0       \\  
SBS   & $\vert K - \widehat{K}\vert$ &	 2.0       &  2.0          &    2.0      &    2.0        \\  
DCBS  & $\vert K - \widehat{K}\vert$ &	  \textbf{0.0}                    &     	  \textbf{0.0}             &     	  \textbf{0.0}         &     0.1       \\  
\hline	
MNP    & $d(\widehat{\mathcal{C}}|\mathcal{C})$ &	   1.0              &   2.0           &     1.0               &     2.0        \\
EMNCP    & $d(\widehat{\mathcal{C}}|\mathcal{C})$ &	   \textbf{0.0}           &  \textbf{1.0}              &     \textbf{0.0}              &    \textbf{0.0}      \\   
KCPA    & $d(\widehat{\mathcal{C}}|\mathcal{C})$ &	    175.0       &    119.0  &  60.0      &      62.0      \\  
SBS  & $d(\widehat{\mathcal{C}}|\mathcal{C})$ &	   $\infty$         &     $\infty$        &       $\infty$          &     $\infty$             \\  
DCBS & $d(\widehat{\mathcal{C}}|\mathcal{C})$ &	   1.0          &     \textbf{1.0}          &    1.0        &    3.0        \\  
\hline	
MNP    & $d(\mathcal{C} |\widehat{\mathcal{C}})$ &	   1.0               &     2.0           &       1.0     &         2.0     \\
EMNCP    & $d(\mathcal{C} |\widehat{\mathcal{C}})$ &	\textbf{0.0}           &        \textbf{1.0}         &\textbf{0.0}                     &  \textbf{0.0}           \\   
KCPA    & $d(\mathcal{C} |\widehat{\mathcal{C}})$ &    1.0         &   19      &   1.0   &      13.0      \\      
SBS    & $d(\mathcal{C} |\widehat{\mathcal{C}})$ &	   $-\infty$     &   $-\infty$             &   $-\infty$                &      $-\infty$        \\
DCBS  & $d(\mathcal{C} |\widehat{\mathcal{C}})$ &	1.0              &  \textbf{1.0}     &   1.0         &         3.0   \\          
%	d(\mathcal{C} |\hat{\mathcal{C}})$            &\textbf{18.0} &28.0             &254.5   & 2179  &$-\infty$  &129.5      &257       &19.5       &20.0\\ 
\hline		
		\end{tabular}
	\end{small}
\end{table}

\begin{table}[t!]
	\centering
	\caption{\label{tab2} Scenario 2.}
	\medskip
	\setlength{\tabcolsep}{14pt}
	\begin{small}
		\begin{tabular}{rrrrrrr}
			%{ rrrrrrrrrrr }
				\hline
		Method   & Metric                 &    T =  300&    T =  300      &    T =  150 &    T =  150 \\
		&                        &     p=20  &         p=10      &           p=20 &       p=10  \\
		\hline
			\hline
		MNP    & $\vert K - \widehat{K}\vert$ &	 \textbf{0.9}                 & 1.1            &  1.5           &     1.5      \\        
		EMNCP    & $\vert K - \widehat{K}\vert$ &	        1.7         &      1.7         &       1.8       &    1.8         \\  
		KCPA    & $\vert K - \widehat{K}\vert$ &1.04	            &      \textbf{1.0}     &  \textbf{1.0}            &     \textbf{1.4}            \\  
		SBS   & $\vert K - \widehat{K}\vert$ &	  2.0             &   2.0           &    2.0        &     2.0       \\  
		DCBS  & $\vert K - \widehat{K}\vert$ &	2.0              &     2.0          &    1.96        &  2.0          \\  
		\hline	
		MNP    & $d(\widehat{\mathcal{C}}|\mathcal{C})$ &  \textbf{74.0}               &      \textbf{90}        &    86.0                 &    \textbf{81.0}         \\
		EMNCP    & $d(\widehat{\mathcal{C}}|\mathcal{C})$ &	  $\infty$            &    $\infty$            &      $\infty$                &         $\infty$ \\   
KCPA    & $d(\widehat{\mathcal{C}}|\mathcal{C})$ &	126.0           &  133.0     &     \textbf{ 51.0}        &       86.0     \\  
SBS  & $d(\widehat{\mathcal{C}}|\mathcal{C})$ &	         $\infty$         &            $\infty$      &      $\infty$       &    $\infty$         \\  
DCBS & $d(\widehat{\mathcal{C}}|\mathcal{C})$ &	      $\infty$        &       $\infty$              &     $\infty$        &          $\infty$         \\  		
		\hline	
		MNP    & $d(\mathcal{C} |\widehat{\mathcal{C}})$ &	  \textbf{26.0}                &  35.0               & $-\infty$                  &  \textbf{1.0}            \\
		EMNCP    & $d(\mathcal{C} |\widehat{\mathcal{C}})$ &	$-\infty$           &   $-\infty$              &    $-\infty$                   &  $-\infty$            \\    
		KCPA    & $d(\mathcal{C} |\widehat{\mathcal{C}})$ &38.0	          &           \textbf{33.0}   &      \textbf{12.0}        &        12.0   \\      
		SBS    & $d(\mathcal{C} |\widehat{\mathcal{C}})$ &	    $-\infty$                &      $-\infty$        &      $-\infty$       &         $-\infty$          \\
		DCBS  & $d(\mathcal{C} |\widehat{\mathcal{C}})$ &	    $-\infty$                &       $-\infty$         &   $-\infty$          &            $-\infty$       \\          
		%	d(\mathcal{C} |\hat{\mathcal{C}})$            &\textbf{18.0} &28.0             &254.5   & 2179  &$-\infty$  &129.5      &257       &19.5       &20.0\\ 
		\hline	
			%	d(\mathcal{C} |\hat{\mathcal{C}})$            &\textbf{18.0} &28.0             &254.5   & 2179  &$-\infty$  &129.5      &257       &19.5       &20.0\\ 
		%	\hline		
		\end{tabular}
	\end{small}
\end{table}

\begin{table}[t!]
	\centering
	\caption{\label{tab3} Scenario 3.}
	\medskip
	\setlength{\tabcolsep}{14pt}
	\begin{small}
		\begin{tabular}{rrrrrrr}
			%{ rrrrrrrrrrr }
			\hline
	Method   & Metric                 &    T =  300&    T =  300      &    T =  150 &    T =  150 \\
	&                        &     p=20  &         p=10      &           p=20 &       p=10  \\
	\hline
		\hline
	MNP    & $\vert K - \widehat{K}\vert$ &	    \textbf{0.5}             &   \textbf{0.9}            &   1.4             &    1.3        \\         
		EMNCP    & $\vert K - \widehat{K}\vert$ &	      1.3           &    1.4          &       1.6        &   2.0          \\ 
				KCPA    & $\vert K - \widehat{K}\vert$ &	1.2                &      0.9       &   \textbf{ 1.0 }    &   \textbf{1.1}         \\  
		SBS   & $\vert K - \widehat{K}\vert$ &	     2.0           &      2.0       &      2.0      &        2.0    \\  
		DCBS  & $\vert K - \widehat{K}\vert$ &	   2.0             &     2.0        &        1.96    &        2.0    \\  
	\hline	
	MNP    & $d(\widehat{\mathcal{C}}|\mathcal{C})$ &	  \textbf{35.0}               &    \textbf{72.0}         & \textbf{50.0}                   &   \textbf{42.0}          \\
	EMNCP    & $d(\widehat{\mathcal{C}}|\mathcal{C})$ &	  $\infty$              &       $\infty$            &          $\infty$           &     $\infty$      \\   
KCPA    & $d(\widehat{\mathcal{C}}|\mathcal{C})$ &	     97.0           &    102.0         &   \textbf{ 50.0}        &    51       \\  
SBS  & $d(\widehat{\mathcal{C}}|\mathcal{C})$ &	      $\infty$            &  $\infty$           &   $\infty$            &       $\infty$           \\  
DCBS & $d(\widehat{\mathcal{C}}|\mathcal{C})$ &	     $\infty$             &    $\infty$         &   $\infty$            &       $\infty$           \\  	
	\hline	
	MNP    & $d(\mathcal{C} |\widehat{\mathcal{C}})$ &	 18.0               &       27.0             &      12.0             &    \textbf{8.0}          \\
	EMNCP    & $d(\mathcal{C} |\widehat{\mathcal{C}})$ &	   $-\infty$          &    $-\infty$               &         $-\infty$             &      $-\infty$        \\ 
		KCPA    & $d(\mathcal{C} |\widehat{\mathcal{C}})$ &	   \textbf{6.0}            &   \textbf{6.0}          &       \textbf{2.0}     &   10         \\      
SBS    & $d(\mathcal{C} |\widehat{\mathcal{C}})$ &	   $-\infty$               &      $-\infty$       &       $-\infty$         &       $-\infty$           \\
DCBS  & $d(\mathcal{C} |\widehat{\mathcal{C}})$ &	  $-\infty$                &     $-\infty$        &      $-\infty$          &       $-\infty$           \\   	        
	%	d(\mathcal{C} |\hat{\mathcal{C}})$            &\textbf{18.0} &28.0             &254.5   & 2179  &$-\infty$  &129.5      &257       &19.5       &20.0\\ 
	\hline		
		\end{tabular}
	\end{small}
\end{table}

\begin{table}[t!]
	\centering
	\caption{\label{tab4} Scenario 4.}
	\medskip
	\setlength{\tabcolsep}{14pt}
	\begin{small}
		\begin{tabular}{rrrrrrr}
		\hline
Method   & Metric                 &    T =  300&    T =  300      &    T =  150 &    T =  150 \\
&                        &     p=20  &         p=10      &           p=20 &       p=10  \\
\hline
	\hline
MNP    & $\vert K - \widehat{K}\vert$ &	  \textbf{0.7}               &  \textbf{0.9}             & \textbf{1.1}             &  1.4           \\         
	EMNCP    & $\vert K - \widehat{K}\vert$ &	 1.8                &  1.8             &            2.0  &      1.8      \\
				KCPA    & $\vert K - \widehat{K}\vert$ &	1.2                &     1.1        &  1.4          &   \textbf{1.2}         \\  
SBS   & $\vert K - \widehat{K}\vert$ &	    2.0            &    2.0         &      2.0      &     2.0       \\  
DCBS  & $\vert K - \widehat{K}\vert$ &	       2.0         &       2.0      &        1.9    &       2.0     \\  	 
\hline	
MNP    & $d(\widehat{\mathcal{C}}|\mathcal{C})$ &	\textbf{38.0}                 &  \textbf{68.0}            &   \textbf{43.0}                 &   \textbf{65.0}          \\
EMNCP    & $d(\widehat{\mathcal{C}}|\mathcal{C})$ &	   $\infty$             &         $\infty$         &     $\infty$                &    $\infty$       \\   
KCPA    & $d(\widehat{\mathcal{C}}|\mathcal{C})$ &	   99.0             &   112.0         &     72.0       &    81.0         \\  
SBS  & $d(\widehat{\mathcal{C}}|\mathcal{C})$ &	    $\infty$             &    $\infty$           &        $\infty$           &     $\infty$             \\  
DCBS & $d(\widehat{\mathcal{C}}|\mathcal{C})$ &	    $\infty$             &      $\infty$         &        $\infty$           &       $\infty$           \\  	
\hline	
MNP    & $d(\mathcal{C} |\widehat{\mathcal{C}})$ &	36.0                  &   \textbf{33.0}               &        7.0          &   \textbf{6.0}           \\
EMNCP    & $d(\mathcal{C} |\widehat{\mathcal{C}})$ &	   $-\infty$          &     $-\infty$              &          $-\infty$           &        $\infty$      \\    
		KCPA    & $d(\mathcal{C} |\widehat{\mathcal{C}})$ &	  \textbf{3.0}              &    34         &      \textbf{1.0}       &       22.0     \\      
SBS    & $d(\mathcal{C} |\widehat{\mathcal{C}})$ &	     $-\infty$            &      $-\infty$          &       $-\infty$            &       $-\infty$           \\
DCBS  & $d(\mathcal{C} |\widehat{\mathcal{C}})$ &	    $-\infty$             &     $-\infty$           &       $-\infty$            &       $-\infty$           \\   	      
%	d(\mathcal{C} |\hat{\mathcal{C}})$            &\textbf{18.0} &28.0             &254.5   & 2179  &$-\infty$  &129.5      &257       &19.5       &20.0\\ 
\hline		
		\end{tabular}
	\end{small}
\end{table}

\begin{table}[t!]
\centering
\caption{\label{tab5} Scenario 5.}
\medskip
\setlength{\tabcolsep}{14pt}
\begin{small}
	\begin{tabular}{rrrrrrr}
		%{ rrrrrrrrrrr }
		\hline
		Method   & Metric                 &    T =  300&    T =  300      &    T =  150 &    T =  150 \\
		&                        &     p=20  &         p=10      &           p=20 &       p=10  \\
		\hline
			\hline
		MNP    & $\vert K - \widehat{K}\vert$ &	  \textbf{0.8}               &     0.8          &    1.7         &    \textbf{1.2}         \\         
			EMNCP    & $\vert K - \widehat{K}\vert$ &	1.6                 &       \textbf{0.6}        &    1.9          &    1.4         \\ 
				KCPA    & $\vert K - \widehat{K}\vert$ &0.9	                &   0.9          &     \textbf{1.3}        &     1.4       \\  
SBS   & $\vert K - \widehat{K}\vert$ &	       2.0         &    2.0         &    2.0        &        2.0    \\  
DCBS  & $\vert K - \widehat{K}\vert$ &	         1.8       &   1.9          &       2.0     &           2.0 \\			
		\hline	
		MNP    & $d(\widehat{\mathcal{C}}|\mathcal{C})$ &	   \textbf{41.0}             &   45           &      $\infty$                &   \textbf{53.0}          \\
		EMNCP    & $d(\widehat{\mathcal{C}}|\mathcal{C})$ &	    $\infty$            &   \textbf{12.0}             &      $\infty$               &    $\infty$       \\ 
KCPA    & $d(\widehat{\mathcal{C}}|\mathcal{C})$ &	     105.0           &   124.0          &   \textbf{95.0}         &75.0            \\  
SBS  & $d(\widehat{\mathcal{C}}|\mathcal{C})$ &	     $\infty$              &       $\infty$            &   $\infty$           &      $\infty$         \\  
DCBS & $d(\widehat{\mathcal{C}}|\mathcal{C})$ &	    $\infty$               &       $\infty$            &   $\infty$           &       $\infty$        \\  			  
		\hline	
		MNP    & $d(\mathcal{C} |\widehat{\mathcal{C}})$ &	  35.0               &  42                &     $-\infty$                 &    \textbf{9.0}         \\
		EMNCP    & $d(\mathcal{C} |\widehat{\mathcal{C}})$ &	    $-\infty$         & \textbf{7.0}               &     $-\infty$                 &     $-\infty$         \\ 
		KCPA    & $d(\mathcal{C} |\widehat{\mathcal{C}})$ &	  \textbf{21.0}              &  24.0           &     \textbf{17.0}       &     17.0       \\      
SBS    & $d(\mathcal{C} |\widehat{\mathcal{C}})$ &	     $-\infty$              &      $-\infty$            &    $-\infty$          &      $-\infty$         \\
DCBS  & $d(\mathcal{C} |\widehat{\mathcal{C}})$ &	       $-\infty$            &     $-\infty$             &   $-\infty$           &       $-\infty$        \\   	 		        
		%	d(\mathcal{C} |\hat{\mathcal{C}})$            &\textbf{18.0} &28.0             &$-\infty$     254.5   & 2179  &$-\infty$  &129.5      &257       &19.5       &20.0\\ 
		\hline		
	\end{tabular}
\end{small}
\end{table}

%\begin{figure}[t!]
%	\begin{center}
%		\includegraphics[width=3.2in,height= 2.7in]{plot_stocks.pdf} 
%	%	\includegraphics[width=3.2in,height= 2.7in]{f02.pdf}
%		%	\includegraphics[width=3.2in,height= 2.7in]{f03.pdf} 
%%		%\includegraphics[width=3.2in,height= 2.7in]{f04.pdf}
%		\caption{\label{fig4}  Densities  taken from \cite{padilla2018sequential}.}
%%	\end{center}
%\end{figure}

Figure  \ref{fig1} illustrates  examples of data generated  from  each of the scenarios that we consider.  This is complemented by the results in Tables \ref{tab1}--\ref{tab5}. Specifically, we observe that for Scenario  1, a setting with mean  changes, the best methods seem to be MNP,  DCBS and EMNCP.
% both methods EMNCP and MNP perform  well with the former providing slightly better  estimates.
%Algorithm   \ref{algorithm:WBS}  outperforms  EMNCP   in both  estimating the number  of change points and estimating their locations.

Interestingly,  from  Table \ref{tab2}, we see that   KCPA and   MNP outperform the other  methods.  This setting presents a bigger challenge than Scenario 1,  as it involves  a heavy-tailed distribution of the errors and smaller changes in mean.
% qualitative different from Scenario 1,  as in Scenario 2 the mean remains constant  in time and jumps happen at the covariance level.
%MNP shows considerable advantage over EMNCP  in  Scenario  2. This setting presents a bigger challenge than Scenario 1,  as it involves  a heavy-tailed distribution of the errors and smaller changes in mean.
% qualitative different from Scenario 1,  as in Scenario 2 the mean remains constant  in time and jumps happen at the covariance level.

Scenario  3 posses a situation  where  the mean  remains  constant and the covariance structure changes.  From Table \ref{tab3}, we observe that MNP  and KCPA attain the best performance.
%both the mean and  covariance 
%  do not change in time, and the breaks occur in the number of modes. Thus,  in some  intervals in time the data  is generated from a mixture of two  multivariate normal distributions, whereas in other intervals  the distribution is a single multivariate normal. Table  \ref{tab3}   suggests that Algorithm  \ref{tab3} is capable of detecting  shape  changes  more accurately   than  EMNCP for this example.

In Table \ref{tab4}, we also see the advantage of the  MNP method which is the best at estimating the number of change points.  This is in the context of  Scenario  4   where  the mean and covariance remain unchanged and the jumps happen in the shape of the distribution.

Finally,  Scenario 5 is an example of a model that does not belong to a usual parametric family. In such setting, Table  \ref{tab5} shows that MNP, KCPA and EMNCP seem to provide better  estimation  of the number of change points and their locations as compared to the other two methods.

Overall, we can see that SBS and DCBS, these two methods designed for mean change point detection, are not robust in the cases where the changes are not mean or the noise is not sub-Gaussian.  MNP and KCPA are, arguably, the best performing methods.  KCPA is competitive and sometimes outperforms MNP.
% although it lacks theoretical guarantees on the localization rates. \textcolor{red}{Actually, the EJS paper ``Consistent change-point detection with kernels'' has theory for localization. See Theorem 3.1 in that reference}

\subsection{Real data example}
\label{sec:real_Data}
  The experiments  section concludes  with an example using financial data. Specifically, our data consist of  the 
daily close stock price, from Jan-1-2016 to Aug-11-2019, 	of the    20  companies with highest average stock price  from  the S\&P500 market.  The data   can be   downloaded from \cite{yahoo}.  Our final   dataset is then  a matrix  $X  \in \mathbb{ R}^{T \times p}$, with $T = 907$ and  $p=20$.  

We then  run  both the MNP  procedure and the estimator from \cite{matteson2014nonparametric}.  The implementation and details  are the same  as those in Section \ref{sec_sim}. 
Our goal   is to detect  potential change points   in the  period  aforementioned  and determine if they might have  a financial meaning.

We find that  our estimator localizes change  points at the dates May-17-2016, March-2-2017, August-7-2017, December-21-2017, June-1-2018 and January-24-2019. The first change point seems to correspond with the moment when President Donald Trump, while still a presidential candidate,  outlined  his plan for the  USA vs.~China   trade war \citep[see e.g.][]{tradewar}. The second change point,  February-21-2017,  might be associated   with Trump  signing two executive orders  increasing  tariffs on the trade with China;  the date August-7-2017  could  correspond  to  the bipartite  agreement on July-19 2017 to reduce USA deficit with China;  the date  December-21-2017   could   be explained  by the  threats and tariffs  imposed by  Trump to China  in January of 2018. The other two dates  are also relatively close to important dates  in the USA vs.~China  trade war time-line. The raw data, scaled to the interval $[0,1]$, and the estimated change points  can be seen in the right panel in the third row in \Cref{fig1}.

As for   EMNCP,   we find  a total of 22 change points with spacings between 30 and 58 units of time. This might suggest that some of the change point are spurious as  the minimum spacing parameter of the function \texttt{e.divisive()}  is by default  set to 30.

Finally, we also considered comparing  with KCPA. However, the elbow could not be used as the scores produced by the function \textit{KernSeg\_MultiD} did not have an inflection point.
%\section{Conclusion}

%In this paper, we tackle a multivariate nonparametric change point detection problem, which aims to provide with change point estimators robust against model mis-specification.  The computational-efficient method we propose has matched minimax lower bounds, off by logarithm factors, in terms of both the signal-to-noise ratio condition and the localization rate.  The lower bounds are also presented in this paper, which is self-contained.  The theoretical findings are backed up by extensive numerical experiments, including a real data example.  Possible extensions of this this paper include characterizing change points by other measures, instead of the supreme norm of the density function differences.  Different measures would require different methods, the algorithmic efficiency and theoretical optimality are remained interesting and open. 

\appendix
\section{Large probability events}
\label{sec:appendx_a}

In this section, we deal with all the large probability events occurred in the proof of \Cref{thm-wbs}.  \Cref{lem-bousquet-2.1} is almost identical to Theorem~2.1 in \cite{bousquet2002bennett} and therefore we omit the proof.  \Cref{lem-bousquet-2.3} is an adaptation of Theorem~2.3 in \cite{bousquet2002bennett} and Proposition~8 in \cite{kim2018uniform}, but we allow for non-i.i.d.~cases.    \Cref{lem-gine} is a non-i.i.d.~version of Proposition~2.1 in \cite{gine2001consistency}.  Lemma~\ref{lem-1} is to control the deviance between the sample and population quantities and provides an lower bound on a large probability event.  Lemma~\ref{lem-4} is to provide a lower bound on the probability of the event that the data can reach the maxima closely enough.  \Cref{lem-event-M} is identical to Lemma~13 in \cite{wang2018univariate}, controlling the random intervals selected in \Cref{algorithm:WBS}.

\begin{lemma}\label{lem-bousquet-2.1}
	Let $\mathcal{D}$ be the $\sigma$-field generated by $\{X(i)\}_{i=1}^T$, $\mathcal{D}^t_T$ be the $\sigma$-field generated by $\{X(i)\}_{i=1}^T \setminus \{X(t)\}$ and $\mathbb{E}^t_T(\cdot)$ be the conditional expectation given $\mathcal{D}_T^t$, for all $t \in \{1, \ldots, T\}$.  Let $(Z, Z_1', \ldots, Z_T')$ be a sequence of $\mathcal{D}$-measurable random variables, and  $\{Z_k\}_{k=1}^T$ be a sequence of random variables such that $Z_k$ measurable with respect to $\mathcal{D}_T^k$, for all $k$.  Assume that there exists $u > 0$ such that for all $k = 1, \ldots, T$, the following inequalities hold
		\begin{equation}\label{eq-lem2-con-1}
			Z_k' \leq Z - Z_k \, \mbox{a.s.}, \quad \mathbb{E}_T^k(Z_k') \geq 0 \quad \mbox{and} \quad Z_k' \leq u \quad \mbox{a.s.}.
		\end{equation}
		Let $\sigma$ be a real value satisfying $\sigma^2 \geq \sum_{k = 1}^T \mathbb{E}^k_T\{(Z_k')^2\}$ almost surely and let $\nu = (1 + u)\mathbb{E}(Z) + \sigma^2$.  If 
		\begin{equation}\label{eq-lem2-con-2}
			\sum_{k = 1}^T (Z - Z_k) \leq Z \quad \mbox{a.s.},
		\end{equation}
		then for all $x > 0$,
		\[
			\mathbb{P}\bigl\{Z \geq \mathbb{E}(Z) + \sqrt{2\nu x} + x/3 \bigr\} \leq e^{-x}.
		\]
\end{lemma}

\begin{lemma}\label{lem-bousquet-2.3}
	Assume that $\{X(i)\}_{i = 1}^T$ satisfy \Cref{assump-model}.  Let $\mathcal{F}$ be a class of functions from $\mathbb{R}^p$ to $\mathbb{R}$ that is separable in $L_{\infty}(\mathbb{R}^p)$.  Suppose all functions $g \in \mathcal{F}$ are measurable with respect to $P_{\eta_k}$, $k \in \{1, \ldots, K + 1\}$, and there exist $B, \sigma > 0$ such that for all $g \in \mathcal{F}$
		\[
			\mathbb{E}_{P_{\eta_k}} \{g^2\} - (\mathbb{E}_{P_{\eta_k}} \{g\})^2 \leq \sigma^2 \quad \mbox{and} \quad \|g\|_{\infty} \leq B.
		\]
		Let $Z = \sup_{g \in \mathcal{F}} \big|\sum_{i = 1}^T w_i [g(X(i))- \mathbb{E}_{P_i}\{g(X(i))\}]\big|$, with $\sum_{i = 1}^T w_i^2 = 1$ and $\max_{i = 1, \ldots, T} |w_i| = w$.  Then for any $\varepsilon > 0$	, we have
		\[
			\mathbb{P}\left\{Z \geq \mathbb{E}(Z) + \sqrt{2\{(1+ wB)\mathbb{E}(Z) + \sigma^2\}x} + x/3\right\} \leq e^{-x}.
		\]
\end{lemma}

\begin{proof}
	For all $k \in \{1, \ldots, T\}$, define
		\[
			Z_k = \sup_{g \in \mathcal{F}} \left|\sum_{i \neq k} w_i [g(X(i))- \mathbb{E}_{P_i}\{g(X(i))\}]\right|
		\]	 
		and
		\[
			Z_k' = \left|\sum_{i = 1}^T w_i [g_k(X(i))- \mathbb{E}_{P_i}\{g_k(X(i))\}] \right| - Z_k,
		\]
		where $g_k$ denotes the function for which the supremum is obtained in $Z_k$.  We then have
		\begin{align*}
			Z_k' \leq Z - Z_k & \leq \left|\sum_{i = 1}^T w_i [g_0(X(i)) - \mathbb{E}_{P_i}\{g_0(X(i))\}]\right| - \left|\sum_{i \neq k} w_i [g_0(X(i)) - \mathbb{E}_{P_i}\{g_0(X(i))\}]\right| \\
			& \leq |w_k [g_0(X(k)) - \mathbb{E}_{P_k}\{g_0(X(k))\}]| \leq wB \quad \mbox{a.s.},
		\end{align*}
		where $g_0$ is the function for which the supremum is obtained in $Z$.  Moreover, we have
		\[
			\mathbb{E}^k_T(Z_k') \geq \left|\sum_{i = 1}^T\mathbb{E}_T^k \{w_i(  g_k(X(i))  - \mathbb{E}_{P_i}\{g_k(X(i))\}   )\}\right| - Z_k = 0,
		\]
		which concludes the proof of \eqref{eq-lem2-con-1} with $u = B$.  In addition,
		\begin{align*}
			(T-1)Z & = \left|\sum_{k = 1}^T \sum_{i \neq k} w_i [g_0(X(i)) - \mathbb{E}_{P_i}\{g_k(X(i))\}]\right| \\
			& \leq \sum_{k = 1}^T \left|\sum_{i \neq k} w_i [g_0(X(i)) - \mathbb{E}_{P_i}\{g_k(X(i))\}]\right| \leq \sum_{k = 1}^T Z_k,
		\end{align*}
		which leads to \eqref{eq-lem2-con-2}.  Finally, since
		\[
			\sum_{k = 1}^T \mathbb{E}_T^k \bigl\{(Z_k')^2\bigr\} \leq \sum_{k = 1}^T \mathrm{Var}_T^k \bigl\{w_k g_k(X(k))\bigr\} \leq \max_{k} \sup_{g} \mathrm{Var}\{g(X(k))\} \leq \sigma^2,
		\]
		it follows due to \Cref{lem-bousquet-2.1} that
		\[
			\mathbb{P}\left\{Z \geq \mathbb{E}(Z) + \sqrt{2\{(1+ wB)\mathbb{E}(Z) + \sigma^2\}x} + x/3\right\} \leq e^{-x},
		\]
		for all $x > 0$.
\end{proof}

\begin{lemma}\label{lem-gine}
	Let $\mathcal{F}$ be a uniformly bounded VC class of functions, and measurable with respect to all $P_{\eta_k}$, $k = 1, \ldots, K+1$.  Suppose
		\[
			\sup_{g \in \mathcal{F}}\mathrm{Var}_{P_{\eta_k}}(g) \leq \sigma^2, \quad \sup_{g \in \mathcal{F}} \|g\|_{\infty} \leq B, \quad \mbox{and} \quad 0 < \sigma \leq B.  
		\]
		Then there exist positive constants $A$ and $\nu$ depending on $\mathcal{F}$ but not on $\{P_{\eta_k}\}_{k = 1}^{K+1}$ or $T$, such that for all $T \in \mathbb{N}$,
		\[
			\sup_{g \in \mathcal{F}} \mathbb{E}\left\|\sum_{i = 1}^T w_i \{g(X_i) - \mathbb{E}(g(X_i))\}\right\| \leq C\left\{\nu wB \log(2AwB/\sigma) + \sqrt{\nu} \sigma \sqrt{\log(2AwB/\sigma)}\right\},
		\]
		where $C$ is a universal constant, $\sum_{i=1}^T w_i^2 = 1$ and $\max_{i = 1, \ldots, T} |w_i| = w$.
\end{lemma}

The proof of \Cref{lem-gine} is almost identical to that of Propostion~2.1 in \cite{gine2001consistency}, except noticing that $\sum_{i=1}^T w_i^2 = 1$. 

For any $x \in \mathbb{R}^p$, $0 \leq s < t < e \leq T$ and $h > 0$, define
	\begin{equation}\label{eq-ftilde}
		\widetilde{f}^{s, e}_{t, h}(x) = \sqrt{\frac{e - t}{(e-s)(t-s)}} \sum_{j = s + 1}^t f_{j, h}(x) - \sqrt{\frac{t - s}{(e-s)(e - t)}} \sum_{j = t + 1}^e f_{j, h}(x),
	\end{equation}
	where
	\[
		f_{j, h}(x) = h^{-p} \mathbb{E} \left\{\mathpzc{k}\left(\frac{x - X(j)}{h}\right) \right\}
	\]
	and the expectation is taken with respect to the distribution $P_j$.

\begin{lemma} \label{lem-1}
Define the events
	\[
 		\mathcal{A}_1(\gamma, h) = \left\{\max_{0 \leq s < t - h^{-p} < t + h^{-p} < e \leq T}\, \sup_{z \in \mathbb{R}^p}\,\left|\widetilde{Y}^{s, e}_{t}(z)  -\widetilde{f}^{s, e}_{t, h}(z) \right| \leq \gamma\right\}
	\]
	and
	\begin{align*}
		\mathcal{A}_2(\gamma, h) = \Bigg\{\max_{0 \leq s < t - h^{-p} < t + h^{-p} < e \leq T}\, \sup_{z \in \mathbb{R}^p} \frac{1}{\sqrt{e-s}}\Bigg|\sum_{j = s + 1}^e \left(\hat{f}_{ j, h}(z) - f_{j, h}(z)\right) \Bigg| \leq \gamma\Bigg\}.
	\end{align*}

	Under Assumptions~\ref{assump-model} and \ref{assump-kernel}, we have  that
	\[
		\mathbb{P}\left\{\mathcal{A}_1\left(C h^{-p/2}\sqrt{\log(T)}, h \right)\right\} \geq 1 - T^{-c}
	\]
	and
	\[
		\mathbb{P}\left\{\mathcal{A}_2\left(C h^{-p/2} \sqrt{\log(T)}, h \right)\right\} \geq 1 - T^{-c},
	\]
	where $C, c > 0$ are absolute constants depending on $\|\mathpzc{k}\|_{\infty}$, $A$ and $\nu$.
\end{lemma}

We remark that the proof here is an adaptation of Theorem~12 in \cite{kim2018uniform}.

\begin{proof}
For any fixed $x \in \mathbb{R}^p$, it holds that
	\begin{align}\label{eq-y-expand}
		\widetilde{Y}^{s, e}_t(x) - 	\widetilde{f}^{s, e}_{t, h}(x) = \sum_{j = s + 1}^e w_j \left[h^{-p} \mathpzc{k}\left(\frac{x - X(j)}{h}\right) - \mathbb{E}\left\{h^{-p} \mathpzc{k}\left(\frac{x - X(j)}{h}\right) \right\} \right],
	\end{align}
	where 
	\[
		w_j = \begin{cases}
 			\sqrt{\frac{e - t}{(e - s)(t - s)}}, & j = s + 1, \ldots, t, \\
 			- \sqrt{\frac{t - s}{(e - s)(e - t)}}, & j = t + 1, \ldots, e,
 		\end{cases}
	\]
	satisfying that 
	\[%begin{equation}\label{eq-w-l2-1}
		\sum_{j = s + 1}^e w_j^2 = 1 \quad \mbox{and} \quad \max_{j = s + 1, \ldots, e}|w_j| \leq  h^{p/2}.		
	\]%end{equation}

\vskip 3mm
\noindent \textbf{Step 1.} Let $\mathpzc{K}_{x, h}: \, \mathbb{R}^p \to \mathbb{R}$ be $\mathpzc{K}_{x, h}(\cdot) = \mathpzc{k}(h^{-1}x - h^{-1}\cdot)$ and 
	\[
		\widetilde{\mathcal{F}}_{\mathpzc{k}, h} = \{h^{-p}\mathpzc{K}_{x, h}: \, x \in \mathcal{X}\} 
	\]	
	be a class of normalized kernel functions centred on $\mathcal{X}$ and bandwidth $h$.  It follows from \eqref{eq-y-expand} that, for each $s, t, e$, 
	\[
		\sup_{x \in \mathcal{X}}\bigl|\widetilde{Y}^{s, e}_t(x) - \widetilde{f}^{s, e}_{t, h}(x)\bigr| = \sup_{g \in \widetilde{\mathcal{F}}_{\mathpzc{k}, h}} \left|\sum_{j = s+1}^e w_j \bigl[g(X(j)) - \mathbb{E}\{g(X(j))\}\bigr] \right| = W_{s, t, e}.
	\]
	It is immediate to check that for any $g \in \widetilde{\mathcal{F}}_{\mathpzc{k}, h}$,
	\[
		\|g\|_{\infty} \leq h^{-p} \|\mathpzc{k}\|_{\infty}.
	\]
	
	Due to the arguments used in Theorem~12 in \cite{kim2018uniform} and \Cref{assump-kernel} (i), for every probability measure $Q$ on $\mathbb{R}^p$ and for every $\zeta \in (0, h^{-p}\|\mathpzc{k}\|_{\infty})$, the covering number $\mathcal{N}(\widetilde{F}_{\mathpzc{k}, h}, L_2(Q), \zeta)$ is upper bounded as
	\[
		\sup_{Q} \mathcal{N}(\widetilde{F}_{\mathpzc{k}, h}, L_2(Q), \zeta) \leq \left(\frac{2Ap \|\mathpzc{k}\|_{\infty}}{h^p \zeta}\right)^{\nu + 2}.
	\]
	Under \Cref{assump-kernel}, due to Lemma~11 in \cite{kim2018uniform}, it holds that for any $j = 1, \ldots, T$, 
	\[
		\mathbb{E}\left\{\left(h^{-p}\mathpzc{K}_{x, h}(X(j))\right)^2\right\} \leq C_1 h^{-p},
	\]
	where $C_1$ is an absolute constant.% depending on the kernel function and underlying distributions.
	
It follows from \Cref{lem-bousquet-2.3} that for any $x > 0$,
	\begin{equation}\label{eq-large-prob-pf-1}
		\mathbb{P}\left\{W_{s, t, e} < \mathbb{E}(W_{s, t, e}) + \sqrt{2\{(1 + h^{-p/2} \|\mathpzc{k}\|_{\infty})\mathbb{E}(W_{s, t, e}) + C_1 h^{-p}\}x } + x/3  \right\} \geq 1 - e^{-x}.
	\end{equation}
	
\vskip 3mm
\noindent \textbf{Step 2.}	We then need to bound $\mathbb{E}(W_{s, t, e})$, where the expectation is taken on the product of $P_1 \otimes \ldots \otimes P_T$.  Let $\check{\mathcal{F}} = \{g - a:\, g \in \widetilde{\mathcal{F}}_{\mathpzc{k}, h}, a\in [-h^{-p}\|\mathpzc{k}\|_{\infty}, h^{-p}\|\mathpzc{k}\|_{\infty}]\}$.  Then for any $a \in  [-h^{-p}\|\mathpzc{k}\|_{\infty}, h^{-p}\|\mathpzc{k}\|_{\infty}]$, it follows from the proof of Theorem~30 in \cite{kim2018uniform}  that
	\[
		\sup_{P} \mathcal{N}(\check{\mathcal{F}}, L_2(P), a) \leq (2Ah^{-p}\|\mathpzc{k}\|_{\infty}/a)^{\nu + 1}.
	\]
	Applying \Cref{lem-gine}, we have
	\begin{equation}\label{eq-large-prob-pf-2}
		\mathbb{E}(W_{s, t, e}) \leq C\left\{(\nu+1)\frac{\|\mathpzc{k}\|_{\infty}}{h^{p/2}} \log\left(\frac{8Ah^{-p/2} \|\mathpzc{k}\|_{\infty}}{C_1^{1/2}h^{-p/2}}\right) + h^{-p/2} \sqrt{(\nu+1)\log\left(\frac{8Ah^{-p/2} \|\mathpzc{k}\|_{\infty}}{C_1^{1/2}h^{-p/2}}\right)}\right\}.
	\end{equation}

\vskip 3mm
\noindent \textbf{Step 3.} We now plug \eqref{eq-large-prob-pf-2} into \eqref{eq-large-prob-pf-1} and take $x = \log(T^m)$, with $m > 4$, resulting in 
	\[
		\mathbb{P}\left\{W_{s, t, e} < C_2 h^{-p/2}\log^{1/2}(T)\right\} \geq 1 - C_3 T^{-m},
	\]
	where $C_2, C_3 > 0$ are absolute constants depending on $\|\mathpzc{k}\|_{\infty}$, $A$ and $\nu$.  The final claims follow with a union bound argument over $s, t, e$.
	
\end{proof}

\begin{lemma}\label{lem-4}
	Under Assumptions~\ref{assump-model}, \ref{assump-kernel} and \ref{assump-rates}, for $s < t < e$, define	
	\[
	z^*_{s, e, t} \in \argmax_{z \in \mathbb{R}^p} \left|\widetilde{f}^{s, e}_{t}(z)\right|.
	\]
	With $h = c_h \kappa$, define the event
	\[
	\mathcal{B}(\gamma) = \left\{\max_{\stackrel{0 \leq s < t < e \leq T}{e-s \leq C_R \Delta}} \left|\max_{j = 1, \ldots, T} \left|\widetilde{f}^{s, e}_{t, h}(X(j))\right| - \left|\widetilde{f}^{s, e}_{t, h}(z^*_{s, e, t})\right|\right| \leq \gamma, \, (s, e) \mbox{satisfies Condition } \mathpzc{SE} \right\},
	\]
	where Condition $\mathpzc{SE}$ is defined as follows: the interval $(s, e)$ is such that either
	\begin{itemize}
	\item [(a)]	there is no true change point in $(s, e)$; or
	\item [(b)] there exists at least one true change point in $\eta_k \in (s, e)$ satisfying 
		\[
			\min\left\{\min_{\eta_k > s} \{\eta_k - s\}, \, \min_{\eta_k < e} \{e - \eta_k\}\right\} > c_1 \Delta, 
		\]
		for some $c_1 > 0$;
	\item [(c)] there exists one and only one change point $\eta_k \in (s, e)$ satisfying 
		\[
			\min\{\eta_k - s, \, e - \eta_k\} \leq C_{\epsilon} \log(T) V_p^2 \kappa^{-p} \kappa_k^{-2};
		\]
		or
	\item [(d)]	there exist exactly two change points $\eta_k, \eta_{k+1} \in (s, e)$ with $\eta_k < \eta_{k+1}$ satisfying 
		\[
			\eta_k - s \leq C_{\epsilon} \log(T) V_p^2  \kappa^{-p} \kappa_k^{-2}, \quad \mbox{and} \quad  e - \eta_{k+1} \leq C_{\epsilon} \log(T) V_p^2  \kappa^{-p} \kappa_k^{-2}.
		\]
	\end{itemize}
	Then for 
	\begin{equation}\label{eq-gamma-b}
		\gamma = C_{\gamma} h\sqrt{\Delta},
	\end{equation}
	with 
	\begin{equation}\label{eq-gamma-b-c}
		C_{\gamma} > 2C_{\mathrm{Lip}}\sqrt{C_R},
	\end{equation}
	it holds that
	\[
		\mathbb{P}\left\{\mathcal{B}(\gamma) \right\} \geq 1 - T^3 \exp\left\{- \frac{\Delta}{8} \left(\frac{c\gamma}{4 \sqrt{C_R \Delta}C_{\mathrm{Lip}}}\right)^{p+1} V_p \right\},
	\]
	for some constant $c > 0$.
\end{lemma}

\begin{proof}
Fix $0 \leq s < t < e \leq T$ with $e - s \leq C_R \Delta$.  

For case (a), it holds that $\widetilde{f}^{s, e}_{t, h}(x) = 0$, for all $x \in \mathbb{R}^p$, and the claim holds consequently.

For case (b), if $\left|\widetilde{f}^{s, e}_{t, h}(z^*_{s, e, t})\right| < \gamma$, then by the definition of $z^*_{s, e, t}$, we have that
	\[
		\left|\max_{j = 1, \ldots, T} \left|\widetilde{f}^{s, e}_{t, h}(X(j))\right| - \left|\widetilde{f}^{s, e}_{t, h}(z^*_{s, e, t})\right|\right| = \left|\widetilde{f}^{s, e}_{t, h}(z^*_{s, e, t})\right| - \max_{j = 1, \ldots, T} \left|\widetilde{f}^{s, e}_{t, h}(X(j))\right| \leq  \left|\widetilde{f}^{s, e}_{t, h}(z^*_{s, e, t})\right| < \gamma,
	\]
	which implies that 
%	\[
 \begin{equation}
 	\label{eqn:first_case}
 		\mathbb{P}\left\{\left|\underset{j = 1,\ldots,T}{\max} \left|\widetilde{f}^{s, e}_{t, h}(X(j))\right| - \left|\widetilde{f}^{s, e}_{t, h}(z^*_{s, e, t})\right|\right| > \gamma \right\}  =0.
 \end{equation}	
	%	\mathbb{P}\{\mathcal{B}(\gamma)\} = 1.
%	\]

If $\left|\widetilde{f}^{s, e}_{t, h}(z^*_{s, e, t})\right| > \gamma$, then 
	\begin{equation}\label{eq-gamma-max-f}
		\gamma < \left|\widetilde{f}^{s, e}_{t, h}(z^*_{s, e, t})\right| \leq 2 \min\bigl\{\sqrt{t-s},\, \sqrt{ e-t } \bigr\}\underset{j=1, \ldots, T}{\max} \vert  f_{j, h}(z^*_{s, e, t}) \vert,
	\end{equation}
	there exists $j_0 \in \{1, \ldots, K + 1\}$ such that 
	\begin{align}
		f_{\eta_{j_0}}(z^*_{s, e, t}) & \geq f_{\eta_{j_0}, h}(z^*_{s, e, t}) - C_{\mathrm{Lip}} h \geq \frac{\gamma}{2 \min\bigl\{\sqrt{t-s},\, \sqrt{ e-t } \bigr\}} -  C_{\mathrm{Lip}} h \nonumber \\
		& \geq \frac{c\gamma}{2 \min\bigl\{\sqrt{t-s},\, \sqrt{ e-t } \bigr\}}, \label{eq-f-eta-z-star}
	\end{align}	
	where $0 < c < 1$ is an absolute constant, the first inequality follows from \eqref{eq-bias}, the second inequality follows from \eqref{eq-gamma-max-f}, and the last inequality follows from \Cref{assump-rates} and the choice of $\gamma$.

	As for the function	$\widetilde{f}^{s, e}_{t, h}(\cdot)$, for any $x_1, x_2 \in \mathbb{R}^p$, it holds that
	\begin{align}
		\left|\widetilde{f}^{s, e}_{t, h}(x_1) - \widetilde{f}^{s, e}_{t, h}(x_2)\right| & = \Bigg|\sqrt{\frac{e-t}{(e-s)(t-s)}}\sum_{j=s+1}^t \int_{\mathbb{R}^p} \mathpzc{k}(y) \left\{f_j(x_1 - h y) - f_j(x_2 - h y)\right\}\, dy \nonumber \\
		& \hspace{-1cm} - \sqrt{\frac{t-s}{(e-s)(e-t)}}\sum_{j=t+1}^e \int_{\mathbb{R}^p} \mathpzc{k}(y) \left\{f_j(x_1 - h y) - f_j(x_2 - h y)\right\}\, dy \Bigg| \nonumber\\
		& \leq 2 \min\{\sqrt{e-t}, \, \sqrt{t-s}\} C_{\mathrm{Lip}} \|x_1 - x_2\|, \label{eq-x1-x-2-f-f}
	\end{align}
	where the last inequality follows from Assumption~\ref{assump-model}.  As a result, the function $\widetilde{f}^{s, e}_{t, h}(\cdot)$ is Lipschitz with constant $2 \min\{\sqrt{e-t}, \, \sqrt{t-s}\} C_{\mathrm{Lip}}$.  Furthermore,  defining 
	\[
		d_{j_0} = \left \vert \left\{j \in \{\eta_{j_0 - 1} + 1, \ldots, \eta_{j_0}\}:\, \|X(j) -  z_{s,e,t}^* \|  \,\leq \,  \frac{\gamma}{2 \min\{ \sqrt{  t-s},\sqrt{ e-t } \}  C_{\mathrm{Lip}}  }      \right\} \right \vert, 
	\]
	and noticing that 
	\[
		d_{j_0} \sim \text{Binomial}\left( \eta_{j_0 +1} -  \eta_{j_0},  \int_{ B(z_{s,e,t}^*,\frac{\gamma}{2 \min\{ \sqrt{  t-s},\sqrt{ e-t } \}  C_{\mathrm{Lip}}  }) }  f_{ \eta_{j_0} }(z)\,dz  \right),
	\]
	we arrive at
	\begin{align}
		& \mathbb{P}\left\{\left|\underset{j = 1,\ldots,T}{\max} \left|\widetilde{f}^{s, e}_{t, h}(X(j))\right| - \left|\widetilde{f}^{s, e}_{t, h}(z^*_{s, e, t})\right|\right| > \gamma \right\} = \mathbb{P}\left\{\underset{j = 1, \ldots, T}{\min }\left|\widetilde{f}^{s, e}_{t, h}(X(j)) - \widetilde{f}^{s, e}_{t}(z^*_{s, e, t})\right| > \gamma \right\} \nonumber \\
		\leq & \mathbb{P}\left\{\underset{j = 1, \ldots, T}{\min }\left\|X(j) - z^*_{s, e, t}\right\| > \frac{\gamma}{2 \min\{\sqrt{e-t}, \, \sqrt{t-s}\}  C_{\mathrm{Lip}}}\right\} \leq \mathbb{P}\{d_{j_0} = 0\},\label{eq-gamma-dj0}
	\end{align}
	where the identity follows from the definition of $z_{s, e, t}^*$, the first inequality follows from \eqref{eq-x1-x-2-f-f} and the second inequality follows from the definition of $d_{j_0}$.

	In addition, we have that
	\begin{align}
		& \int_{B\left(z_{s,e,t}^*, \, \frac{\gamma}{2 \min\{\sqrt{t-s},\, \sqrt{e-t}\}C_{\mathrm{Lip}}}\right)}  f_{\eta_{j_0}}(z)\, dz \geq \int_{B\left(z_{s,e,t}^*, \, \frac{c\gamma}{4 \min\{\sqrt{t-s},\, \sqrt{e-t}\}C_{\mathrm{Lip}}}\right)}  f_{\eta_{j_0}}(z)\, dz \nonumber \\
		\geq & \int_{B\left(z_{s,e,t}^*, \, \frac{c\gamma}{4 \min\{\sqrt{t-s},\, \sqrt{e-t}\}C_{\mathrm{Lip}}}\right)}  \left\{f_{\eta_{j_0}}(z^*_{s, e, t}) - \frac{c\gamma}{4 \min\{\sqrt{t-s},\, \sqrt{e-t}\}C_{\mathrm{Lip}}}\right\}\, dz \nonumber \\
		\geq & \left(\frac{c\gamma}{4 \min\{\sqrt{t-s},\, \sqrt{e-t}\}C_{\mathrm{Lip}}}\right)^{p+1} V_p,\label{eq-int-volume}
	\end{align}
	where the last inequality is due to \eqref{eq-f-eta-z-star}.  Therefore, 
	\begin{align}
		\mathbb{P}\{d_{j_0} = 0\} & \leq \mathbb{P}\left\{d_{j_0} \leq  \frac{\Delta}{2} \left(\frac{c\gamma}{4 \min\{\sqrt{t-s},\, \sqrt{e-t}\}C_{\mathrm{Lip}}}\right)^{p+1} V_p\right\} \nonumber \\
		& \leq \mathbb{P}\left\{d_{j_0} \leq \frac{(\eta_{j_0} -  \eta_{j_0 - 1})}{2} \int_{B\left(z_{s,e,t}^*, \, \frac{\gamma}{2 \min\{\sqrt{t-s},\, \sqrt{e-t}\}C_{\mathrm{Lip}}}\right)}  f_{\eta_{j_0}}(z)\, dz \right\} \nonumber \\
		& \leq \exp\left\{- \frac{(\eta_{j_0} -  \eta_{j_0 - 1})}{8} \int_{B\left(z_{s,e,t}^*, \, \frac{\gamma}{2 \min\{\sqrt{t-s},\, \sqrt{e-t}\}C_{\mathrm{Lip}}}\right)}  f_{\eta_{j_0}}(z)\, dz \right\} \nonumber \\
		& \leq \exp\left\{- \frac{\Delta}{8} \left(\frac{c\gamma}{4 \sqrt{C_R \Delta}C_{\mathrm{Lip}}}\right)^{p+1} V_p \right\}, \label{eq-djo-upper}
	\end{align}
	where the second and the fourth inequality follow from \eqref{eq-int-volume}, and the third by the Chernoff bound \citep[e.g.][]{mitzenmacher2017probability}.  Combining   \eqref{eqn:first_case}, \eqref{eq-gamma-dj0} and \eqref{eq-djo-upper} results in 
	\[
		\mathbb{P}\left\{\left|\underset{j = 1,\ldots,T}{\max} \left|\widetilde{f}^{s, e}_{t, h}(X(j))\right| - \left|\widetilde{f}^{s, e}_{t, h}(z^*_{s, e, t})\right|\right| > \gamma \right\} \leq \exp\left\{- \frac{\Delta}{8} \left(\frac{c\gamma}{4 \sqrt{C_R \Delta}C_{\mathrm{Lip}}}\right)^{p+1} V_p \right\}.
    \]
    The conclusion follows from a union bound.
    
Cases (c) and (d) are similar, and we only deal with case (c) here.  Note that
	\begin{equation}\label{eq-lem8-1111}
		\left|\widetilde{f}^{s, e}_{t}(z^*_{s, e, t})\right| \leq \kappa_k \sqrt{C_{\epsilon} \log(T) V_p^2 \kappa_k^{-2} \kappa^{-p}} \leq \gamma,
	\end{equation}
	where the first inequality follows from \Cref{lem-5} (i) and the second follows from \Cref{assump-rates}.  The final claim holds due to the fact that $\widetilde{f}^{s, e}_{t, h}$ is a smoothed version of $\widetilde{f}^{s, e}_t$.
\end{proof}

We independently select at random from $\{1, \ldots, T\}$ two sequences $\{\alpha_m\}_{m=1}^{M_1},\{\beta_m\}_{m=1}^{M_1}$, then we keep the pairs which satisfy $\beta_m - \alpha_m \leq C_R\Delta$, with $C_R \geq 3/2$.  For notational simplicity, we label them as $\{\alpha_m\}_{m=1}^R,\{\beta_m\}_{m=1}^R$.  Let 
	\begin{equation}\label{event-M}
		\mathcal{M} = \bigcap_{k = 1}^K \bigl\{\alpha_m \in \mathcal{S}_k, \beta_m \in \mathcal{E}_k, \, \mbox{for some }m \in \{1, \ldots, R\}\bigr\}, 
	\end{equation}
	where $\mathcal S_{k}= [\eta_k-3\Delta/4, \eta_k-\Delta/2 ]$ and $\mathcal
	E_{k}= [\eta_k+\Delta/2, \eta_k+3\Delta/4 ]$, $k = 1, \ldots, K$.  In the lemma below, we give a lower bound on the probability of $\mathcal{M}$. 

\begin{lemma}\label{lem-event-M}
	For the event $\mathcal{M}$ defined in \eqref{event-M}, we have
	\[
		%\begin{equation}\label{eq:event.E}
			\mathbb{P}(\mathcal M) \geq 1 -\exp\left\{\log\left(\frac{T}{\Delta}\right) - \frac{R\Delta}{4C_R T} \right\}.
		%\end{equation}
		\]
\end{lemma}
	
See Lemma~S.24 in \cite{wang2018optimal} for the proof of \Cref{lem-event-M}.

\section{Change point detection lemmas and the proof of \Cref{thm-wbs}}\label{sec:proof.thm}

Lemma~\ref{lem-2} below provides a lower bound on the maximum of the population CUSUM statistic when there exists a true change point.  Lemma~\ref{lem-3} shows that the maxima of the population CUSUM statistic are the true change points. Lemma~\ref{lem-5} is a collection of results on the population quantities.  Lemma~\ref{lem-6} provides an initial upper bound for the localization error.  \Cref{lem-7} is the key lemma to provide the final localization rate.  The proof of \Cref{thm-wbs} is collected at the end of this section.

In the rest of this section, we will adopt the notation 
	\[
		\widetilde{f}^{s, e}_{t}(x) = \sqrt{\frac{e - t}{(e-s)(t-s)}} \sum_{j = s + 1}^t f_{j}(x) - \sqrt{\frac{t - s}{(e-s)(e - t)}} \sum_{j = t + 1}^e f_{j}(x),
	\]
	for all $0 \leq s < t < e \leq T$ and $x \in \mathbb{R}^p$.
	
\begin{lemma}\label{lem-2}
Under Assumptions ~\ref{assump-model}-\ref{assump-rates}, let $(s, e)$ be an interval such that $e - s \leq C_R \Delta$ and there exists a true change point $\eta_k \in (s, e)$ with
	\[
		\min\{\eta_k - s, \, e - \eta_k\} > c_1 \Delta,
	\]
	where $c_1 >0$ is a large enough constant, depending on all the other absolute constants. Then for any $h$ such that
	\begin{equation}\label{eq-h-cond}
		(\log(T)/\Delta)^{1/p} \leq h \leq \frac{c_1}{C_R C_{\mathrm{Lip}} C_{\mathpzc{k}}} \kappa,
	\end{equation}
	it holds that 
	\[
		\max_{s + h^{-p} < t < e - h^{-p}}\sup_{z \in \mathbb{R}^p}\left|\widetilde{f}^{s, e}_{t, h}(z)\right| \geq \frac{c_1 \kappa  \Delta}{4\sqrt{e-s}}.
	\] 
\end{lemma}

\begin{proof}
Let $z_1 \in \argmax_{z \in \mathbb{R}^p}\bigl|f_{\eta_k}(z) - f_{\eta_{k+1}}(z)\bigr|$.  Due to Assumption~\ref{assump-model}, we have that
	\[
		\bigl|f_{\eta_k}(z_1) - f_{\eta_{k+1}}(z_1)\bigr| \geq \kappa_k \geq \kappa.
	\]
	Then by the argument in Lemma~2.4 of \cite{venkatraman1992consistency}, we have that
	\begin{equation}\label{eqn:lb1}
		\max_{t \in \{\eta_k +  c_1\Delta/2, \eta_k -  c_1\Delta/2   \}} \left|\widetilde{f}^{s, e}_{t}(z_1)\right| \geq \frac{c_1 \kappa\Delta}{2\sqrt{e-s}}.
	\end{equation}
	
Next, for any $x \in \mathbb{R}^p$, $h > 0$ and $j \in \{1, \ldots, T\}$, we have
	\begin{align}\label{eq-bias}
		\bigl|f_{j}(x) - f_{j, h}(x)\bigr| & = \left|\int_{\mathbb{R}^p} \frac{1}{h^p} \mathpzc{k}(y/h) \{f_j(x-y)   - f_j(x)\}\, dy \right|  \leq \frac{C_{\mathrm{Lip}}}{h^p} \int_{\mathbb{R}^p} \mathpzc{k}(y/h) \|y\|\, dy \nonumber \\
		& \leq h C_{\mathrm{Lip}} \int_{\mathbb{R}^p} \mathpzc{k}(z) \|z\|\, dz \leq C_{\mathrm{Lip}} C_{\mathpzc{k}}h,
	\end{align}
	where the last inequality follows from \Cref{assump-kernel} (iii).  Hence, for $t \in \{\eta_k + c_1\Delta/2, \eta_k - c_1\Delta/2   \}$
	\begin{equation}\label{eqn:lb2}
		\left|\widetilde{f}^{s, e}_{t, h}(z_1) - \widetilde{f}^{s, e}_{t}(z_1) \right| \leq C_{\mathrm{Lip}} C_{\mathpzc{k}}h \sqrt{\frac{(e-t)(t-s)}{e-s}} \leq \sqrt{(e-s)} C_{\mathrm{Lip}} C_{\mathpzc{k}}h \leq \frac{c_1 \kappa \Delta}{4 \sqrt{e-s}},
	\end{equation}
	which follows from \eqref{eq-h-cond}.  Finally, the claim follows combining \eqref{eqn:lb1}  and \eqref{eqn:lb2}.
\end{proof}

\begin{lemma}\label{lem-3}
	Under Assumption~\ref{assump-model}, for any interval $(s, e) \subset (0, T)$	satisfying
	\begin{equation*}
	\eta_{k-1} \le s \le \eta_k \le \ldots \le \eta_{k+q} \le e \le \eta_{k+q+1}, \quad q\ge 0.
	\end{equation*} 
	Let
	\[
		b \in \argmax_{t = s + 1, \ldots, e} \sup_{x \in \mathbb{R}^p} \left|\widetilde{f}^{s, e}_{t, h}(x)\right|.
	\]	
	If 
	\[
		h \leq \frac{\kappa}{4C_{\mathrm{Lip}} C_{\mathpzc{k}}},
	\]	
	then $b \in \{\eta_1, \ldots, \eta_K\}$.
	
	For any fixed $z \in \mathbb{R}^p$, if $\widetilde{f}^{s, e}_{t, h}(z) > 0$ for some $t \in (s, e)$, then $\widetilde{f}^{s, e}_{t, h}(z)$ is either strictly monotonic or decreases and then increases within each of the interval $(s, \eta_k), (\eta_k, \eta_{k+1}), \ldots, (\eta_{k+q}, e)$.
\end{lemma}

\begin{proof}
	We prove by contradiction.  Assume that $b \notin \{\eta_1, \ldots, \eta_K\}$.  Let $z_1 \in \argmax_{x \in \mathbb{R}^p}\left|\widetilde{f}^{s, e}_{b, h}(x)\right|$.  Due to the definition of $b$, we have
	\[
	b \in \argmax_{t = s+1, \ldots, e} \left|\widetilde{f}^{s, e}_{t, h}(z_1)\right|.
	\]
	
	It is easy to see that the collection of the change points of $\{f_{t, h}(z_1)\}_{t = s+1}^e$ is a subset of the change points of $\{f_{t, h}\}_{t=s+1}^e$.  In addition, due to \eqref{eq-bias}, it holds that
	\[
		\min_{k = 1, \ldots, K+1}\|f_{\eta_k, h} - f_{\eta_{k-1}, h}\|_{\infty} \geq \kappa - 2C_{\mathrm{Lip}} C_{\mathpzc{k}}h \geq \kappa/2,
	\]	
	which implies that the collection of the change points of $\{f_{t, h}\}_{t = s+1}^e$ is the collection of the change points of $\{f_{t}\}_{t=s+1}^e$.

It follows from Lemma~2.2 in \cite{venkatraman1992consistency} that 
	\[
	\widetilde{f}^{s, e}_{b, h}(z_1) < \max_{j \in \{k, \ldots, k+q\}} \widetilde{f}^{s, e}_{\eta_j, h}(z_1) \leq \max_{t = s+1, \ldots, e} \sup_{x \in \mathbb{R}^p}\left|\widetilde{f}^{s, e}_{t, h}(x)\right|,
	\]
	which is a contradiction.
\end{proof}

Recall that  in Algorithm \ref{algorithm:WBS},  when searching for change points in the interval  $(s,e)$, we actually   restrict  to  values  $t \in (s + h^{-p}  , e -  h^{-p} )$.
We now show that  for intervals  satisfying condition $\mathpzc{SE} $  from Lemma \ref{lem-4},   taking the maximum of the CUSUM statistic over $(s + h^{-p}  , e -  h^{-p} )$ is equivalent to searching on $(s , e )$, when there are change points in $(s + h^{-p}  , e -  h^{-p})$.

\begin{lemma}\label{lem-arg_max}
	Suppose  that  Assumptions~\ref{assump-model} and \ref{assump-rates}  hold, and the events $\mathcal{A}_1(\gamma_{\mathcal{A} })$  and $\mathcal{B}(\gamma_{\gamma_{\mathcal{B}}})$  happens where
	  $$\gamma_{\mathcal{A}} = C h^{-p/2}\sqrt{\log(T)}, \,\,\,\,\text{and}  \,\,\,\,\gamma_{\mathcal{B}} =  C_{\gamma} h\sqrt{\Delta}$$ 
	  with $C$ as in Lemma \ref{lem-1}, and $C_{\gamma} $ as in (\ref{eq-gamma-b}).
	Let $(s, e) \subset (0, T)$ satisfy $e-s \leq  C_R \Delta$.  Assume that Condition $\mathpzc{SE} $ from Lemma \ref{lem-4} holds, and that
	\begin{equation*}
	\eta_{k-1} \le s \le \eta_k \le \ldots \le \eta_{k+q} \le e \le \eta_{k+q+1}, \quad q\ge 0.
	\end{equation*} 
	Then 
	\begin{equation}
	\label{eqn:first}
		\argmax_{t = s + h^{-p}, \ldots, e-h^{-p}} \sup_{x \in \mathbb{R}^p} \left|\widetilde{f}^{s, e}_{t, h}(x)\right|  = \argmax_{t = s + 1, \ldots, e} \sup_{x \in \mathbb{R}^p} \left|\widetilde{f}^{s, e}_{t, h}(x)\right|,
	\end{equation}
	and 
    \begin{equation}
    	\label{eqn:second}
    		\argmax_{t = s + h^{-p}, \ldots, e- h^{-p} } \, \underset{j = 1,\ldots,T}{\max} \left|\widetilde{Y}^{s, e}_{t}(X(j))\right|  = \argmax_{t = s + 1, \ldots, e}\underset{j = 1,\ldots,T}{\max}   \left|\widetilde{Y}^{s, e}_{t}(X(j))\right|.
    \end{equation}
	\end{lemma}
\begin{proof}
	Firs notice that, due to \Cref{lem-2}, there exists $\eta_k \in (s,e)$ such that 
		\[
\sup_{z \in \mathbb{R}^p}\left|\widetilde{f}^{s, e}_{\eta_k, h}(z)\right| \geq \frac{c_1 \kappa  \Delta}{4\sqrt{e-s}}.
	\]
	Furthermore,  if 
    \begin{equation}
    	\label{eqn:boundary}
    		t  \in (s,e)\backslash   (s+ \max\{ h^{-p},C_{\epsilon} \log(T) V_p^2 \kappa_k^{-2} \kappa^{-p} \},e- \max\{ h^{-p},C_{\epsilon} \log(T) V_p^2 \kappa_k^{-2} \kappa^{-p} \} ),  	
    \end{equation}
	then
	\begin{align*}
		\sup_{z \in \mathbb{R}^p} \left|\widetilde{f}^{s, e}_{t,h}(z)\right| & \leq 2 \sqrt{ \min\{e-t, \,t-s \}} \max_{t=,1\ldots,T}  \sup_{z \in \mathbb{R}^p} |f_{t,h} (z)|\\
		& \leq 2  \max\left\{ h^{-p/2},\sqrt{C_{\epsilon} \log(T) V_p^2 \kappa_k^{-2} \kappa^{-p}} \right\}\, \max{t=,1\ldots,T} \sup_{z \in \mathbb{R}^p} |f_{t,h} (z)| <\frac{c_1 \kappa  \Delta}{32\sqrt{e-s}},
	\end{align*} 
	where  the last inequality follows from Assumption \ref{assump-rates}. Therefore,   (\ref{eqn:first})  follows.
	
As for \eqref{eqn:second}, 	we notice  that
\begin{align*}
	\max_{j = 1, \ldots, T}\left|\widetilde{Y}^{s, e}_{\eta_k}(X(j))\right| \geq \sup_{z \in \mathbb{R}^p}\left|\widetilde{f}^{s, e}_{\eta_k, h}(z)\right|  - \gamma_{\mathcal{A}} - \gamma_{\mathcal{B}} \geq \frac{c_1 \kappa  \Delta}{4\sqrt{e-s}}  - \gamma_{\mathcal{A}} - \gamma_{\mathcal{B}} \geq \frac{c_1 \kappa  \Delta}{8\sqrt{e-s}}.
\end{align*}
Moreover, for  $t$ satisfying (\ref{eqn:boundary}), we have
	\begin{align*}
		\max_{j = 1, \ldots, T}\left|\widetilde{Y}^{s, e}_t(X(j))\right|  & \leq \sup{z \in \mathbb{R}^p} \left|\widetilde{f}^{s, e}_{t, h}(z)\right| + \gamma_{\mathcal{A}} + \gamma_{\mathcal{B}}\\
		& \leq 2 \sqrt{\min\{e-t, \,t-s\}} \max_{t = 1, \ldots, T}  \sup_{z \in \mathbb{R}^p} |f_{t, h}(z)| +  \gamma_{\mathcal{A}} + \gamma_{\mathcal{B}}\\
		& \leq 2 \max\left\{h^{-p/2}, \, \sqrt{C_{\epsilon} \log(T) V_p^2 \kappa_k^{-2} \kappa^{-p}} \right\}\, \max_{t = 1., \ldots, T} \sup_{z\in \mathbb{R}^p} |f_{t, h}(z)| +  \gamma_{\mathcal{A}} + \gamma_{\mathcal{B}}\\
		& < \frac{c_1 \kappa  \Delta}{16\sqrt{e-s}},
\end{align*}
and the claim follows once again  using Assumption \ref{assump-rates}.

%\max_{j = 1, \ldots, T}\left|\widetilde{Y}^{s, e}_{b}(Z(j))\right| - \gamma_{\mathcal{A}} - \gamma_{\mathcal{B}}
\end{proof}

\begin{lemma}\label{lem-5}
Under Assumptions~\ref{assump-model} and \ref{assump-kernel}, the following statements hold.
\begin{itemize}
	\item [(i)]  If $\eta_k$ is the only change point in $(s, e)$, then for any $h$,
		\begin{equation}\label{eq-lem5-i}
			\sup_{x \in \mathbb{R}^p}\left|\widetilde{f}^{s, e}_{\eta_k, h}(s)\right| \leq \kappa_k \min\left\{\sqrt{s - \eta_k}, \, \sqrt{e - \eta_k}\right\}.
		\end{equation}
		
	\item [(ii)]  Suppose $e - s \leq C_R\Delta$, where $C_R > 0$ is an absolute constant, and that
		\begin{equation}\label{eq-lem5ii}
			\eta_{k-1} \le s\le \eta_k \le \ldots\le \eta_{k+q} \le e \le \eta_{k+q+1}, \quad q\ge 0.
		\end{equation}
		Denote
		\[
			\kappa_{\max}^{s,e} =\max \left\{\sup_{x \in \mathbb{R}^p}\bigl|f_{\eta_{p}}(x) - f_{\eta_{p-1}}(x)\bigr|:\, k\le p \le k+q\right\}.
		\]
		Then for any $k-1 \le p \le k+q$, it holds that 
		\begin{equation}\label{eq-lem5ii-2}
			\sup_{x \in \mathbb{R}^p}\left|\frac{1}{e-s}\sum_{i = s+1}^e f_{i, h}(x) - f_{\eta_p, h}(x) \right| \leq C_R \kappa_{\max}^{s,e}.
		\end{equation}
	
	\item [(iii)]  Assume \eqref{eq-lem5ii} and $q \geq 1$.  If 
		\begin{equation}\label{eq-eta-k-s-c-delta}
			\eta_{k}-s \le  c_1\Delta,			
		\end{equation}
		for $c_1 > 0$, then for any $h$,
		\begin{equation}\label{eq-lem5-iii}
			\sup_{z \in \mathbb{R}^p}|\widetilde f^{s,e}_{\eta_k, h}(z)| \le \sqrt{c_1} \sup_{z \in \mathbb{R}^p}|  \widetilde f^{s,e}_{\eta_{k+1}, h}(z)| + 2 \kappa_k  \sqrt {\eta_k -s} + 4 \sqrt{\eta_k - s} C_{\mathrm{Lip}} C_{\mathpzc{k}} h,
		\end{equation}		
		where $C_{\mathpzc{k}} > 0$ is an absolute constant only depending on the kernel function.
		
	\item [(iv)] Assume \eqref{eq-lem5ii} and $q = 1$, then
		\[
			\max_{t = s + 1, \ldots, e} \sup_{z \in \mathbb{R}^p} \left|\widetilde{f}^{s, e}_{t, h} (z)\right| \leq 2\sqrt{e - \eta_k} \kappa_{k+1} + 2\sqrt{\eta_k - s} \kappa_k +  4 \sqrt{\eta_k - s} C_{\mathrm{Lip}} C_{\mathpzc{k}} h + 4 \sqrt{e - \eta_k} C_{\mathrm{Lip}} C_{\mathpzc{k}} h.
		\]   	
\end{itemize}
\end{lemma}

\begin{proof}
Note that for (i),
	\begin{align*}
		\sup_{x \in \mathbb{R}^p}\left|\widetilde{f}^{s, e}_{\eta_k, h}(x)\right| & =  \sqrt{\frac{(e - \eta_k)(\eta_k - s)}{e - s}} \sup_{x \in \mathbb{R}^p} \left| \int_{\mathbb{R}^p} \mathpzc{k}(y) \left\{f_{\eta_k}(x -  h y) - f_{\eta_{k+1}}(x -  h y)\right\}\, dy \right| \\
		&\leq \kappa_k \min\left\{\sqrt{s - \eta_k}, \, \sqrt{e - \eta_k}\right\}.
	\end{align*}
	
The claim (ii)  follows from the same arguments used in showing (i) and 	Lemmas~17 and 19 in \cite{wang2018univariate}. For the claim (iii), we define
	\[
		\widetilde{g}^{s, e}_{t, h} = \begin{cases}
 			f_{\eta_{k+1}, h}, & t = s + 1, \ldots, \eta_k, \\
 			f_{t, h}, & t = \eta_k + 1, \ldots, e.
 		\end{cases}
	\]
	Thus,
	\begin{align*}
		& \left|\widetilde{f}^{s, e}_{\eta_k, h}\right| \leq  \left|\widetilde{g}^{s, e}_{\eta_k, h}\right| + \sqrt{\frac{(e - \eta_k)(\eta_k - s)}{e - s}} (f_{\eta_{k+1}, h} - f_{\eta_k, h}) \\
		\leq & \sqrt{\frac{(\eta_k - s)(e - \eta_{k+1})}{(\eta_{k+1} - s)(e - \eta_k)}}  \left|\widetilde{g}^{s, e}_{\eta_{k+1}, h}\right| + \sqrt{\frac{(e - \eta_k)(\eta_k - s)}{e - s}} (f_{\eta_{k+1}, h} - f_{\eta_k, h}) \\
		\leq & \sqrt{c_1} \left|\widetilde{g}^{s, e}_{\eta_{k+1}, h}\right| + \sqrt{\frac{(e - \eta_k)(\eta_k - s)}{e - s}} (f_{\eta_{k+1}, h} - f_{\eta_k, h}) \\
		\leq & \sqrt{c_1} \left|\widetilde{f}^{s, e}_{\eta_{k+1}, h}\right| + 2\sqrt{\frac{(e - \eta_k)(\eta_k - s)}{e - s}} (f_{\eta_{k+1}, h} - f_{\eta_k, h}) \\
		\leq & \sqrt{c_1} \left|\widetilde{f}^{s, e}_{\eta_{k+1}, h}\right| + 2\sqrt{\eta_k - s} \kappa_k +  4 \sqrt{\eta_k - s} C_{\mathrm{Lip}} C_{\mathpzc{k}} h,
	\end{align*}
	where the first, second and fourth inequalities follow from the definition of $\widetilde{g}^{s, e}_{t, h}$, the second follows from \eqref{eq-eta-k-s-c-delta} and the last follows from \eqref{eq-bias}.
	
	As for (iv), we define
	\[
		\widetilde{q}^{s, e}_{t, h} = 
		\begin{cases}
			f_{\eta_k, h}, & t = s + 1, \ldots, \eta_k, \\
			f_t, & t = \eta_k + 1, \ldots, e.
		\end{cases}
	\]
	For any $t \geq \eta_k$, it holds that 
	\begin{align*}
		\widetilde{f}^{s, e}_{t, h} - \widetilde{q}^{s, e}_{t, h} = \sqrt{\frac{e - t}{(e - s)(t - s)}}	(\eta_k - s)(f_{\eta_k, h} - f_{\eta_{k - 1}, h}).
	\end{align*}
	Therefore, for $t \geq \eta_k$, 
	\begin{align*}
		\max_{t = s + 1, \ldots, e}|\widetilde{f}^{s, e}_{t, h}| & \leq \max\{|\widetilde{f}^{s, e}_{\eta_k, h}|, \, |\widetilde{f}^{s, e}_{\eta_{k+1}, h}|\}	 \leq \max_{t = s + 1, \ldots, e} |\widetilde{q}^{s, e}_{t, h}| + 2\sqrt{\eta_k - s} \kappa_k +  4 \sqrt{\eta_k - s} C_{\mathrm{Lip}} C_{\mathpzc{k}} h\\
		& \leq 2\sqrt{e - \eta_k} \kappa_{k+1} + 2\sqrt{\eta_k - s} \kappa_k +  4 \sqrt{\eta_k - s} C_{\mathrm{Lip}} C_{\mathpzc{k}} h + 4 \sqrt{e - \eta_k} C_{\mathrm{Lip}} C_{\mathpzc{k}} h.
	\end{align*}

\end{proof}

\begin{lemma}\label{lem-6}
Let $z_0 \in \mathbb{R}^p$, $(s, e) \subset (0, T)$.  Suppose that there exits a true change point $\eta_k \in (s, e)$ such that 
	\begin{equation} \label{eqn:interval_lb}
		\min\{\eta_k - s, \, e - \eta_k\} \geq c_1 \Delta,
	\end{equation}
	and 
	\begin{equation} \label{eqn:lower_bound}
		\left|\widetilde{f}^{s, e}_{\eta_k, h}(z_0)\right| \geq (c_1/4) \frac{\kappa \Delta }{ \sqrt{e - s}},
	\end{equation}
	where $c_1 > 0$ is a sufficiently small constant.  In addition, assume that
	\begin{equation} \label{eqn:upper_bound} 
		\max_{t = s +1, \ldots, e} \left|\widetilde{f}^{s, e}_{t, h} (z_0)\right| -  \left|\widetilde{f}^{s, e}_{\eta_k, h} (z_0)\right| \leq c_2\Delta^{4}(e-s)^{-7/2}\kappa, %\leq 2\gamma_{\mathcal{A}} + 2\gamma_{\mathcal{B}} 
	\end{equation}
	where $c_2 > 0$ is a sufficiently small constant. 

Then for any $d \in (s, e)$ satisfying
	\begin{equation} \label{eqn:distance}
		\vert d -  \eta_k\vert \leq c_1\Delta/32,
	\end{equation} 
	it holds that
	\[
		\left|\widetilde{f}^{s, e}_{\eta_k, h} (z_0)\right| -  \left|\widetilde{f}^{s, e}_{d, h}(z_0)\right| >  c  |d - \eta_k| \Delta  \left|\widetilde{f}^{s, e}_{\eta_k, h} (z_0)\right| (e-s)^{-2},
	\]
	where $c > 0$ is a sufficiently small constant, depending on all the other absolute constants.
\end{lemma}

\begin{proof}
Without loss of generality, we assume that $d \geq \eta_k$ and $\widetilde{f}_{\eta_k, h}^{s, e}(z_0) \geq 0$.  Following the arguments in Lemma~2.6 in \cite{venkatraman1992consistency}, it suffices to consider two cases: (i) $\eta_{k+1} > e$ and (ii) $\eta_{k+1} \leq e$.

\vskip 3mm
\noindent \textbf{Case (i).}  Note that
	\[
		\widetilde{f}^{s, e}_{\eta_k, h} (z_0) =  \sqrt{\frac{(e - \eta_k)(\eta_k - s)}{e - s}}  \left\{f_{\eta_k, h}(z_0) - f_{\eta_{k+1}, h}(z_0)\right\}
	\]
	and
	\[
		\widetilde{f}^{s, e}_{d, h} (z_0) = (\eta_k - s) \sqrt{\frac{e - d}{(e-s)(d-s)}} \left\{f_{\eta_k, h}(z_0) - f_{\eta_{k+1}, h}(z_0)\right\}.
	\]
	Therefore, it follows from \eqref{eqn:interval_lb} that
	\begin{align}\label{eq-EL1}
		\widetilde{f}^{s, e}_{\eta_k, h} (z_0) - \widetilde{f}^{s, e}_{d, h} (z_0) = \left(1 - \sqrt{\frac{(e-d)(\eta_k - s)}{(d-s)(e - \eta_k)}}\right) \widetilde{f}^{s, e}_{\eta_k, h} (z_0) \geq c\Delta |d - \eta_k|(e-s)^{-2}\widetilde{f}^{s, e}_{\eta_k, h} (z_0). 
	\end{align}
	The inequality follows from the following arguments.  Let $u = \eta_k - s$, $v = e - \eta_k$ and $w = d - \eta_k$.  Then
	\begin{align*}
		& 1 - \sqrt{\frac{(e-d)(\eta_k - s)}{(d-s)(e - \eta_k)}} - c\Delta |d-\eta_k| (e-s)^2 \\
		= & 1 - \sqrt{\frac{(v-w)u}{(u+w)v}} - c\frac{\Delta w}{(u+v)^2} \\
		= & \frac{w(u+v)}{\sqrt{(u+w)v}(\sqrt{(v-w)u} + \sqrt{(u+w)v})} - c\frac{\Delta w}{(u+v)^2}.
	\end{align*}
	The numerator of the above equals
	\begin{align*}
		& w(u+v)^3 - c\Delta w(u+w)v - c\Delta w \sqrt{uv(u+w)(v-w)} \\
		\geq & 2c_1\Delta w \left\{(u+v)^2 - \frac{c(u+w)v}{2c_1} - \frac{c\sqrt{uv(u+w)(v-w)}}{2c_1}\right\} \\
		\geq & 2c_1\Delta w \left\{(1-c/(2c_1)) (u+v)^2 - 2^{-1/2}c/c_1 uv\right\} > 0,
	\end{align*}
	as long as 
	\[
		c < \frac{\sqrt{2}c_1}{4 + 1/(\sqrt{2}c_1)}.
	\]

\vskip 3mm
\noindent \textbf{Case (ii).}  Let $g = c_1\Delta/16$.  We can write
	\[
		\widetilde{f}^{s, e}_{\eta_k, h} (z_0) = a\sqrt{\frac{e-s}{(\eta_k - s)(e-\eta_k)}}, \quad \widetilde{f}^{s, e}_{\eta_k+g, h} (z_0) = \bigl(a + g\theta\bigr)\sqrt{\frac{e-s}{(e - \eta_k - g)(\eta_k + g - s)}},
	\]
	where
	\[
		 a = \sum_{j = s+1}^{\eta_k} \left\{f_{j, h}(z_0) - \frac{1}{e-s} \sum_{j = s+1}^e f_{j, h}(z_0)\right\},
	\]
	\[
		  \theta = \frac{a\sqrt{(\eta_k + g -s)(e - \eta_k - g)}}{g} \left\{ \frac{1}{\sqrt{(\eta_k - s)(e-\eta_k)}} - \frac{1}{(\eta_k + g -s)(e-\eta_k - g)} + \frac{b}{a\sqrt{e-s}}\right\},
	\]
	and $b = \widetilde{f}^{s, e}_{\eta_k+g, h} (z_0) - \widetilde{f}^{s, e}_{\eta_k, h} (z_0)$.
	
To ease notation, let $d - \eta_k = l \leq g/2$, $N_1 = \eta_k - s$ and $N_2 = e - \eta_k - g$.  We have	
	\begin{equation}\label{eq-El1}
	    E_l = \widetilde{f}^{s, e}_{\eta_k, h} (z_0) - \widetilde{f}^{s, e}_{d, h} (z_0) = E_{1l}(1 + E_{2l}) + E_{3l},
    \end{equation}
    where     
    \[
    E_{1l} = \frac{a l(g - l)  \sqrt{e-s}} {\sqrt{N_1(N_2 + g)} \sqrt{(N_1 + l)(g + N_2 - l)} \left(\sqrt{(N_1 + l) (g + N_2 - l)} +\sqrt{N_1(g + N_2)}\right)},
    \]
    \[
    E_{2l} = \frac{(N_2 - N_1)(N_2 - N_1 - l)}{\left(\sqrt{(N_1 + l)(g + N_2 - l)} + \sqrt{(N_1 + g)N_2}\right)\left(\sqrt{N_1(g + N_2)} + \sqrt{(N_1 + g)N_2}\right)},
    \]
    and
    \[
    E_{3l} = -\frac{b l}{g} \sqrt{\frac{(N_1 + g)N_2}{(N_1 + l)(g + N_2 - l)}}.
    \]  
 
 Next,  we notice that $g - l \geq c_1 \Delta/32$.  It holds that
    \begin{equation}\label{eq-El2}
		E_{1l} \geq c_{1l}|d - \eta_k| \Delta \widetilde{f}_{\eta_k, h}^{s, e}(z_0)(e-s)^{-2}, 		
	\end{equation}
	where $c_{1l} > 0$ is a sufficiently small constant depending on $c_1$.  As for $E_{2l}$, due to \eqref{eqn:distance}, we have
	\begin{align}\label{eq-El3}
		E_{2l} \geq -1/2.
	\end{align}
	As for $E_{3l}$, we have
	\begin{align}
		E_{3l} & \geq -c_{3l,1}b|d - \eta_k| (e-s)\Delta^{-2} \geq - c_{3l, 2} b|d - \eta_k|\Delta^{-3} (e-s)^{3/2} \widetilde{f}^{s, e}_{\eta_k, h} (z_0) \kappa^{-1}\nonumber \\
		& \geq -c_{1l}/2|d - \eta_k| \Delta \widetilde{f}^{s, e}_{\eta_k, h} (z_0) (e-s)^{-2}, \label{eq-El4}
	\end{align}
	where the second inequality follows from \eqref{eqn:lower_bound} and the third inequality follows from \eqref{eqn:upper_bound}, $c_{3l,1}, c_{3l,2} > 0$ are sufficiently small constants, depending on all the other absolute constants.

Combining \eqref{eq-El1}, \eqref{eq-El2}, \eqref{eq-El3} and \eqref{eq-El4}, we have
    \begin{equation}\label{eq-EL2}
		\widetilde{f}^{s, e}_{\eta_k, h} (z_0) - \widetilde{f}^{s, e}_{d, h} (z_0) \geq c  |d - \eta_k| \Delta \widetilde{f}^{s, e}_{\eta_k, h} (z_0) (e-s)^{-2},   
    \end{equation}
    where $c > 0$ is a sufficiently small constant.  
   
In view of \eqref{eq-EL1} and \eqref{eq-EL2}, the proof is complete.
	
\end{proof}

\begin{lemma}\label{lem-7}
Under Assumptions~\ref{assump-model}, \ref{assump-kernel} and \ref{assump-rates}, let $(s_0, e_0)$ be an interval with $e_0 - s_0 \leq C_R \Delta$ and containing at least one change point $\eta_l$ such that
	\[
		\eta_{l-1} \leq s_0 \leq \eta_l \leq \ldots \leq \eta_{l+q} \leq e_0 \leq \eta_{l+q+1}, \quad q \geq 0.
	\]
	Suppose that there exists $k'$ such that 
	\[
		\min \bigl\{\eta_{k'} - s_0, \, e_0 - \eta_{k'} \bigr\} \geq \Delta/16.
	\]
	Let 
	\[
		\kappa^{\max}_{s_0, e_0} = \max\bigl\{\kappa_p:\, \min\{\eta_p - s_0, \, e_0 - \eta_p\} \geq \Delta/16\bigr\}.
	\]
	Consider any generic $(s, e) \subset (s_0, e_0)$, satisfying
	\[
		\min_{l: \, \eta_l \in (s, e)}\min\{\eta_l - s_0, e_0 - \eta_l\} \geq \Delta/16.
	\]

Let 
	\[
		b \in \argmax_{t = s+h^{-p}, \ldots, e -h^{-p}  } \max_{j = 1, \ldots, T}\, \left|\widetilde{Y}^{s, e}_{t}(X(j))\right|.
	\]  
	Assume 
	\begin{equation}\label{eq-h-kappa-cond}
		h \leq \frac{\kappa}{16C_R C_{\mathrm{Lip}} C_{\mathpzc{k}}},
	\end{equation}
	where $C_{\mathpzc{k}} > 0$ is an absolute constant depending only on the kernel function.  For some $c_1 > 0$ and $\gamma > 0$, suppose that
	\begin{equation}\label{eq-lem13-2}
		\max_{j = 1, \ldots, T}\left|\widetilde{Y}^{s_0, e_0}_{b}(X(j)) \right|\geq c_1 \kappa^{\max}_{s, e}\sqrt{\Delta}.
	\end{equation}

Then on the event $\mathcal{A}_1(\gamma_{\mathcal{A}}) \cap \mathcal{A}_2(\gamma_{\mathcal{A}}) \cap \mathcal{B}(\gamma_{\mathcal{B}})$, defined in Lemmas~\ref{lem-1} and  \ref{lem-4}, where
	\begin{equation}\label{eq-lem13-3}
		\max\{\gamma_{\mathcal{A}}, \, \gamma_{\mathcal{B}}\} \leq c_2 \kappa \sqrt{\Delta},
	\end{equation}
	with a sufficiently small constant $0 < c_2 < c_1/4$, there exists a change point $\eta_k \in (s, e)$ such that
	\[
		\min\{e - \eta_k, \, \eta_k - s\} \geq \Delta/4 \quad \mbox{and} \quad |\eta_k - b| \leq C\kappa_k^{-2}\gamma^{2}_{\mathcal{A}},
	\]		
	where $C > 0$ is a sufficiently large constant depending on all the other absolute constants.
\end{lemma}

\begin{proof}
Let $z_1 \in \argmax_{z \in \mathbb{R}^p} \bigl|\widetilde{f}^{s, e}_{b, h}(z)\bigr|$.  Without loss of generality, assume that $\widetilde{f}^{s, e}_{b, h}(z_1) > 0$ and that $\widetilde{f}^{s, e}_{b, h}(z_1)$ as a function of $t$ is locally decreasing at $b$.  Observe that there has to be a change point $\eta_k \in (s, b)$, or otherwise $\widetilde{f}^{s, e}_{b, h}(z_1) > 0$ implies that $\widetilde{f}^{s, e}_{t, h}(z_1)$ is decreasing, as a consequence of Lemma~\ref{lem-3}.  

Thus, there exists a change point $\eta_k \in (s, b)$ satisfying that
	\begin{align}
		\sup_{z \in \mathbb{R}^p} \left|\widetilde{f}^{s, e}_{\eta_k, h}(z)\right| & \geq \left|\widetilde{f}^{s, e}_{\eta_k, h}(z_1)\right| > \left|\widetilde{f}^{s, e}_{b, h}(z_1)\right| \geq \max_{j = 1, \ldots, T}\left|\widetilde{f}^{s, e}_{b, h}(X(j))\right| - \gamma_{\mathcal{B}} \nonumber \\
		& \geq \max_{j = 1, \ldots, T}\left|\widetilde{Y}^{s, e}_{b}(X(j))\right| - \gamma_{\mathcal{A}} - \gamma_{\mathcal{B}} \geq c\kappa_k \sqrt{\Delta}, \label{eq-lem7-pf-1}
	\end{align}
	where the second inequality follows from Lemma~\ref{lem-3}, the third and fourth inequalities hold on the events $\mathcal{A}_1(\gamma_{\mathcal{A}}, h) \cap \mathcal{A}_2(\gamma_{\mathcal{A}}, h) \cap \mathcal{B}(\gamma_{\mathcal{B}})$, and $c > 0$ is an absolute constant.
	
Observe that $e-s\le e_0-s_0\le C_R\Delta $ and that $(s, e)$ has to contain at least one change point or otherwise $ \sup_{z \in \mathbb{R}}|\widetilde f^{s,e}_{\eta_k, h}(z)| =0$ which contradicts \eqref{eq-lem7-pf-1}.

\vskip 3mm
\noindent \textbf{Step 1.}  In this step, we are to show that 
	\begin{equation}\label{eq-lem13-pf-2}
		\min\{\eta_k - s, \, e - \eta_k\} \geq \min\{1, c_1^2\}\Delta/16.
	\end{equation}	

Suppose that $\eta_k$ is the only change point in $(s, e)$.  Then \eqref{eq-lem13-pf-2} must hold or otherwise it follows from \eqref{eq-lem5-i} that
	\[
		\sup_{z \in \mathbb{R}^p}\left|\widetilde{f}^{s, e}_{\eta_k, h}(z)\right| \leq \kappa_k \frac{c_1\sqrt{\Delta}}{4},
	\]
	which contradicts \eqref{eq-lem7-pf-1}.
	
Suppose $(s, e)$ contains at least two change points.  Then arguing by contradiction, if $\eta_k - s < \min\{1, \, c_1^2\}\Delta /16$, it must be the cast that $\eta_k$ is the left most change point in $(s, e)$.  Therefore
	\begin{align*}
		\sup_{z \in \mathbb{R}^p}\left|\widetilde{f}^{s, e}_{\eta_k, h} (z)\right| & \leq c_1/4 \sup_{z \in \mathbb{R}^p}|  \widetilde f^{s,e}_{\eta_{k+1}, h}(z)| + 2\kappa_k  \sqrt {\eta_k -s} +  4 \sqrt{\eta_k - s} C_{\mathrm{Lip}} C_{\mathpzc{k}} h\\
		& < c_1/4 \max_{s +h^{-p} < t < e-h^{-p}}\sup_{z \in \mathbb{R}^p}|  \widetilde f^{s,e}_{t, h}(z)| + \frac{\sqrt{\Delta}}{2}c_1\kappa_k \\ 
		& \leq c_1/4 \max_{s +h^{-p}< t < e-h^{-p}}\max_{j = 1, \ldots, T}|  \widetilde f^{s,e}_{t, h}(X(j))|  + c_1/4 \gamma_{\mathcal{B}} + \frac{\sqrt{\Delta}}{2}c_1\kappa_k \\
		& \leq c_1/4 \max_{s +h^{-p}< t < e-h^{-p}} \max_{j = 1, \ldots, T} \left|\widetilde{Y}^{s, e}_{t}(X(j))\right| + c_1/4 \gamma_{\mathcal{A}} + c_1/4 \gamma_{\mathcal{B}} + \frac{\sqrt{\Delta}}{2}c_1\kappa_k \\
		& \leq  \max_{j = 1, \ldots, T} \left|\widetilde{Y}^{s, e}_{b}(X(j))\right| - \gamma_{\mathcal{A}} - \gamma_{\mathcal{B}}, 
	\end{align*}
	where the first inequality follows from \eqref{eq-lem5-iii}, the second follows from \eqref{eq-h-kappa-cond}, the third from the definition of the event $\mathcal{B}$, the fourth from the definition of the event $\mathcal{A}$ and the last from \eqref{eq-lem13-2}.  The last display contradicts \eqref{eq-lem7-pf-1}, thus \eqref{eq-lem13-pf-2} must hold.

\vskip 3mm
\noindent \textbf{Step 2.}  Let
	\[
		z_0 \in \argmax_{z \in \mathbb{R}^p}\left|\widetilde{f}^{s, e}_{\eta_k, h}(z)\right|.
	\]
	It follows from Lemma~\ref{lem-6} that there exits $d \in (\eta_k, \eta_k + c_1\Delta/32)$ such that
	\begin{equation}\label{eq-lem13-pf-3}
		\widetilde{f}^{s, e}_{\eta_k, h}(z_0) - \widetilde{f}^{s, e}_{d, h}(z_0) \geq 2\gamma_{\mathcal{A}} + 2\gamma_{\mathcal{B}}.
	\end{equation}
	We claim that $b \in (\eta_k, d) \subset (\eta_k, \eta_k + c_1\Delta/16)$.  By contradiction, suppose that $b \geq d$.  Then
	\begin{align}
		\widetilde{f}^{s, e}_{b, h}(z_0) \leq \widetilde{f}^{s, e}_{d, h}(z_0) \leq \max_{s < t < e}\sup_{z \in \mathbb{R}^p}\left|\widetilde{f}^{s, e}_{t, h}(z)\right| - 2\gamma_{\mathcal{A}} -2\gamma_{\mathcal{B}} \leq \max_{j = 1, \ldots, T}\left|\widetilde{Y}^{s, e}_{b}(X(j))\right| - \gamma_{\mathcal{A}} - \gamma_{\mathcal{B}}, \label{eq-lem13-pf-4}
	\end{align}
	where the first inequality follows from Lemma~\ref{lem-3}, the second follows from \eqref{eq-lem13-pf-3} and the third follows from the definition of the event $\mathcal{A}_1(\gamma_{\mathcal{A}}, h) \cap \mathcal{A}_2(\gamma_{\mathcal{A}}, h) \cap \mathcal{B}(\gamma_{\mathcal{B}})$.   Note that \eqref{eq-lem13-pf-4} is a contradiction to the bound in \eqref{eq-lem7-pf-1}, therefore we have $b \in (\eta_k, \eta_k + c_1\Delta /32)$. 

\vskip 3mm
\noindent {\bf Step 3.}  Let 
	\[
		j^* \in \argmax_{j = 1, \ldots, T}\left|\widetilde{Y}^{s, e}_{b}(X(j))\right|, \quad f^{s, e} = \left(f_{s+1, h}(X(j^*)), \ldots, f_{e, h}(X(j^*))\right)^{\top} \in \mathbb{R}^{(e-s)}
	\]
	and
	\[
		Y^{s, e} = \left(\frac{1}{h^p}\mathpzc{k}\left(\frac{X(j^*) - X(s)}{h}\right), \ldots, \frac{1}{h^p}\mathpzc{k}\left(\frac{X(j^*) - X(e)}{h}\right)\right) \in \mathbb{R}^{(e-s)}.
	\]
	By the definition of $b$, it holds that
	\[
		\bigl\|Y^{s,e} - \mathcal{P}^{s,e}_{b}(Y^{s,e})\bigr\|^2 \leq \bigl \|Y^{s,e} - \mathcal{P}^{s,e}_{\eta_k}(Y^{s,e})\bigr\|^2  \leq \bigl\|Y^{s,e} - \mathcal{P}_{\eta_k}^{s,e}(f^{s,e})\bigr\|^2,
	\]
	where the operator $\mathcal{P}^{s, e}_{\cdot}(\cdot)$ is defined in Lemma~20 in \cite{wang2018univariate}.  For the sake of contradiction, throughout the rest of this argument suppose that, for some sufficiently large constant $C_3 > 0$ to be specified,
	\begin{align}\label{eq:wbs contradict assume}
		\eta_k + C_3\gamma^2_{\mathcal{A}}\kappa_k^{-2}< b.
	\end{align}
	We will show that this leads to the bound
	\begin{align}\label{eq:WBS sufficient}
		\bigl\|Y^{s,e} - \mathcal{P}_{b}^{s,e} (Y^{s,e})\bigr\|^2 > \bigl\|Y^{s,e} - \mathcal{P}^{s,e}_{\eta_k}(f^{s,e})\bigr\|^2,	
    \end{align}
	which is a contradiction.  If we can show that 
	\begin{equation}\label{eq-lem13-pf-6}
		2\langle Y^{s, e}  - f^{s, e},\, \mathcal{P}_b^{s, e}\bigl(Y^{s, e}\bigr) - \mathcal{P}_{\eta_k}^{s, e}\bigl(f^{s, e}\bigr)\rangle < \bigl\| f^{s, e} - \mathcal{P}_b^{s, e}\bigl(f^{s, e}\bigr)\bigr\|^2 - \bigl\|f^{s, e} - \mathcal{P}_{\eta_k}^{s, e}\bigl(f^{s, e}\bigr)\bigr\|^2,
	\end{equation}
	then \eqref{eq:WBS sufficient} holds. 

To derive \eqref{eq-lem13-pf-6} from \eqref{eq:wbs contradict assume}, we first note that  $\min\{e - \eta_k, \eta_k - s\} \geq \min\{1, c_1^2\}\Delta/16$ and that $|b - \eta_k| \leq c_1\Delta /32$ implies that 
	\[%begin{align}\label{eq:wbs size of intervals}
		\min\{e-b, b-s\} \geq \min\{1, c_1^2\}\Delta/16 - c_1\Delta /32 \geq \min\{1, c_1^2\}\Delta/32.
	\]%end{align}

As for the right-hand side of \eqref{eq-lem13-pf-6}, we have 
	\begin{align}
		& \bigl\| f^{s, e} - \mathcal{P}_b^{s, e}\bigl(f^{s, e}\bigr)\bigr\|^2 - \bigl\|f^{s, e} - \mathcal{P}_{\eta_k}^{s, e}\bigl(f^{s, e}\bigr)\bigr\|^2 = \left(\widetilde{f}^{s, e}_{\eta_k, h}(X(j^*))\right)^2 - \left(\widetilde{f}^{s, e}_{b, h}(X(j^*))\right)^2 \nonumber\\
		\geq & \left(\widetilde{f}^{s, e}_{\eta_k, h}(X(j^*)) - \widetilde{f}^{s, e}_{b, h}(X(j^*))\right) \bigl|\widetilde{f}^{s, e}_{\eta_k, h}(X(j^*)) \bigr|. \label{eq-lem13-47-1}
	\end{align}

On the event $\mathcal{A}_1(\gamma_{\mathcal{A}}, h) \cap \mathcal{A}_2(\gamma_{\mathcal{A}}, h) \cap \mathcal{B}(\gamma_{\mathcal{B}})$, we are to use Lemma~\ref{lem-6}.  Note that \eqref{eqn:lower_bound} holds due to the fact that here we have
	\begin{align}
		\left|\widetilde{f}^{s, e}_{\eta_k, h}(X(j^*))\right| \geq \bigl|\widetilde{f}^{s, e}_{b, h}(X(j^*))\bigr| \geq \left|\widetilde{Y}^{s,e}_{b}(X(j^*))\right| - \gamma_{\mathcal{A}} \geq c_1 \kappa_k\sqrt{\Delta} - \gamma_{\mathcal{A}} \geq (c_1)/2\kappa_k \sqrt{\Delta}, \label{eq-lem13-47-2}
	\end{align}
	where the first inequality follows from the fact that $\eta_k$ is a true change point, the second inequality holds due to the event $\mathcal{A}_1(\gamma_{\mathcal{A}}, h)$, the third inequality follows from \eqref{eq-lem13-2}, and the final inequality follows from  \eqref{eq-lem13-3}.   Towards this end, it follows from Lemma~\ref{lem-6} that 
	\begin{equation}\label{eq-lem13-47-3}
		\left|\widetilde{f}^{s, e}_{\eta_k, h} (X(j^*)))\right| - \left|\widetilde{f}^{s, e}_{b, h}(X(j^*)))\right| >  c |b - \eta_k| \Delta  \left|\widetilde{f}^{s, e}_{\eta_k, h} (X(j^*)))\right| (e-s)^{-2}.
	\end{equation}
	Combining \eqref{eq-lem13-47-1}, \eqref{eq-lem13-47-2} and \eqref{eq-lem13-47-3}, we have
	\begin{equation}\label{eq-lem13-47-rhs}
		\bigl\| f^{s, e} - \mathcal{P}_b^{s, e}\bigl(f^{s, e}\bigr)\bigr\|^2 - \bigl\|f^{s, e} - \mathcal{P}_{\eta_k}^{s, e}\bigl(f^{s, e}\bigr)\bigr\|^2 \geq \frac{cc_1^2}{4} \Delta^2 \kappa_k\mathcal{A}_1(\gamma_{\mathcal{A}}, h) ^2 (e-s)^{-2} |b - \eta_k|.
	\end{equation}

The left-hand side of \eqref{eq-lem13-pf-6} can be decomposed as follows.
	\begin{align}
		& 2\langle Y^{s, e}  - f^{s, e},\, \mathcal{P}_b^{s, e}\bigl(Y^{s, e}\bigr) - \mathcal{P}_{\eta_k}^{s, e}\bigl(f^{s, e}\bigr)\rangle \nonumber \\
		= & 2 \langle Y^{s, e} - f^{s, e},\, \mathcal{P}_b^{s, e}\bigl(Y^{s, e}\bigr) - \mathcal{P}_{b}^{s, e}\bigl(f^{s, e}\bigr) \rangle + 2 \langle Y^{s, e} - f^{s, e}, \,\mathcal{P}_{b}^{s, e}\bigl(f^{s, e}\bigr) -  \mathcal{P}_{\eta_k}^{s, e}\bigl(f^{s, e}\bigr)\rangle \nonumber \\
		= & (I) + 2 \left(\sum_{i = 1}^{\eta_k - s} + \sum_{i = \eta_k - s + 1}^{b-s} + \sum_{i = b-s + 1}^{e-s} \right) \bigl(Y^{s, e}  - f^{s, e}\bigr)_i \left(\mathcal{P}_{b}^{s, e}\bigl(f^{s, e}\bigr) -  \mathcal{P}_{\eta_k}^{s, e}\bigl(f^{s, e}\bigr)\right)_i \nonumber \\
		= & (I) + (II.1) + (II.2) + (II.3).\label{eq-lem13-47-lhs0}
	\end{align}
	
As for the term (I), we have
	\begin{align}\label{eq-lem13-47-lhs1}
		(I) \leq 2 \gamma^2_{\mathcal{A}}.
	\end{align}
	As for the the term (II.1), we have	
	\begin{align*}
		(II.1) = 2\sqrt{\eta_k - s} \left\{\frac{1}{\sqrt{\eta_k - s}} \sum_{i = 1}^{\eta_k - s}\bigl(Y^{s, e}  - f^{s, e}\bigr)_i \right\}	\left\{\frac{1}{b-s}\sum_{i = 1}^{b-s} (f^{s,e})_i - \frac{1}{\eta_k-s}\sum_{i=1}^{\eta_k-s} (f^{s,e})_i\right\}.
	\end{align*}
	In addition, it holds that
	\begin{align*}
		& \left|\frac{1}{b-s}\sum_{i = 1}^{b-s} (f^{s,e})_i - \frac{1}{\eta_k-s}\sum_{i=1}^{\eta_k-s} (f^{s,e})_i\right| = \frac{b - \eta_k}{b-s}\left|-\frac{1}{\eta_k - s} \sum_{i = 1}^{\eta_k - s} f_{i, h}(X(j^*)) + f_{\eta_{k + 1}, h}(X(j^*))\right| \\
		\leq & \frac{b - \eta_k}{b - s} (C_R+1)\kappa_{s_0, e_0}^{\max},
	\end{align*}
	where the inequality follows from \eqref{eq-lem5ii-2}.  Combining with Lemma~\ref{lem-1}, it leads to that
	\begin{align}
		(II.1) & \leq 2 \sqrt{\eta_k - s}\frac{b - \eta_k}{b - s} (C_R+1)\kappa_{s_0, e_0}^{\max} \gamma_{\mathcal{A}} \nonumber \\ 
		& \leq 2 \frac{4}{\min\{1, \, c_1^2\}}\Delta^{-1/2} \gamma_{\mathcal{A}} |b - \eta_k| (C_R+1) \kappa^{\max}_{s_0, e_0}. \label{eq-lem13-47-lhs21}
	\end{align}
	As for the term (II.2), it holds that
	\begin{align}\label{eq-lem13-47-lhs22}
		(II.2) \leq 2 \sqrt{|b - \eta_k|} \gamma_{\mathcal{A}} (2C_R + 3) \kappa^{\max}_{s_0, e_0}.
	\end{align}
	As for the term (II.3), it holds that 
	\begin{align}\label{eq-lem13-47-lhs23}
		(II.3) \leq 2 \frac{4}{\min\{1, \, c_1^2\}}\Delta^{-1/2} \gamma_{\mathcal{A}} |b - \eta_k| (C_R+1) \kappa^{\max}_{s_0, e_0}.
	\end{align}
	
Therefore, combining \eqref{eq-lem13-47-rhs}, \eqref{eq-lem13-47-lhs0}, \eqref{eq-lem13-47-lhs1}, \eqref{eq-lem13-47-lhs21}, \eqref{eq-lem13-47-lhs22} and \eqref{eq-lem13-47-lhs22}, we have that \eqref{eq-lem13-pf-6} holds if 
	\[
		\Delta^2 \kappa_k^2 (e-s)^{-2} |b - \eta_k| \gtrsim \max\left\{\gamma^2_{\mathcal{A}}, \, \Delta^{-1/2} \gamma_{\mathcal{A}} |b - \eta_k| \kappa_k, \, \sqrt{|b - \eta_k|} \gamma_{\mathcal{A}} \kappa_k \right\}.
	\]
	The second inequality holds due to \Cref{assump-rates}, the third inequality holds due to \eqref{eq:wbs contradict assume} and the first inequality is a consequence of the third inequality and \Cref{assump-rates}.

\end{proof}

\begin{proof}[Proof of Theorem~\ref{thm-wbs}]
Let $\epsilon_k = C_{\epsilon} \log^{1+\xi}(T) \kappa_k^{-2} \kappa^{-p} \leq \epsilon = C_{\epsilon} \log^{1+\xi}(T) \kappa^{-(p+2)} $.  Since $\epsilon$ is the upper bound of the localization error, by induction, it suffices to consider any interval $(s, e) \subset (0, T)$	 that satisfies
	\[
		\eta_{k-1} \leq s \leq \eta_k \leq \ldots \leq \eta_{k+q} \leq e \leq \eta_{k+q+1}, \quad q \geq -1,
	\]
	and
	\[
		\max\bigl\{\min\{\eta_k - s, \, s - \eta_{k-1}\}, \, \min \{\eta_{k+q+1} - e, \, e - \eta_{k+q}\} \bigr\} \leq \epsilon,
	\]
	where $q = -1$ indicates that there is no change point contained in $(s, e)$.
	
By Assumption~\ref{assump-rates}, it holds that $\epsilon \leq \Delta/4$.  It has to be the case that for any change point $\eta_k \in (0, T)$, either $|\eta_k - s| \leq \epsilon$ or $|\eta_k - s| \geq \Delta - \epsilon \geq 3\Delta/4$.  This means that $\min\{|\eta_k - s|, \, |\eta_k - e|\}\leq \epsilon$ indicates that $\eta_k$ is a detected change point in the previous induction step, even if $\eta_k \in (s, e)$.  We refer to $\eta_k \in (s, e)$ an undetected change point if $\min\{|\eta_k - s|, \, |\eta_k - e|\} \geq 3\Delta/4$.
	
In order to complete the induction step, it suffices to show that we (i) will not detect any new change point in $(s, e)$ if all the change points in that interval have been previous detected, and (ii) will find a point $b \in (s, e)$, such that $|\eta_k - b| \leq \epsilon$ if there exists at least one undetected change point in $(s, e)$.

Define
	\[
		\mathcal{S} = \bigcap_{k = 1}^K\left\{\alpha_s \in [\eta_k - 3\Delta/4, \eta_k - \Delta/2], \, \beta_s \in [\eta_k + \Delta/2, \eta_k + 3\Delta/4], \mbox{ for some } s = 1, \ldots, S\right\}.
	\] 
	The rest of the proof assumes the event $\mathcal{A}_1 (\gamma_{\mathcal{A}}) \cap \mathcal{A}_2 (\gamma_{\mathcal{A}}) \cap \mathcal{B} (\gamma_{\mathcal{B}}) \cap \mathcal{M}$, with 
	\[
		\gamma_{\mathcal{A}} = C_{\gamma_{\mathcal{A}}}h^{-p/2}\sqrt{\log(T)} \quad \mbox{and} \quad  \gamma_{\mathcal{B}} = C_{\gamma_{\mathcal{B}}}h\sqrt{\Delta},
	\]
	and $C_{\gamma_{\mathcal{A}}}, C_{\gamma_{\mathcal{A}}} > 0$ are absolute constants.  The probability of the event $\mathcal{A}_1 (\gamma_{\mathcal{A}}) \cap \mathcal{A}_2 (\gamma_{\mathcal{A}}) \cap \mathcal{B} (\gamma_{\mathcal{B}}) \cap \mathcal{M}$ is lower bounded in Lemmas~\ref{lem-1}, \ref{lem-4} and \ref{lem-event-M}. \\

\noindent \textbf{Step 1.}  In this step, we will show that we will consistently detect or reject the existence of undetected change points within $(s, e)$.  Let $a_m$, $b_m$ and $m^*$ be defined as in Algorithm~\ref{algorithm:WBS}.   Suppose there exists a change point $\eta_k \in (s, e)$ such that $\min \{\eta_k - s, \, e - \eta_k\} \geq 3\Delta/4$.  In the event $\mathcal{S}$, there exists an interval $(\alpha_m, \beta_m)$ selected such that $\alpha_m \in [\eta_k - 3\Delta/4, \eta_k - \Delta/2]$ and $\beta_m \in [\eta_k + \Delta/2, \eta_k + 3\Delta/4]$.  Following Algorithm~\ref{algorithm:WBS}, $[s_m, e_m] = [\alpha_m, \beta_m] \cap [s, e]$.  We have that $\min \{\eta_k - s_m, e_m - \eta_k\} \ge (1/4)\Delta$ and $[s_m, e_m] $ contains at most one true change point. 

It follows from Lemma~\ref{lem-2},  Lemma \ref{lem-arg_max}, and Assumption~\ref{assump-rates}, with $c_1$ there chosen to be $1/4$, that
	\[
		\max_{s_m + h^{-p} < t < e_m- h^{-p}  }\sup_{z \in \mathbb{R}^p}\left|\widetilde{f}^{s, e}_{t, h}(z)\right| \geq \frac{\kappa  \Delta}{16\sqrt{e-s}}.
	\]
	Therefore
	\begin{align*}
		 a_m & = \max_{s_m + h^{-p} < t < e_m  - h^{-p} }\max_{j = 1, \ldots, T}\left| \widetilde{Y}_{t}^{s_m, e_m}(X(j)) \right|\geq \max_{s_m +h^{-p}< t < e_m  -h^{-p} } \max_{j = 1, \ldots, T}\left| \widetilde{f}_{t, h}^{s_m   , e_m }(X(j)) \right| - \gamma_{\mathcal{A}} \\
		& \geq \max_{s_m + h^{-p}< t < e_m  - h^{-p} }\sup_{z \in \mathbb{R}^p}\left| \widetilde{f}_{t, h}^{s_m , e_m   }(z) \right| - \gamma_{\mathcal{A}} - \gamma_{\mathcal{B}} \geq \frac{\kappa  \Delta}{16\sqrt{e-s}} - \gamma_{\mathcal{A}} - \gamma_{\mathcal{B}},
	\end{align*}
	where  $\gamma_{\mathcal{A} }$ and $\gamma_{\mathcal{B}} $  are the same  as in (\ref{eq-lem13-3}).
	Thus for any undetected change point $\eta_k \in (s, e)$, it holds that
		\begin{equation}\label{eq:wbsrp size of population}
			a_{m^*} = \sup_{1\le m\le  S} a_m \geq \frac{\kappa  \Delta}{16\sqrt{e-s}} - \gamma_{\mathcal{A}} - \gamma_{\mathcal{B}} \geq c_{\tau, 2}  \kappa \Delta^{1/2},   
		\end{equation}
		where $c_{\tau, 2} > 0$ is achievable with a sufficiently large $C_{\mathrm{SNR}}$ in Assumption~\ref{assump-rates}.  This means we accept the existence of undetected change points.

Suppose that there is no any undetected change point within $(s, e)$, then for any $(s_m, e_m) = (\alpha_m, \beta_m) \cap (s, e)$, one of the following situations must hold.
	\begin{itemize}
		\item [(a)]	There is no change point within $(s_m, e_m)$;
		\item [(b)] there exists only one change point $\eta_k \in (s_m, e_m)$ and $\min\{\eta_k - s_m, e_m - \eta_k\} \le \epsilon_k$; or
		\item [(c)] there exist two change points $\eta_k, \eta_{k+1} \in (s_m, e_m)$ and $\eta_k - s_m \leq \epsilon_k$, $e_m - \eta_{k+1} \leq \epsilon_{k+1}$.
	\end{itemize}

Observe that if (a) holds, then we have
	\[
		\max_{s_m  +  h^{-p} < t < e_m -  h^{-p} }\max_{j = 1, \ldots, T}\left| \widetilde{Y}_{t}^{s_m, e_m}(X(j)) \right| \leq \max_{s_m +  h^{-p} < t < e_m -  h^{-p} }\sup_{z \in \mathbb{R}^p}\left| \widetilde{f}_{t, h}^{s_m   , e_m}(z) \right| + \gamma_{\mathcal{A}} + \gamma_{\mathcal{B}} = \gamma_{\mathcal{A}} + \gamma_{\mathcal{B}}.
	\]
	Cases (b) and (c) can be dealt with using similar arguments.  We will only work on (c) here.  It follows from \Cref{lem-5} (iv) that
	\begin{align*}
		& \max_{s_m + h^{-p}< t < e_m- h^{-p} }\max_{j = 1, \ldots, T}\left| \widetilde{Y}_{t}^{s_m, e_m}(X(j)) \right| \leq \max_{s_m < t < e_m}\sup_{z \in \mathbb{R}^p}\left| \widetilde{f}_{t, h}^{s_m, e_m}(z) \right| + \gamma_{\mathcal{A}} + \gamma_{\mathcal{B}} \\
		\leq & 2\sqrt{e - \eta_k} \kappa_{k+1} + 2\sqrt{\eta_k - s} \kappa_k +  8 \sqrt{\eta_k - s} C_{\mathrm{Lip}} C_{\mathpzc{k}} h + \gamma_{\mathcal{A}} + \gamma_{\mathcal{B}} \leq 2(\gamma_\mathcal{A} + \gamma_\mathcal{B}).
	\end{align*}
	Under \eqref{eq-thm4-tau}, we will always correctly reject the existence of undetected change points. \\
	
\noindent \textbf{Step 2.}  Assume that there exists a change point $\eta_k \in (s, e)$ such that $\min\{\eta_k - s, \eta_k - e\} \ge 3\Delta/4$.  Let $s_m$, $e_m$ and $m^*$ be defined as in Algorithm~\ref{algorithm:WBS}.  To complete the proof it suffices to show that, there exists a change point $\eta_k \in (s_{m*}, e_{m*})$ such that $\min\{\eta_k - s_{m*}, \eta_k - e_{m*}\} \geq \Delta/4$ and $|b_{m*} - \eta_k| \leq \epsilon$.

To this end, we are to ensure that the assumptions of Lemma~\ref{lem-7} are verified.  Note that \eqref{eq-lem13-2} follows from \eqref{eq:wbsrp size of population}, and \eqref{eq-lem13-3} follows from Assumption~\ref{assump-rates}.

Thus, all the conditions in Lemma~\ref{lem-7} are met, and we therefore conclude that there exists a change point $\eta_{k}$, satisfying
	\begin{equation}
		\min \{e_{m^*}-\eta_k,\eta_k-s_{m^*}\} > \Delta /4 \label{eq:coro wbsrp 1d re1}
	\end{equation}
	and
	\begin{equation}\label{eq:coro wbsrp 1d re1-1}
		| b_{m*}-\eta_{k}| \leq C\kappa_k^{-2}\gamma^{2}_{\mathcal{A}} \leq \epsilon,
	\end{equation}
	where the last inequality holds from the choice of $\gamma_{\mathcal{A}}$ and Assumption~\ref{assump-rates}.

The proof is complete by noticing the fact that \eqref{eq:coro wbsrp 1d re1} and  $(s_{m^*}, e_{m^*}) \subset (s, e)$ imply that
	\[
		\min \{e-\eta_k,\eta_k-s\}  >  \Delta /4 > \epsilon.
	\]
	As discussed in the argument before {\bf Step 1}, this implies that $\eta_k $ must be an undetected change point.	
	
\end{proof}

\section{Proofs of Lemmas~\ref{lem-snr-lb} and \ref{lemma-error-opt}}

\begin{proof}[Proof of Lemma~\ref{lemma-error-opt}]
Consider distributions $F$ and $G$ in $\mathbb{R}^p$ with densities $f$ and $g$, respectively, constructed as follows.  The density $f$ is  a test function, thus it has compact  support and it is infinitely differentiable. Note also that  we can take $f$  constant in $B(0, V_p^{-1/p} 2^{-1/p})$,  with $f(0) = 1/2$, and with  
	 \begin{equation}
	 	\label{eqn:construction}  
	 	 \max\{   \|f\|_{\infty}  ,   \underset{x}{\max} \|\nabla f(x)\|   \} \,\leq \,  \frac{1}{2}.
	 \end{equation}
	 Then,  by construction, $f$ is  $1$-Lipschitz.   Let $c_1$ be a constant  such that
	 \begin{equation}
	 \label{eqn:c1}
	 0\,<\,	 c_1 \,<\,  V_p^{-1/p} 2^{-1-1/p},
	 \end{equation}
	 for all $p$, which is possible since  $V_p^{-1/p} 2^{p-1-1/p}\rightarrow \infty$ as  $p \rightarrow \infty$. Then define  $g$ as
	 \[
	    g(x)  \,=\, \begin{cases}
	 \frac{1}{2} + \kappa  -  c_1^{-1} \|x -p_1\|    & \text{if}  \,  \|x -p_1\|  <    \kappa  c_1\\
	 \frac{1}{2}  - \kappa  +   c_1^{-1} \|x -p_2\|    & \text{if}  \,   \|x -p_2\|  <  \kappa c_1\\
	 f(x) & \text{otherwise.}
	 \end{cases}	 
	 \]
	 where $p_1 \,=\,   ( V_p^{-1/p} 2^{-1/p -1},0,\ldots,0) \in \mathbb{R}^p $ and  $p_2 \,=\,   ( -V_p^{-1/p} 2^{-1/p -1},0,\ldots,0) \in \mathbb{R}^p $.
	 Notice that $g$  is well defined since (\ref{eqn:c1}) implies $ \kappa c_1 \,\leq \,   C_1c_1  \,<\, V_p^{-1/p} 2^{-1-1/p}$.
	 
	 Furthermore, by the triangle inequality  and (\ref{eqn:construction}), $g$ is $C$-Lipschitz  for a universal constant $C$. 
	 Moreover, 
	 \[
	     \underset{z \in \mathbb{R}^p}{\sup} \vert f(z) -g(z)\vert  \,=\,  \kappa.
	 \]
	 
	 	Let $P_1$ denote the joint distribution of the independent random variables $\{X(t)\}_{ t = 1}^{ T}$, where
	 \[
	 X(1), \ldots, X(\Delta)  \stackrel{i.i.d.}{\sim}  F  \quad \text{and} \quad X(\Delta + 1), \ldots, X(T)  \stackrel{i.i.d.}{\sim}  G;
	 \]
	 and, similarly, let $P_0$ be the joint distribution of the independent random variables $\{Z(t)\}_{t = 1}^{ T}$ such that
	 \[
	 Z(1), \ldots, Z(\Delta + \xi)  \stackrel{i.i.d.}{\sim}  F, \quad \text{and} \quad Z(\Delta + \xi + 1), \ldots, Z(T)  \stackrel{i.i.d.}{\sim}  G,
	 \]
	 where $\xi$ is a positive integer no larger than $n-1 - \Delta$.
	 
	 Observe that $\eta(P_0) = \Delta$ and $\eta(P_1) = \Delta + \xi$.  By Le Cam's Lemma \citep[e.g.][]{yu1997assouad} and Lemma~2.6 in \cite{Tsybakov2009}, it holds that
	 \begin{equation}\label{eq:second.lower}
	 \inf_{\hat \eta} \sup_{P\in \mathcal{Q}} \mathbb{E}_P\bigl(|\hat \eta - \eta|\bigr)  \geq \xi \bigl\{1- d_{\mathrm{TV}}(P_0, P_1)\bigr\} \geq \frac{\xi}{2} \exp\left( - \mathrm{KL}(P_0,P_1) \right).
	 \end{equation}
	 
	 Since 
	 \begin{align*}
	 \mathrm{KL}(P_0,P_1)   \,=\,  \sum_{i \in \{\Delta + 1, \ldots, \Delta + \xi\}}	 \mathrm{KL}(P_{0i}, P_{1i}) =  \xi \mathrm{KL}(F,G).
	 \end{align*}
	 However, 
	 \begin{align*}
	 	\mathrm{KL}(F,G) & = \frac{1}{2} \int_{  B(p_1, \kappa c_1 ) }   \log\left(\frac{1/2}{  \frac{1}{2} + \kappa  -    c_1^{-1}    \|x -p_1\|  }\right)\,  dx + \frac{1}{2} \int_{  B(p_2,  \kappa c_1 ) }   \log\left(\frac{1/2}{ \frac{1}{2}  - \kappa  +   c_1^{-1}     \|x -p_2\|  }\right)\,dx \\
	 	& = -\frac{1}{2} \int_{  B(0, \kappa c_1 ) }   \log\left(1 + 2\kappa  -     2c_1^{-1}    \|x\| \right) \, dx - \frac{1}{2} \int_{  B(0, \kappa c_1) }   \log\left(  1 - 2\kappa  + 2 c_1^{-1}     \|x \|  \right)\,dx \\
	 	& = - \frac{1}{2} \int_{  B(0, \kappa c_1) }  \log\left( 1  -  (2\kappa  -     2 c_1^{-1}    \|x\|)^2  \right) \,dx \leq 4  \kappa^2  V_p  (\kappa c_1)^p \leq 4\kappa^{p+2} V_p,
	 \end{align*}
   by the inequality $-\log(1-x) \leq  2 x$ for $x \in [0,1/2]$.
	 Therefore,
	 \begin{equation}
	 \label{eq:second.lower2}
	  \inf_{\hat \eta} \sup_{P\in \mathcal{Q}} \mathbb{E}_P\bigl(|\hat \eta - \eta|\bigr) \geq \frac{\xi}{2}   \exp\left( - 4\xi  \kappa^{p+2} V_p  \right)
	 \end{equation}
	
	 	Next, set $\xi = \min \{ \lceil \frac{1}{ 4 V_p  \kappa^{p+2} } \rceil, T - 1 - \Delta\}$. 
	 By the assumption on $\zeta_T$, for all $T$ large enough we must have that $\xi = \lceil \frac{1}{ 4V_p^2\kappa^{2(p+1)}} \rceil$.
	  \end{proof}

\begin{proof}[Proof of \Cref{lem-snr-lb}]

{\bf Step 1.} Let $f_1, f_2 :\mathbb R^p \to \mathbb R^+ $ be two densities such that 
	\[
		f_1(x) = \begin{cases}
			\lambda- \kappa + \|x-x_1 \|_2, &  x\in B(x_1, \kappa), \\
			\lambda, & x\in B(x_2, \kappa), \\
			g(x), & \mbox{otherwise},
		\end{cases} \quad f_2(x) = \begin{cases}
			\lambda-\kappa+ \|x-x_2 \|_2, & x\in B(x_2, \kappa), \\
			\lambda, & x\in B(x_1, \kappa), \\
			g(x), & \text{ otherwise}.
		\end{cases}
	\]
	where $g$ is a function such that $f_1$ and $f_2$ are density functions, $\lambda$ is a constant, and $\kappa$ is a model parameter that can change with $T$.   Note that for small enough $\kappa$ and $\lambda$,
	\[
		\int_{B(x_1, \kappa) } f_1(x)\, dx \le 1.
	\]
	Set $\|x_1 -x_2\| \ge 2\kappa $  to be any two fixed points.  The excess probability mass can be place at $(B(x_1, \kappa) \cup B(x_2,\kappa))^c$.  Since $f_1 = f_2$ in this region, it does not affect $KL(f_1, f_2)$ no matter how the functions are defined in this region.

Observe that, by integrating in polar coordinate and using symmetry
	\begin{align*}
		\mathrm{KL} (f_1, f_2) = & 2p V_p \int_{0}^\kappa \left\{\lambda \log \left (\frac{\lambda}{   \lambda -\kappa + r   } \right) r^{p-1}  +  \left(   \lambda -\kappa + r   \right) \log\left (\frac{ \lambda -\kappa + r   }{\lambda}  \right) r^{p-1}  \right\}\,  dr  \\
		= & 2 pV_p \int_{0 }^\kappa  (\kappa -r ) \log\left (\frac{\lambda}{ \lambda- \kappa +r   } \right) r^{p-1} \, dr   \leq 2 pV_p \int_{0 }^\kappa   (\kappa -r ) \frac{\kappa-r }{\lambda -\kappa +r } r^{p-1} \, dr\\
		\le & 2p V_p \int_{0 }^\kappa   (\kappa -r ) \frac{\kappa-r }{\lambda -\kappa +r } r^{p-1} \, dr \leq  2 pV_p\kappa^2  \lambda^{-1} \int_{0 }^\kappa r^{p-1} \, dr \le C_p\kappa^{p+2}
	\end{align*}

\vskip 3mm
\noindent {\bf Step 2.} Define $\mathcal{P}_T^1 $ to be the joint density of $(X(1), \ldots, X(T))$ such that $X(1), \ldots, X(\Delta) \stackrel{\mbox{i.i.d.}}{\sim} f_1$ and  $X(\Delta+1),\ldots, X(T) \stackrel{\mbox{i.i.d.}}{\sim} f_2$.   Define $ \mathcal{P}_T^2$ to be the joint density of $(X(1), \ldots, X(T))$ such that $X(1), \ldots, X(T - \Delta-1) \stackrel{\mbox{i.i.d.}}{\sim} f_2$ and $X(T - \Delta), \ldots, X(T) \stackrel{\mbox{i.i.d.}}{\sim} f_1$.  We have that
	\begin{align*}
		\inf_{\hat{\eta}} \sup_{\mathcal P_n} \mathbb{E}\{|\widehat \eta - \eta(P)|\} \ge (T -2\Delta) d_{\mathrm{TV}} (  \mathcal P _T^1, \mathcal P _T^2)  \ge (T/4)\exp\{ -\mathrm{KL} (  \mathcal P _T^1, \mathcal P _T^2) \}.
	\end{align*}
	Note that
	\[
		\mathrm{KL} (  \mathcal P _T^1, \mathcal P _T^2)  \le 2\Delta  \text{KL } (f_1, f_2)  = C_p ' \kappa^{p+2} \Delta.
	\]
	Since $\Delta \kappa^{p+2 } \le c < \log(2)$, we have
	\[
		\exp( -\mathrm{KL} (  \mathcal P _T^1, \mathcal P _T^2)   )  \ge \exp(-c) \ge 1/2
	\]	
	see e.g. \cite{Tsybakov2009}.  In addition, noticing that $\Delta < T/2$, we reach the final claim.
\end{proof}

\bibliographystyle{ims}
\bibliography{citations}

\end{document}